\nonstopmode
\pdfoutput=1
\documentclass[final, twocolumn, envcountsect, envcountsame]{svjour3}

\usepackage{amsmath, amssymb}
\usepackage{graphicx}        
\usepackage{epsfig} 
\usepackage[all]{xy}         
\usepackage{cite}            
\usepackage{stmaryrd}	      

\smartqed              
\RequirePackage[displaymath, mathlines]{lineno} 
\def\al{\alpha} 
\def\be{\beta} 
\def\ga{\gamma} 
\def\de{\delta} 
\def\ep{\varepsilon}

\def\th{\theta} 
 
\def\ka{\kappa} 
\def\la{\lambda} 
 
\def\si{\sigma} 
 
\def\ph{\varphi} 
 
\def\ps{\psi} 
 
\def\Ga{\Gamma} 
\def\De{\Delta} 
 
\def\La{\Lambda} 
 
\def\Ph{\Phi} 
\def\Ps{\Psi} 
\def\Om{\Omega}
\def\Hor{\on{Hor}}
\def\hor{\on{hor}}
\def\ver{\on{\ver}}
\def\Emb{\on{Emb}}
\def\Imm{\on{Imm}}
\def\Id{\on{Id}}
\def\Diff{\on{Diff}}
\def\Vol{\operatorname{Vol}}
\def\x{\times}
\def\p{\partial} 
\def\X{{\mathfrak X}}
  
\def\L{\mathcal{L}}

\def\R{{\mathbb R}}
\def\ad{\operatorname{ad}}

\let\on=\operatorname
\let\wt=\widetilde
\let\wh=\widehat
\let\ol=\overline

\let\mc=\mathcal
\let\mf=\mathfrak
\newcommand{\ud}{\,\mathrm{d}}
\def\Tr{\on{Tr}}
\def\vol{\on{vol}}
\def\dist{\on{dist}}
\def\Met{\on{Met}}

\spnewtheorem*{openquestion}{Open question.}{\bf}{\rm}
\spnewtheorem*{rmk}{Remark.}{\bf}{\rm}

\begin{document}


\title{Overview of the Geometries of Shape Spaces and Diffeomorphism Groups}
\titlerunning{Overview of Shape Spaces}

\author{Martin Bauer 
\thanks{M. Bauer was supported by FWF Project P24625}
\and Martins Bruveris \and Peter W. Michor}
\institute{
Martin Bauer \and Peter W. Michor
\at Fakult\"at f\"ur Mathematik, Universit\"at Wien, \\
Oskar-Morgenstern-Platz 1,
A-1090 Wien, Austria \\
\email{bauer.martin@univie.ac.at} \\
\email{peter.michor@univie.ac.at}
\and
Martins Bruveris
\at Institut de math\'ematiques, EPFL, CH-1015, \\ Lausanne, Switzerland \\
\email{martins.bruveris@epfl.ch}
}

\date{\today}
\maketitle

\begin{abstract}
This article provides an overview of various notions of shape spaces, including the space of 
parametrized and unparametrized curves, the space of immersions, the diffeomorphism group 
and the space of Riemannian metrics. We discuss the Riemannian metrics that can be defined thereon, 
and what is known about the properties of these metrics. We put particular emphasis on the induced 
geodesic distance, the geodesic equation and its well-posedness, geodesic and metric completeness 
and properties of the curvature.  
\end{abstract}
\keywords{Shape Space \and Diffeomorphism Group \and Manifolds of mappings \and Landmark space 
\and Surface matching \and Riemannian geometry} 
\subclass{\\ 58B20 \and 58D15\and 35Q31}


\tableofcontents
\section{Introduction}
The variability of a certain class of shapes is of interest in various fields of applied mathematics and 
it is of particular importance in the field of computational anatomy. In mathematics and computer vision, shapes have been represented in many different ways: 
point clouds, surfaces or images are only some  examples. These shape spaces are inherently non-linear. As an example, consider the shape space of all surfaces of a certain dimension and genus. 
The nonlinearity makes it difficult to do statistics. One way to overcome this difficulty is to introduce 
a Riemannian structure on the space of shapes. This enables us
to locally linearize the space and develop statistics based on geodesic methods.
Another advantage of the Riemannian setting for shape analysis is its intuitive notion of similarity. Namely, two shapes that differ only by a small deformation are regarded as similar to each other.

In this article we will concentrate on shape spaces of surfaces and we will give an overview of the different Riemannian structures, that have been considered on these spaces. 

\subsection{Spaces of interest}\label{intro_spaces}

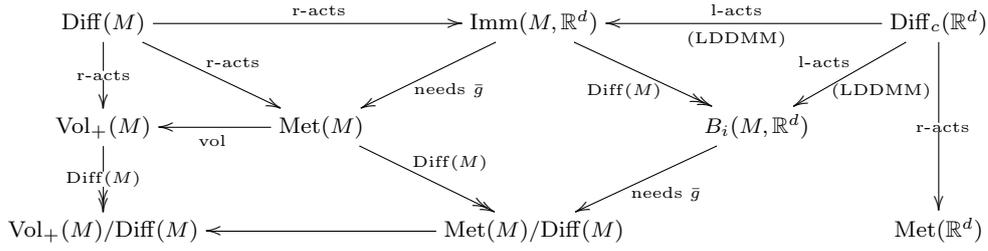
\begin{figure*}
$$
\xymatrix{
{\Diff(M)}  \ar[rr]^{\text{ r-acts }}  \ar[rd]^{\text{ r-acts }} 
\ar[d]|{\text{r-acts}}
& & {\on{Imm}(M,\mathbb R^d)} \ar[dl]^{\text{needs  }\bar g} \ar@{->>}[dr]_{\Diff(M)}
& & {\Diff_{c}(\mathbb R^d)} \ar[dd]|{\text{ r-acts} } \ar[ll]_{\text{ l-acts }}^{\text{(LDDMM)}} 
\ar[ld]_{\text{ l-acts }}^{\text{(LDDMM)}} 
\\
\Vol_{+}(M) \ar@{->>}[d]|{\Diff(M)} & \Met(M) \ar[l]^{\vol}
\ar@{->>}[rd]^{\Diff(M)} & & {B_i(M,\R^d)} \ar[ld]^{\text{needs  }\bar g} &
\\
\Vol_{+}(M)/{\Diff(M)}  && 
\ar[ll] \Met(M)/{\Diff(M)} && \Met(\mathbb R^d) & 
}
$$
\caption{An overview of the relations between the spaces discussed in this article. Here $\Vol_+(M)$ is the space of all positive definite volume densities, $\Diff(M)$ the diffeomorphism group, $\Met(M)$ the space of all Riemannian metrics, $\Imm(M,\R^d)$
the space of immersed surfaces and $B_i(M,\R^d)$ the shape space of unparametrized surfaces; $\ol g$ denotes the Euclidean metric on $\R^d$.}
\end{figure*}
We fix a compact manifold $M$ without boundary of dimension $d-1$. In this paper a shape is a submanifold of $\R^d$ that is diffeomorphic to $M$ and we denote by $B_i(M,\R^d)$ and $B_e(M,\R^d)$ the spaces of all immersed and embedded submanifolds.

One way to represent $B_i(M,\R^d)$ is to start with the space $\on{Imm}(M,\R^d)$ of immersions upon which the diffeomorphism group $\on{Diff}(M)$ acts from the right via
\[
\on{Imm}(M,\R^d)\x \on{Diff}(M) \ni (q, \ph) \mapsto q \on{\circ} \ph \in \on{Imm}(M,\R^d).
\]
The space $\on{Imm}(M,\R^d)$ is the space of parametrized type $M$ submanifolds of $\R^d$ and the action of $\on{Diff}(M)$ represents {\it reparametrizations}. The immersions $q$ and $q \on{\circ} \ph$ have the same image in $\R^d$ and thus one can identify $B_i(M,\R^d)$ with the quotient
\[
B_i(M,\R^d) \cong \on{Imm}(M,\R^d) / \on{Diff}(M)\,.
\]
The space $B_i(M,\R^d)$ is not a manifold, but an orbifold with isolated singular points; see Sect.\ \ref{orbifolds}. To remove these we will work with the slightly smaller space $\on{Imm}_f(M,\R^d)$ of free immersions (see Sect. \ref{imm_emb}) and denote the quotient by
\[
B_{i,f}(M,\R^d) \cong \on{Imm}_f(M,\R^d) / \on{Diff}(M)\,.
\]
Similarly one obtains the manifold $B_e(M,\R^d)$ as the quotient of the space $\on{Emb}(M,\R^d)$ of embedded submanifolds with the group $\on{Diff}(M)$.

The spaces $\on{Imm}(M,\R^d)$ and $\on{Emb}(M,\R^d)$ are sometimes called {\it pre-shape spaces}, 
and $\on{Diff}(M)$ is the {\it reparametrization group}. 
Their quotients $B_i(M,\R^d)$ and $B_e(M,\R^d)$ are called {\it shape spaces}.

We would like to note that our usage of the terms shape and pre-shape space differs from that employed in \cite{Kendall1999,Small1996,Dryden1998}. In the terminology of \cite{Kendall1999} a pre-shape space is the space of labelled landmarks modulo translations and scalings and the shape space additionally quotients out rotations as well. For the purposes of this paper, the pre-shape space contains parametrized curves or surfaces and we pass to the shape space by quotienting out the parametrizations.

A Riemannian metric on $\on{Imm}(M,\R^d)$ that is invariant under the action of $\on{Diff}(M)$ induces a metric on $B_{i, f}(M,\R^d)$, such that the projection
\[
\pi : \on{Imm}_f(M,\R^d) \to B_{i, f}(M,\R^d),\quad q \mapsto q(M)\,,
\]
is a Riemannian submersion. This will be our method of choice to study almost local and 
Sobolev-type metrics on $\on{Imm}(M,\R^d)$ and $B_{i, f}(M,\R^d)$ in Sect.\ \ref{Almost_local} and 
\ref{sobolev_inner}. These are classes of metrics, that are defined via quantities measured 
directly on the submanifold. We might call them {\it inner metrics} to distinguish them from 
{\it outer metrics}, which we will describe next. This is however more a conceptual distinction 
rather than a rigorously mathematical one.     

Most Riemannian metrics, that we consider in this article, will be {\it weak}, i.e., considered as a mapping from the tangent bundle to the cotangent bundle the metric is injective, but not surjective. Weak Riemannian metrics are a purely infinite-dimensional phenomenon and they are the source of most of the analytical complications, that we will encounter.

A way to define Riemannian metrics on the space of parametrized submanifolds is via the left action of $\on{Diff}_c(\R^d)$ on $\on{Emb}(M,\R^d)$,
\[
\on{Diff}_c(\R^d) \x \on{Emb}(M,\R^d)\!\ni\!(\ph, q) \!\mapsto\! \ph\on{\circ} q \in \on{Emb}(M,\R^d).
\]
Here $\on{Diff}_c(\R^d)$ denotes the Lie group of all compactly supported diffeomorphisms of $\R^d$ with  Lie algebra the space $\mf X_c(\R^d)$ of all compactly supported vector fields, see Sect.~\ref{diff_groups}.
Given a right-invariant metric on $\on{Diff}_c(\R^d)$, the left action induces a metric on $\on{Emb}(M,\R^d)$, such that for each embedding $q_0 \in \on{Emb}(M,\R^d)$ the map
\[
\pi_{q_0} : \on{Diff}_c(\R^d) \to \on{Emb}(M,\R^d),\quad \ph \mapsto \ph \on{\circ} q_0\,,
\]
is a Riemannian submersion onto the image. This construction formalizes the idea of measuring the 
cost of deforming a shape as the minimal cost of deforming the ambient space, i.e.,
\begin{equation}
\label{eq:diff_metric_descend}
G^{\on{Emb}}_q(h,h) = \inf_{X \on{\circ} q = h} G^{\on{Diff}}_{\on{Id}}(X, X)\,.
\end{equation}
Here $h \in T_q \on{Emb}(M,\R^d)$ is an infinitesimal deformation of $q$ and the length squared $G^{\on{Emb}}_q(h,h)$, which measures its cost, is given as the infimum of $G^{\on{Diff}}_{\on{Id}}(X, X)$, that is the cost of deforming the ambient space. The infimum is taken over all $X \in \mf X_c(\R^d)$ infinitesimal deformations of $\R^d$, that equal $h$ when restricted to $q$. This motivates the name {\it outer metrics}, since they are defined in terms of deformations of the ambient space.

The natural space to define these metrics is the space of embeddings instead of immersions, because not all orbits of the $\on{Diff}_c(\R^d)$ action on $\on{Imm}(M,\R^d)$ are open. Defining a Riemannian metric on $B_e(M,\R^d)$ is now a two step process
\[
\xymatrix@1{
\on{Diff}_c(\R^d) \ar[r]^-{\pi_{q_0}} & \on{Emb}(M,\R^d) \ar[r]^-\pi & B_e(M,\R^d)
}\,.
\]
One starts with a right-invariant Riemannian metric on $\on{Diff}_c(\R^d)$, which descends via \eqref{eq:diff_metric_descend} to a metric on $\on{Emb}(M,\R^d)$. This metric is invariant under the reparametrization group and thus descends to a metric on $B_e(M,\R^d)$. These metrics are studied in Sect.\ \ref{sec_outer}.

Riemannian metrics on the diffeomorphism groups $\on{Diff}_c(\R^d)$ and $\on{Diff}(M)$ are of interest, not only because these groups act as the deformation group of the ambient space and the reparametrization group respectively. They are related to the configuration spaces for hydrodynamics and various PDEs arising in physics can be interpreted as geodesic equations on the diffeomorphism group. While a geodesic on $\on{Diff}_c(\R^d)$ is a curve $\ph(t)$ of diffeomorphisms, its right-logarithmic derivative $u(t) = \p_t \ph(t) \on{\circ} \ph(t)^{-1}$ is a curve of vector fields. If the metric on $\on{Diff}_c(\R^d)$ is given as $G_{\on{Id}}(u, v) = \int_{\R^d} \langle Lu, v \rangle \ud x$ with a differential operator $L$, then the geodesic equation can be written in terms of $u(t)$ as
\[
\p_t m + (u \cdot \nabla)m + m \on{div} u + Du^T.m = 0,\quad m = Lu\,.
\]
PDEs that are special cases of this equation include the Camassa-Holm equation, the Hunter-Saxton equation and others. See Sect.\ \ref{diff_groups} for details.

So far we encoded shape through the way it lies in the ambient space; i.e., either as a map $q: M\to \R^d$ or as its image $q(M)$. One can also look at how the map $q$ deforms the model space $M$. Denote by $\ol g$ the Euclidean metric on $\R^d$ and consider the pull-back map
\begin{equation}
\label{intro_met_pb}
\on{Imm}(M,\R^d) \to \Met(M),\quad q \mapsto q^\ast \ol g\,,
\end{equation}
where $\Met(M)$ is the space of all Riemannian metrics on $M$ 
and $q^\ast \ol g$ denotes the pull-back of $\ol g$ to a metric on $M$. Depending on the dimension of $M$ one
can expect to capture more or less information about shape with this map. Elements of $\Met(M)$ with $\on{dim}(M) = d-1$ 
are symmetric, positive definite tensor fields of type $0 \choose 2$ and thus have $\tfrac{d(d-1)}2$ components. 
Immersions on the other hand are maps from $M$ into $\R^d$ and thus have $d$ components. For $d=3$, 
the case of surfaces in $\R^3$, the number of components coincide, while for $d>3$ we have $\tfrac{d(d-1)}2 > d$. 
Thus we would expect the pull-back map to capture most aspects of shape. The pull-back is equivariant with respect 
to $\on{Diff}(M)$ and thus we have the commutative diagram
\[
\xymatrix{
\on{Imm}_f(M,\R^d) \ar[r]^-{q\mapsto q^\ast \ol g}\ar[d] & \Met(M) \ar[d] \\
B_{i,f}(M,\R^d) \ar[r] & \Met(M) / \on{Diff}(M)
}
\]
The space in the lower right corner is not far away from $\Met(M) / \on{Diff}_0(M)$, where $\on{Diff}_0(M)$ denotes the connected component of the identity. This space, known as {\it super space}, is used in general relativity. Little is known about the properties of the pull-back map \eqref{intro_met_pb}, but as a first step it is of interest to consider Riemannian metrics on the space $\Met(M)$. This is  done in Sect.\ \ref{sec_mets}, with special emphasis on the $L^2$- or Ebin-metric.

\subsection{Questions discussed}

After having explained the spaces, that will play the main roles in the paper and the relationships 
between them, what are the questions that we will be asking? The questions are motivated by 
applications to comparing shapes.  

After equipping the space with a Riemannian metric, the simplest way to compare shapes is by 
looking at the matrix of pairwise distances, measured with the induced geodesic distance function. 
Thus an important question is, whether the geodesic distance function is point-separating, that is 
whether for two distinct shapes $C_0 \neq C_1$ we have $d(C_0, C_1) > 0$. In finite dimensions the 
answer to this question is always ``yes''. Even more, a standard result of Riemannian geometry 
states that the topology induced by the geodesic distance coincides with the manifold topology. In 
infinite dimensions, when the manifold is equipped with a weak Riemannian metric, this is in 
general not true any more. The topology induced by the geodesic distance will also be weaker than 
the manifold topology. We will therefore survey what is known about the geodesic distance and the 
topology it induces.         

The path realizing the distance between two shapes is, if it exists, a geodesic. So it is natural 
to look at the geodesic equation on the manifold. In finite dimensions the geodesic equation is an 
ODE, the initial value problem for geodesics can be solved, at least for short times, and the 
solution depends smoothly on the initial data. The manifolds of interest in this paper are 
naturally modeled mostly as Fr\'echet manifolds and in coordinates the geodesic equation is usually a 
partial differential equation or even involves pseudo differential operators. Only the regulat Lie 
group $\Diff_c(N)$ of diffeomorphisms with compact support on a noncompact manifold are modeled on 
$(LF)$-spaces. 
Thus even the short-time solvability of the initial-value problem is a non-trivial 
question. For some of the metrics, in particular for the class of almost local metrics, it is still 
open. For the diffeomorphism group the geodesic equations for various metrics are of interest in 
their own right. To reflect this we will discuss in Sect.\ \ref{diff_groups} first the geodesic 
equations before proceeding with the properties of the geodesic distance.         

It is desirable for applications that the Riemannian metric possesses some completeness properties. 
It can be either in form of geodesic completeness, i.e., that geodesics are extendable for all 
time, or metric completeness with respect to the induced geodesic distance. Since we are 
considering only weak Riemannian metrics on spaces of smooth shapes, we cannot expect the space to 
be metrically complete, but in some examples it is possible to at least describe the metric 
completion of shape space.     

In order to perform statistics on shape space one can choose a reference shape, for example by computing the mean of a given set of shapes, and linearize the space around this shape via the Riemannian exponential map and normal 
coordinates. The curvature tensor contains information about the accuracy of this approximation. In 
general computing the curvature leads to long formulas that are hard to interpret, but in some 
cases curvature admits a simple expression. We collect the examples, where more is known about the 
curvature, usually the sectional curvature, than just the formula for it.     

To summarize, we will deal with the following four properties of Riemannian metrics on shape spaces:
\begin{itemize}
\item
Geodesic distance
\item
Geodesic equation and existence of geodesics
\item
Geodesic and metric completeness
\item
Properties of the curvature
\end{itemize}

Riemannian geometry on shape spaces is currently an active area of research. Therefore this paper is less an encyclopedic treatment of the subject but rather resembles an interim report highlighting what is known and more importantly, what is not.

\subsection{Topics not discussed}

There are many topics that lie outside the scope of this paper, among which are the following.

{\it Changes in topology.} In certain applications it may be of interest to consider deformations of a shape that allow for the development of holes or allow the shape to split into several components. In this paper we fix the model manifold $M$ and only consider submanifolds of $\R^d$ diffeomorphic to $M$. Thus by definition all deformations are topology-preserving. See \cite{Delfour2011, Wirth2011, Bronstein2010} for topologically robust approaches to shape matching.

{\it Non-geodesic distances.} Many interesting distances can be defined on shape spaces, that are not induced by an underlying Riemannian metric; see for example \cite{Manay2006, Mumford1991, Memoli2005}. As we are looking at shape spaces through the lens of Riemannian geometry, these metrics will necessarily be left out of focus.

{\it Subgroups of the diffeomorphism groups.} The Riemannian geometry of the diffeomorphism group and its subgroups, especially the group  of volume-preserving diffeomorphisms, has been studied extensively; see for example \cite{Smolentsev2006}. It plays an important role in hydrodynamics, being the configuration space for incompressible fluid flow \cite{Ebin1970}. While the full diffeomorphism group itself is indispensable for shape analysis, its subgroups have not been used much in this context.

{\it Utmost generality.} We did not strive to state the results in the most general setting. It is possible to consider shapes of higher codimension inside $\R^d$ or curved ambient spaces; see \cite{Bauer2011b}. This would include examples like space curves or curves that lie on a sphere. It would also make the presentation more difficult to read.

{\it Numerical methods.} Since shape space is infinite-dimensional, computing the exponential map, the geodesic between two shapes or the geodesic distance are numerically non-trivial tasks. While we present some examples, we do not attempt to provide a comprehensive survey of the numerical methods that have been employed in the context of shape spaces. Finding stable, robust and fast numerical methods and proving their convergence is an area of active research for most of the metrics and spaces discussed in this paper. See \cite{Rumpf2012_preprint, Samir2012, Cotter2008, Cotter2012, Bauer2011a, Gunther2011} for various approaches to discretizing shape space.

\section{Preliminaries}

\subsection{Notation}\label{notation}
In this section we will introduce the basic notation that we will use throughout this article. On 
$\R^d$ we consider the Euclidean metric, which we will denote by $\bar g$ or 
$\langle \cdot,\cdot\rangle$. Unless stated otherwise we will assume that the parameter space  $M$ is a compact, oriented manifold without boundary of dimension $\on{dim}(M)=d-1$. Riemannian metrics on $M$ are usually denoted by $g$. Tensor 
fields like $g$ and its variations $h$ are identified with their associated mappings 
$TM\to T^*M$. For a metric $g$ this yields the musical isomorphisms
\[ \flat: TM\to T^*M \quad\text{ and }\quad \sharp: T^*M\to TM.\]
Immersions from $M$ to $\R^d$ -- i.e., smooth mappings with everywhere injective derivatives 
-- are denoted by $q$ and the corresponding  unit normal field of
an (orientable) immersion $q$ is denoted by $n_q$.  
For every immersion $q: M\to \R^d$ we consider the induced pull-back metric $g=q^*\bar g$ on $M$ given by
\[q^*\bar g(X,Y)=\bar g(Tq.X,Tq.Y),\]
for  vector fields $X,Y\in\X(M)$. Here $Tq$ denotes the tangent mapping of the map $q:M\to\R^d$.
We will denote the induced volume form of the metric $g=q^*\bar g$ as $\vol(g)$. In positively oriented coordinates $(u,U)$ it is given by
\[\on{vol}(g)=\sqrt{\on{det}(\bar g\left(\p_iq,\p_j q
\right))}\ du^1\wedge\dots\wedge du^{d-1}\,.\]
Using the volume form we can calculate the total volume $\on{Vol}_q=\int_M \on{vol}(q^*\bar g)$ of the immersion $q$.

The Levi-Civita covariant derivative determined by a metric $g$ will be denoted by $\nabla^g$ and we will consider the induced Bochner--Laplacian $\Delta^g$, which is defined for all 
vector fields $X\in\X(M)$ via 
\[\De^gX =-\on{Tr}(g^{-1} \nabla^2 X)\,.\]
Note that in $\R^d$ the usual Laplacian $\Delta$ is the negative of the Bochner--Laplacian of the Euclidean metric, i.e., $\Delta^{\bar g}=-\Delta$.

Furthermore, we will need the  second fundamental form $s_q(X,Y)=\bar g\left( \nabla^{q^*\bar g}_X Tq.Y,n_q\right).$
Using it we can define the 
Gau{\ss} curvature $K_q=\on{det}(g^{-1} s_q)$ and the mean curvature $H_q=\on{Tr}(g^{-1} s_q)$. 

In the special case of plane curves ($M=S^1$ and $d=2$) we use the letter $c$ for the immersed curve. 
The curve parameter $\th \in S^1$ will be the positively oriented parameter on $S^1$, 
and differentiation $\p_\th$ will be denoted by the subscript $\th$, i.e., $c_\th=\p_\th c$.
We will use a similar notation for the time derivative of a time dependent family of curves, i.e., $\partial_tc= c_t$.

We denote the corresponding unit length tangent vector by 
\[
v=v_c=\frac{c_\theta}{|c_\theta|}=-Jn_c\quad\text{where }
J=\sqrt{-1}\text{  on }\mathbb C=\mathbb R^2\,,
\]
and $n_c$ is the unit length tangent vector. The covariant derivative of the pull-back metric reduces to arclength derivative, and the induced volume form to arclength integration: \[D_s=\frac{\p_\th}{|c_{\th}|}, \qquad \ud s= |c_{\th}|\ud\th.\]
Using this notation the length of a curve can be written as
\[\ell_c=\int_{S^1}\ud s\,.\]
In this case Gau{\ss} and mean curvature are the same and are denoted by
$\ka = \langle D_s v, n\rangle$. 

\subsection{Riemannian submersions}\label{Submersion}

In this article we will repeatedly induce a Riemannian metric on a quotient space using a given metric on the top space.
The concept of a Riemannian submersion will allow us to achieve this goal in an elegant manner.
We will now explain in general terms what a Riemannian submersion is
and how geodesics in the quotient space correspond to horizontal geodesics in the top space. 

Let $(E,G_E)$ be a possibly infinite dimensional weak Riemannian manifold; weak means that 
$G_E:TE\to T^*E$ is injective, but need not be surjective. A consequence is that the Levi-Civita 
connection (equivalently, the geodesic equation) need not exist; however, if the Levi-Civita connection does exist, it is unique.
Let $\mathcal G$ be a smooth possibly infinite dimensional regular Lie group;
see \cite{KM97r} or \cite[Section 38]{Michor1997} for the notion used here, or \cite{Neeb2006} for 
a more general notion of Lie group. Let $\mathcal G\x E\to E$ be a
smooth group action on $E$ and assume that 
$B:=E/\mathcal G$ is a manifold.  
Denote by $\pi:E\rightarrow B$ the projection, which is then a submersion of smooth manifolds by 
which we means that it admits local smooth sections everywhere;
in particular, $T\pi:TE \rightarrow TB$ is surjective. Then
$$\on{Ver}=\on{Ver}(\pi):=\on{ker}(T\pi) \subset TE$$
is called the {\it vertical subbundle}. 
Assume that $G_E$ is in addition invariant under the action of $\mathcal G$.
Then the expression
\[
\|Y\|^2_{G_B} := \inf\{\|X\|^2_{G_E}: X \in T_x E,\, T\pi.X = Y\}
\]
defines a semi-norm on $B$. If it is a norm, it can be shown (by polarization pushed through the 
completion)
that this norm comes from a weak Riemannian metric $G_B$ on $B$; 
then the projection $\pi: E \to B$ is a Riemannian submersion.

Sometimes the the $G_E$-orthogonal space
$\on{Ver}(\pi)^\bot \subset TE$
is a fiber-linear complement in $TE$. In general, the orthogonal 
space is a 
complement (for the $G_E$-closure of $\on{Ver}(\pi)$) only if taken in
the fiberwise $G_E$-completion $\overline {TE}$ 
of $TE$. This leads to the notion of a {\it robust Riemannian manifold}: a Riemannian manifold $(E,G_E)$ is called robust, if $\overline{TE}$ is a smooth vector-bundle over $E$ and the Levi-Civita connection of $G_E$ exists and is smooth. See \cite{Micheli2013} for details. We will encounter examples, where the use of $\overline{TE}$ is necessary in Sect. \ref{sec_outer}.

The {\it horizontal subbundle} $\Hor=\Hor(\pi,G)$ is the $G_E$-orthogonal complement of $\on{Ver}$ in $TE$ or in 
$\overline{TE}$, respectively. Any vector $X \in TE$ can be decomposed uniquely in vertical and horizontal components as
\[
X=X^{\on{ver}}+X^{\hor}\,.
\]
Note that if we took the complement in $\overline{TE}$, i.e., $\on{Hor} \subset \overline{TE}$, then in general $X^{\on{ver}} \in \overline{\on{Ver}}$. The mapping
\[
T_x \pi|_{\Hor_x}:\Hor_x\rightarrow T_{\pi(x)}B\quad\text{  or }\quad\overline{T_{\pi(x)}B}
\]
is a linear isometry of (pre-)Hilbert spaces for all $x\in E$. Here $\overline{T_{\pi(x)}B}$ is the fiber-wise $G_B$-completion of $T_{\pi(x)}B$. We are not claiming that $\overline{TB}$ forms a smooth vector-bundle over $B$ although this will be true in the examples considered in Sect. \ref{sec_outer}.
  
\begin{theorem}\label{thm:submersion}
Consider a Riemannian submersion $\pi:E\rightarrow B$ between robust weak Riemannian manifolds, 
and let $\ga:[0,1]\rightarrow E$ be a geodesic in $E$.
\begin{enumerate}
\item If $\ga'(t)$ is horizontal at one $t$, then it is horizontal at all $t$. 
\item If $\ga'(t)$ is horizontal then $\pi \circ \ga$ is a geodesic in $B$.
\item If every curve in $B$ can be lifted to a horizontal curve in $E$, 
then, up to the choice of an initial point, there is a one-to-one correspondence between curves in $B$ and horizontal curves in $E$. 
This implies that instead of solving the geodesic equation in $B$ one can equivalently solve
the equation for horizontal geodesics in $E$.
\end{enumerate}
\end{theorem}
See \cite[Sect.~26]{Michor2008b} for a proof, and \cite{Micheli2013} for the case of 
robust Riemannian manifolds.

\section{The spaces of interest}

\subsection{Immersions and embeddings}\label{imm_emb}
Parametrized surfaces will be modeled as immersions or embeddings 
of the configuration manifold $M$ into $\R^d$.
We call immersions and embeddings parametrized since a change in their parametrization 
(i.e., applying a diffeomorphism on the domain of the function) 
results in a different object. 
We will deal with the following sets of functions:

\begin{equation}\label{sh:im:eq1}
\begin{aligned}
 \Emb(M,\R^d) \subset &\Imm_f(M,\R^d) \\&\quad\subset \Imm(M,\R^d) \subset C^\infty(M,\R^d)\,.
\end{aligned}
\end{equation}
Here $C^\infty(M,\R^d)$ is the set of smooth functions from $M$ to $\R^d$, 
$\Imm(M,\R^d)$ is the set of all \emph{immersions} of $M$ into $\R^d$, i.e.,
all functions $q \in C^\infty(M,\R^d)$ such that 
$T_x q$ is injective for all $x \in M$.
The set $\Imm_f(M,\R^d)$  consists of all \emph{free immersions} $q$;
i.e., the diffeomorphism 
group of $M$ acts freely on $q$, i.e., $q \on{\circ} \ph = q$ implies
$\ph=\Id_M$ for all $\ph \in \Diff(M)$. 

By \cite[Lem.~3.1]{Michor1991}, the isotropy group $\Diff(M)_q:=\{\ph\in\Diff(M): q\on{\circ} 
\ph=q\}$ of any 
immersion $q$ is always a finite group which acts properly discontinuously on $M$ so that 
$M\to M/\Diff(M)_q$ is a covering map. 
$\Emb(M,N)$ is the set of all \emph{embeddings} of $M$ into $\R^d$, i.e., 
all immersions $q$ that are homeomorphisms onto their image. 
\begin{theorem}
The spaces $\Imm(M,\R^d), \Imm_f(M,\R^d)$ and $\Emb(M,\R^d)$ are Fr\'echet manifolds.  
\end{theorem}
\begin{proof}
 Since $M$ is compact by assumption it follows that $C^\infty(M,\R^d)$ is a 
Fr\'echet manifold by \cite[Sect.~42.3]{Michor1997}; see also \cite{Hamilton1982}, \cite{Michor1980}. 
All inclusions in \eqref{sh:im:eq1} are inclusions of open subsets: 
first $\Imm(M,\R^d)$ is open in $C^\infty(M,\R^d)$ since
the condition that the differential is injective at every point 
is an open condition on the one-jet of $q$ \cite[Sect.~5.1]{Michor1980}. 
$\Imm_f(M,\R^d)$ is open in $\Imm(M,\R^d)$ by \cite[Thm.~1.5]{Michor1991}.
$\Emb(M,\R^d)$ is open in $\Imm_f(M,\R^d)$ by \cite[Thm.~44.1]{Michor1997}. 
Thus all the spaces are Fr\'echet manifolds as well.
\end{proof}

\subsection{Shape space}

Unparametrized surfaces are equivalence classes of parametrized surfaces under the action of the reparametrization group.

\begin{theorem}[Thm. 1.5, \cite{Michor1991}]
\label{bi_bundle}
The quotient space 
\[
B_{i,f}(M,\R^d) := \on{Imm}_f(M,\R^d) / \on{Diff}(M)
\]
 is a smooth Hausdorff manifold and the projection
\[
\pi :\on{Imm}_f(M,\R^d) \to B_{i,f}(M,\R^d)
\]
is a smooth principal fibration with $\on{Diff}(M)$ as structure group.

For $q \in \on{Imm}_f(M,\R^d)$ we can define a chart around $\pi(q) \in B_{i,f}(M,\R^d)$ by
\[
\pi \on{\circ} \ps_q : C^\infty(M, (-\ep, \ep)) \to B_{i,f}(M,\R^d)
\]
with $\ep$ sufficiently small, where 
\[
\ps_q : C^\infty(M, (-\ep, \ep)) \to \on{Imm}_{f}(M,\R^d)
\]
 is defined by $\ps_q(a) = q + an_q$ and $n_q$ is the unit-length normal vector to $q$.
\end{theorem}

\begin{corollary}
\label{be_bundle}
The statement of Thm. \ref{bi_bundle} does not change, if we replace $B_{i,f}(M,\R^d)$ by $B_{e}(M,\R^d)$ and $\on{Imm}_f(M,\R^d)$ by $\on{Emb}(M,\R^d)$.
\end{corollary}

The result for embeddings is proven in \cite{Michor1980b,Michor1980,BinzFischer1981,Hamilton1982}. 
As $\on{Emb}(M,\R^d)$ is an open subset of $\on{Imm}_f(M,\R^d)$ and is $\on{Diff}(M)$-invariant, the quotient
\[
 B_e(M,\R^d):=\on{Emb}(M,\R^d)/\Diff(M)
\]
is an open subset of $B_{i,f}(M,\R^d)$ and as such itself a smooth principal bundle with structure group $\on{Diff}(M)$.

\subsection{Some words on orbifolds}
\label{orbifolds}
The projection 
$$
\Imm(M,\R^d) \to B_i(M,\R^d):=\Imm(M,\R^d)/\Diff(M)
$$
is the prototype of a Riemannian submersion onto an infinite dimensional Riemannian orbifold.
In the article \cite[Prop.~2.1]{Stanhope2011} it is stated that the finite dimensional Riemannian orbifolds are 
exactly of the form $M/G$ for a Riemannian manifold $M$ and a compact group $G$ of isometries with 
finite isotropy groups. Curvature on Riemannian orbifolds is well defined, and it 
suffices to treat it on the dense regular subset. In our case $B_{i,f}(M,\R^d)$ is the 
regular stratum of the orbifold $B_i(M,\R^d)$.
For the behavior of geodesics on Riemannian orbit spaces $M/G$ see 
for example \cite{Michor2003}; the easiest way to carry these results over to infinite 
dimensions is by using  
Gauss' lemma, which only holds if the Riemannian exponential mapping is a diffeomorphism on an 
$G_{\Imm}$-open neighborhood of 0 in each tangent space. This is rarely true. 

Given a $\Diff(M)$-invariant Riemannian metric on $\Imm(M,\R^d)$,
one can define a metric distance $\on{dist}^{B_i}$ on $B_i(M,\R^d)$ by taking as distance between two shapes the 
infimum of the lengths of all (equivalently, horizontal) smooth curves connecting the corresponding 
$\Diff(M)$-orbits. 
There are the following questions:
\begin{itemize} 
  \item Does $\on{dist}^{B_i}$ separate points? In many cases this has been decided.
	\item Is $(B_i(M,\R^d),\on{dist}^{B_i})$ a geodesic metric space? In other words,  does there exists a 
        rectifiable curve connecting two shapes in the same connected component whose length is 
        exactly the distance? This is widely open, but it is settled as soon as local minimality of 
        geodesics in $\Imm(M,\R^d)$ is established.
\end{itemize}

In this article we are discussing Riemannian metrics on $B_{i,f}(M,\R^d)$, that are induced by Riemannian metrics on $\on{Imm}_f(M,\R^d)$ via Riemannian submersions. However all metrics on $\on{Imm}_f(M,\R^d)$, that we consider, arise as restrictions of metrics on $\on{Imm}(M,\R^2)$. Thus, when dealing with parametrized shapes, we will use the space $\on{Imm}(M,\R^2)$ and restrict ourselves to the open and dense subset $\on{Imm}_f(M,\R^2)$, whenever we consider the space  $B_{i,f}(M,\R^2)$ of unparametrized shapes.

\subsection{Diffeomorphism group}

Concerning the Lie group structure of the diffeomorphism group we have the following theorem.

\begin{theorem}[Thm. 43.1, \cite{Michor1997}]
Let $M$ be a smooth manifold, not necessarily compact. The group
\begin{multline*}
\on{Diff}_c(M) = \big\{ \ph : \ph, \ph^{-1} \in C^\infty(M,M),\\
\{x:\ph(x) \ne x \} \text{ has compact closure} \big\}
\end{multline*}
of all compactly supported diffeomorphisms is an open submanifold of $C^\infty(M,M)$ (equipped wit 
a refinement of the Whitney $C^\infty$-topology) and 
composition and inversion are smooth maps. It is a regular Lie group and the Lie algebra is the 
space $\mf X_c(M)$ of all compactly supported vector fields, whose bracket is the negative of the usual Lie 
bracket.  
\end{theorem}

An infinite dimensional smooth Lie group $G$ with Lie algebra $\mf g$ is called regular, if the following two conditions hold:
\begin{itemize}
\item 
For each smooth curve 
$X\in C^{\infty}(\mathbb R,\mathfrak g)$ there exists a unique smooth curve 
$g\in C^{\infty}(\mathbb R,G)$ whose right logarithmic derivative is $X$, i.e.,
\begin{equation}
\label{eq:regular}
\begin{split} g(0) &= e \\
\p_t g(t) &= T_e(\mu^{g(t)})X(t) = X(t).g(t)\;.
\end{split} 
\end{equation}
Here $\mu^g :G\to G$ denotes the right multiplication: $$\mu^g x=x.g\,.$$
\item
The map $\on{evol}^r_G: C^{\infty}(\mathbb R,\mathfrak g)\to G$ is smooth, where 
$\on{evol}^r_G(X)=g(1)$ and $g$ is the 
unique solution of \eqref{eq:regular}.
\end{itemize}

If $M$ is compact, then all diffeomorphisms have compact support and $\on{Diff}_c(M) = \on{Diff}(M)$. For $\R^n$ the group $\on{Diff}(\R^d)$ of all orientation preserving diffeomorphisms is not an open subset of $C^\infty(\R^d,\R^d)$ endowed with the compact $C^\infty$-topology and thus it is not a smooth manifold with charts in the usual sense.  Therefore, it is necessary to work with the smaller space $\on{Diff}_c(\R^d)$ of compactly supported diffeomorphisms. In Sect.\ \ref{diff_groups} we will also introduce the groups  $\Diff_{H^{\infty}}(\R^d)$ and $\Diff_{\mc S}(\R^d)$ with weaker decay conditions towards infinity. Like $\on{Diff}_c(\R^d)$ they are smooth regular Lie groups.

\subsection{The space of Riemannian metrics}

We denote by $\Met(M)$ the space of all smooth Riemannian metrics. Each $g \in \Met(M)$ is a symmetric, positive definite $0 \choose 2$ tensor field on $M$, or equivalently a pointwise positive definite section of the bundle $S^2 T^\ast M$.

\begin{theorem}[Sect. 1.1, \cite{Gil-Medrano1991}]
Let $M$ be a compact manifold without boundary. The space $\Met(M)$ of all Riemannian metrics on $M$ is an open subset of the space $\Ga(S^2 T^\ast M)$ of all symmetric $0 \choose 2$ tensor fields and thus itself a smooth Fr\'echet-manifold.
\end{theorem}

For each $g \in \Met(M)$ and $x \in M$ we can regard $g(x)$ as either a map
\[
g(x): T_xM \x T_x M \to \R
\]
or as an invertible map
\[
g_x: T_xM \to T_x^\ast M\,.
\]
The latter interpretation allows us to compose $g, h \in \Met(M)$ to obtain a fiber-linear map $g^{-1}.h:TM \to TM$.

\section{The $L^2$-metric on plane curves}

\subsection{Properties of the $L^2$-metric}

We first look at the simplest shape space, the space of plane curves. 
In order to induce a metric on the manifold of un-parametrized curves $B_{i,f}(S^1,\R^2)$ 
we need to define a metric on parametrized curves 
$\Imm(S^1,\R^2)$, that 
is invariant under reparametrizations, c.f.~Sect.~\ref{Submersion}.
The simplest such metric on the space of immersed plane curves  is the $L^2$-type metric
\begin{align*}
G_c^0(h, k) = \int_{S^1} \langle h(\th), k(\th) \rangle \ud s\,.
\end{align*}
The horizontal bundle of this metric, when restricted to $\on{Imm}_f(S^1,\R^2)$, consists of all tangent vectors, $h$ that are pointwise orthogonal
to $c_\theta$, i.e., $h(\th)=a(\th)n_c(\th)$ for some scalar function $a\in C^{\infty}(S^1)$.
An expression for the metric on the quotient space $B_{i,f}(S^1,\R^2)$, using  the charts from Thm. \ref{bi_bundle}, is given by 
\[
G_C^0(T_c\pi(a.n_c), T_c\pi(b.n_c)) = \int_{S^1} a(\th) b(\th) \ud s\,.
\]

This metric was first studied in the context of shape analysis in \cite{Michor2006c}. 
The geodesic equation for the $G^0$-metric  on $\on{Imm}_f(S^1,\R^2)$ is given by
\begin{equation}
\label{ge_l2_imm}
\left(|c_{\th}| c_t\right)_t = -\frac 12 \left( \frac{|c_t|^2 c_\th}{|c_\th|}\right)_\th\,.
\end{equation}
Geodesics on $B_{i,f}(S^1,\R^2)$ correspond to horizontal geodesics on 
$\on{Imm}_f(S^1,\R^2)$ by Thm.~\ref{thm:submersion}; these satisfy $c_t = a.n_c$, with a scalar function $a(t,\th)$. 
Thus the geodesic equation \eqref{ge_l2_imm} reduces to an equation for $a(t,\th)$,
\[
a_t = \frac 12 \ka a^2\,.
\]
Note that this is not an ODE for $a$, because $\ka_c$, being the curvature of $c$, depends implicitly on $a$. 
It is however possible to eliminate $\ka$ and arrive at (see \cite[Sect.~4.3]{Michor2006c})
\begin{multline*}
a_{tt}-4 \frac {a_t^2}{a} - \frac{a^6 a_{\th\th}}{2w^4} + \frac{a^6a_{\th}w_\th}{w^5} - \frac{a^5a_\th^2}{w^4} = 0 \\
  w(\th)= a(0,\th)\sqrt{|c_\th(0,\th)|}\,,
\end{multline*}
a nonlinear hyperbolic PDE of second order.

\begin{openquestion}
Are the geodesic equations on either of the spaces $\on{Imm}(S^1,\R^2)$ or $B_{i,f}(S^1,\R^2)$ for the $L^2$-metric (locally) well-posed?
\end{openquestion}

The $L^2$-metric is among the few for which the sectional curvature on $B_{i,f}(S^1,\R^2)$ has a simple expression. 
Let $C = \pi(c) \in B_{i,f}(S^1,\R^2)$ and choose $c \in \on{Imm}_f(S^1,\R^2)$ such that it is parametrized by constant speed. 
Take $a.n_c, b.n_c \in \on{Hor}_{G^0}(c)$ two orthonormal horizontal tangent vectors at $c$. 
Then the sectional curvature of the plane spanned by them 
is given by the Wronskian
\begin{equation}
\label{l2_sec_curv}
k_C(P(T_c\pi(a.n_c),T_c\pi(b.n_c)) = \frac 12 \int_{S^1} \left(ab_\th - a_\th b\right)^2 \ud s\,.
\end{equation}
In particular the sectional curvature is non-negative and unbounded. 

\begin{remark}
This metric has a natural generalization to the space $\on{Imm}(M,\R^d)$ 
of immersions of an arbitrary compact manifold $M$. 
This can be done by replacing the integration over arc-length with integration 
over the volume form of the induced pull-back metric. For $q\in\Imm(M,\R^d)$ the metric is defined by
\begin{align*}
G_q^0(h, k) = \int_{M} \langle h(x), k(x) \rangle \on{vol}(q^*\bar g)\,.
\end{align*}
The geodesic spray of this metric was computed in \cite{Binz1980} and the curvature in \cite{Kainz1984}.
\end{remark}

For all its simplicity the main drawback of the $L^2$-metric is that the induced geodesic distance vanishes on $\on{Imm}(S^1,\R^2)$. 
If $c :[0,1] \to \on{Imm}(S^1,\R^2)$ is a path, denote by
\[
\on{Len}^{G}_{\on{Imm}}(c) = \int_0^1 \sqrt{ G_{c(t)}(c_t(t), c_t(t))} \ud t
\]
its length.
The geodesic distance between two points is defined as the infimum of the pathlength over all paths connecting the two points,
\[
\on{dist}^G_{\on{Imm}}(c_0, c_1) = \inf_{\substack{c(0)=c_0, \\c(1)=c_1}} \on{Len}^G_{\on{Imm}}(c)\,.
\]
For a finite dimensional Riemannian manifold $(M,G)$ this distance is always positive,  due to  the local invertibility of the exponential map. This does not need to be true for a weak Riemannian metric in infinite dimensions and the $L^2$-metric on $B_{i,f}(S^1,\R^2)$ was the first known example, where this was indeed false. We have the following result.

\begin{theorem}
The geodesic distance function induced by the metric $G^0$ vanishes identically on $\on{Imm}(S^1,\R^2)$ and $B_{i,f}(S^1,\R^2)$.

For any two curves $c_0, c_1 \in \on{Imm}(S^1,\R^2)$ and $\ep >0$ there exists a smooth path $c :[0,1] \to \on{Imm}(S^1,\R^2)$ with $c(0)=c_0$, $c(1)=c_1$ and length $\on{Len}^{G^0}_{\on{Imm}}(c) < \ep$.
\end{theorem}

For the space $B_{i,f}(S^1,\R^2)$ an explicit construction of the path with arbitrarily short length was given in \cite{Michor2006c}. Heuristically, if the curve is made to zig-zag wildly, then the normal component of the motion will be inversely proportional to the length of the curve. Since the normal component is squared the length of the path can be made arbitrary small.  This construction is visualized in Fig.\ref{fig:vanishing}.

For $\on{Imm}(S^1,\R^2)$ vanishing of the geodesic distance is proven in \cite{Bauer2012c}; the proof makes use of the vanishing of the distance on $B_{i,f}(S^1,\R^2)$ and on $\Diff(S^1)$.

\begin{figure}[ht]
\includegraphics[width=.48\textwidth]{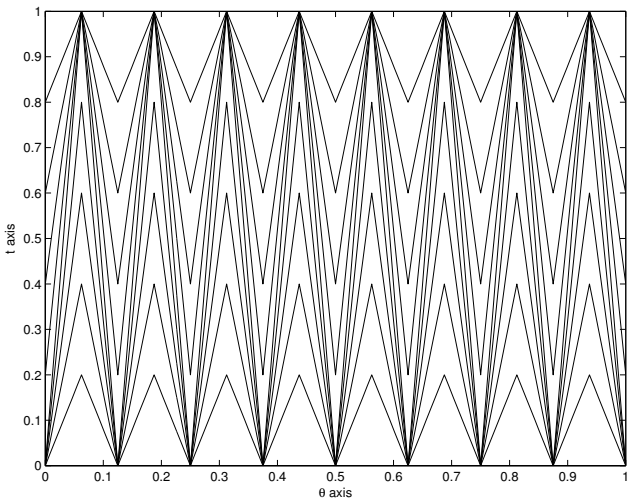}
\caption{A short path in the space of un-parametrized curves that connects the bottom to the top line. Original image published in \cite{Michor2006c}.}
\label{fig:vanishing}
\end{figure}
\begin{remark}
 In fact, this result holds more generally for the space $\on{Imm}(M, \R^d)$. One can also replace $\R^d$ by an arbitrary Riemannian manifold $N$; see~\cite{Michor2005}.
\end{remark}

The vanishing of the geodesic distance leads us to consider stronger metrics that prevent this behavior. 
In this article we will present three different classes of metrics:
\begin{itemize}
 \item Almost local metrics: $$G_q^\Ps(h,k) = \int_M \Ps(q) \langle h, k \rangle \on{vol}(q^*\bar g),$$ where  
 $\Psi:\on{Imm}(M,\R^d)\rightarrow C^\infty(M, \R_{>0})$ is a suitable smooth function.
 \item Sobolev type metrics:
  $$G_q^L(h,k) = \int_M  \langle L_{q} h, k \rangle \on{vol}(q^*\bar g),$$
  where $L_{q}: T_{q}\Imm(M,\R^d)\rightarrow T_{q}\Imm(M,\R^d)$ is a suitable differential operator.
  \item Metrics that are induced by right invariant metrics on the diffeomorphism group of the ambient space.
\end{itemize}

\subsection{Gradient flows on curves}

The $L^2$-metric is used in geometric active contour models to define gradient flows for various energies. For example the curve shortening flow
\[
c_t = \ka_c n_c
\]
is the gradient flow of the energy $E(c) = \int_{S^1} \ud s = \ell_c$ with respect to the $L^2$-metric.

The following example is taken from \cite{Mennucci2009}. The centroid based energy $E(c) = \tfrac 12 |\mu(c) - w|^2$, with $w \in \R^2$ fixed and $\mu(c) = \tfrac 1{\ell_c}\int_{S^1} c \ud s$ denoting the centroid, attains its minimum when $\mu(c) = w$. The $L^2$-gradient of the energy is
\[
\nabla^0 E(c) = \langle \mu(c) - w, n_c \rangle n_c + \ka_c \langle \mu(c) - c, \mu(c) - w \rangle n_c\,.
\]
We see from the second term that the gradient flow
\[
c_t = - \nabla^0 E(c)
\]
tries to decrease the length of the curve for points with $\langle \mu(c) - c, \mu(c) - w \rangle \leq 0$, but increase for $\langle \mu(c) - c, \mu(c) - w \rangle > 0$. This latter part is ill-posed. However the ill-posedness of the gradient flow is not an intrinsic property of the energy, it is a consequence of the metric we chose to define the gradient. For example the gradient flow with respect to the $H^1$-metric
\[
G_c^1(h,k) = \int_{S^1} \tfrac 1{\ell_c} \langle h, k \rangle + \ell_c \langle D_s h, D_s k \rangle \ud s
\]
is locally well-posed. See \cite{Mennucci2007,Sundaramoorthi2011,Sundaramoorthi2008} for more details on Sobolev active contours and applications to segmentation and tracking. The same idea has been employed for gradient flows of surfaces in \cite{Zhang2008}.

\section{Almost local metrics on shape space}\label{Almost_local}

Almost local metrics are metrics of the form
\[G_q^\Psi(h,k)=\int_M \Psi(q) \langle h, k \rangle \on{vol}(q^*\bar g)\,,\]
where $\Ps:\Imm(M,\R^d)\to C^{\infty}(M,\R_{>0})$ is a smooth function
that is equivariant with respect to the action of $\on{Diff}(M)$, i.e.,
\[\Psi(q\on{\circ}\ph)=\Psi(q)\on{\circ}\ph\,,\quad q\in\on{Imm}(M,\R^d)\,,\; \ph\in\on{Diff}(M)\,.\]
Equivariance of $\Ps$ then implies the invariance of $G^\Ps$ and thus $G^\Ps$ induces a Riemannian metric on the quotient $B_{i,f}(M,\R^d)$.

Examples of almost local metrics that have been considered are of the form
\begin{equation}\label{almost-local-metric}
G_q^\Ph(h,k)=\int_M \Ph(\on{Vol}_q,H_q,K_q) \langle h, k \rangle \on{vol}(q^*\bar g)\,,
\end{equation}
where $\Ph \in C^\infty(\R^3,\R_{>0})$ is a function of the total volume $\on{Vol}_q$, the mean 
curvature $H_q$ and the Gau\ss{} curvature $K_q$. The name ``almost local'' is derived from the 
fact that while $H_q$ and $K_q$ are local quantities, the total volume $\on{Vol}_q$ induces a mild 
non-locality in the metric. If $\Ph=\Ph(\on{Vol})$ depends only on the total volume, the resulting 
metric is conformally equivalent to the $L^2$-metric, the latter corresponding to $\Ph\equiv 1$.

For an almost local metric $G^\Ps$ the horizontal bundle at $q \in \on{Imm}_f(M,\R^d)$ consists of those tangent vectors $h$ that are pointwise orthogonal to $q$,
\begin{multline*}
\on{Hor}^{\Ps}(q) = \{ h \in T_q \on{Imm}_f(M,\R^d) : \\
h=a.n_q\,, a \in C^\infty(M,\R)\}\,.
\end{multline*}
Using the charts from Thm. \ref{bi_bundle}, 
the metric $G^\Psi$ on $B_{i,f}(M,\R^d)$ is given by
\[G_{\pi(q)}^\Psi\left(T_q\pi(a.n_q),T_q\pi(b.n_q)\right)
=\int_M \Psi(q).a.b\, \on{vol}(q^*\bar g)\,,\]
with $a, b \in C^\infty(M,\mathbb R)$.

Almost local metrics, that were studied in more detail include the curvature weighted $G^A$-metrics
\begin{equation}
\label{ga_metric}
G^A_c(h,k)=\int_{S^1} (1+A\ka^2_c)\langle h,k\rangle\ud s\,,
\end{equation}
with $A>0$ in \cite{Michor2006c} and the conformal rescalings of the $L^2$-metric
\[
G^\Ph_c(h,k)= \Ph(\ell_c) \int_{S^1} \langle h,k\rangle\ud s\,,
\]
with $\Ph\in C^{\infty}(\R_{>0},\R_{>0})$ in \cite{YezziMennucci2004a, Shah2008}, both on the space of plane curves. More general almost local metrics on the space of plane curves were considered in \cite{Michor2007} and they have been generalized to hypersurfaces in higher dimensions in \cite{Bauer2010, Bauer2012a, Bauer2012b}.

\subsection{Geodesic distance}
\label{al_dist}

Under certain conditions on the function $\Ps$ almost local metrics are strong enough to induce a point-separating geodesic distance function on the shape space.

\begin{theorem}
\label{thm:al_non_vanish}
If $\Psi$ satisfies one of the following conditions
\begin{enumerate}
\item
$\Psi(q)\geq 1+A H^2_q$\label{conditionGA}
\item
$\Psi(q)\geq A \Vol_q$\label{conditionConformal}
\end{enumerate}
for some $A>0$, then the metric $G^{\Psi}$ induces a point-separating geodesic distance function on $B_{i,f}(M,\R^d)$, i.e., for $C_0 \neq C_1$ we have $\on{dist}^\Ps_{B_{i,f}}(C_0, C_1) > 0$.
\end{theorem}

For planar curves the result under assumption \ref{conditionGA} is proven in \cite[Sect. 3.4]{Michor2006c} and under assumption \ref{conditionConformal} in \cite[Thm. 3.1]{Shah2008}. The proof was generalized to the space of hypersurfaces in higher dimensions in \cite[Thm.~8.7]{Bauer2012a}.

The proof is based on the observation that under the above assumptions the $G^\Ps$-length of a path of immersions can be bounded from below by the area swept out by the path. A second ingredient in the proof is the Lipschitz-continuity of the function $\sqrt{\on{Vol}_q}$.

\begin{theorem}
If $\Ps$ satisfies
\begin{align*}
\Psi(q)&\geq 1+A H^2_q\,,
\end{align*}
then the geodesic distance satisfies
\[
\left|\sqrt{\Vol_{Q_1}}-\sqrt{\Vol_{Q_2}}\right| \leq  \frac{1}{2 \sqrt{A}} \on{dist}^{G^\Psi}_{B_{i,f}}(Q_1,Q_2)\,.
\]
In particular the map
\[\sqrt{\Vol}:(B_{i,f}(M,\R^d),\on{dist}^{G^\Psi}_{B_{i,f}}) \to \mathbb R_{\geq 0}\]
is Lipschitz continuous.
\end{theorem}

This result is proven in \cite[Sect. 3.3]{Michor2006c} for plane curves and in \cite[Lem. 8.4]{Bauer2012a} for hypersurfaces in higher dimensions.

In the case of planar curves \cite{Shah2008} showed that for the almost local metric with $\Ps(c)=\ell_c$ the geodesic distance on $B_{i,f}(S^1,\R^2)$ is not only bounded by but equal to the infimum over the area swept out,
\[
\on{dist}^{\ell}_{B_{i,f}}(C_0, C_1) = \inf_{\substack{\pi(c(0)) = C_0 \\ \pi(c(1)) = C_1}}
\int_{S^1\x [0,1]} \!\left| \on{det} dc(t, \th) \right| \ud \th \ud t.
\]

\begin{remark}
No almost local metric can induce a point separating geodesic distance function on $\on{Imm}_f(M,\R^d)$ and thus neither on $\on{Imm}(M,\R^d)$. When we restrict the metric $G^\Ps$ to an orbit $q \on{\circ} \on{Diff}(M)$ of the $\on{Diff}(M)$-action, the induced metric on the space $q \on{\circ} \on{Diff}(M) \cong \on{Diff}(M)$ is a right-invariant weighted $L^2$-type metric, for which the geodesic distance vanishes. Thus
\[
\on{dist}^\Ps_{\on{Imm}}(q, q\on{\circ} \ph) = 0\,,
\]
for all $q \in \on{Imm}_f(M,\R^d)$ and $\ph \in \on{Diff}(M)$. See Sect. \ref{diff_dist} for further details.

This is not a contradiction to Thm. \ref{thm:al_non_vanish}, since a point-separating distance on the quotient $B_{i,f}(M,\R^d)$ only implies that the distance on $\on{Imm}_f(M,\R^d)$ separates the fibers of the projection $\pi : B_{i,f}(M,\R^d) \to \on{Imm}_f(M,\R^d)$. On each fiber $\pi^{-1}(C) \subset \on{Imm}_f(M,\R^d)$ the distance can still be vanishing, as it is the case for the almost local metrics.
\end{remark}

It is possible to compare the geodesic distance on shape space with the Fr\'echet distance. The Fr\'echet distance is defined as
\begin{equation}
\label{frechet_dist}
\operatorname{dist}^{L^\infty}_{B_{i,f}}(Q_0,Q_1) = \inf_{q_0,q_1} \|q_0 - q_1\|_{L^\infty},
\end{equation}
where the infimum is taken over all immersions $q_0, q_1$ with $\pi(q_i)=Q_i$. Depending on the behavior of the metric under scaling, it may or may not be possible to bound the Fr\'echet distance by the geodesic distance.

\begin{theorem}[Thm. 8.9, \cite{Bauer2012a}]
\label{thm_lip_noncont}
If $\Ps$ satisfies one of the conditions,
\begin{enumerate}
\item
$\Psi(q) \leq C_1 + C_2 H_q^{2k}$
\item
$\Psi(q) \leq C_1 \Vol_q^k$,
\end{enumerate}
with some constants $C_1, C_2 > 0$ and $k < \tfrac{d+1}2$, then there exists no constant $C>0$, such that
\[
\on{dist}^{L^\infty}_{B_{i,f}}(Q_0,Q_1) \leq C \on{dist}^{\Ps}_{B_{i,f}}(Q_0,Q_1)\,,
\]
holds for all $Q_0,Q_1 \in B_{i,f}(M,\R^d)$.
\end{theorem}

Note that this theorem also applies to the $G^A$-metric for planar curves defined in \eqref{ga_metric}. Even though Thm. \ref{thm_lip_noncont} states that the identity map
\begin{equation*}
\iota:\Big(B_{i,f}(S^1,\R^2),\dist^{G^A}\Big)\rightarrow 
\Big(B_{i,f}(S^1,\R^2),\dist^{L^\infty}\Big)
\end{equation*}
is not Lipschitz continuous, it can be shown that is continuous and thus the topology induces by $\on{dist}^{G^A}$ is stronger than that induced by $\on{dist}^{L^\infty}$.

\begin{theorem}[Cor. 3.6, \cite{Michor2006c}]
The identity map on $B_{i,f}(S^1,\R^2)$ is continuous from $(B_{i,f}(S^1,\R^2),\dist^{G^A})$ to $(B_{i,f}(S^1,\R^2),\dist^{L^\infty})$ and uniformly continuous on every subset, where the length $\ell_C$ is bounded.
\end{theorem}

As a corollary to this result we obtain another proof that the geodesic distance for the $G^A$-metric is point-separating on $B_{i,f}(S^1,\R^2)$.

\subsection{Geodesic equation}

Since geodesics on $B_{i,f}$ correspond to horizontal geodesics on $\on{Imm}_f(M,\R^d)$, see Thm.~\ref{thm:submersion}, to compute the geodesic equation on $B_{i,f}(M,\R^d)$ it is enough to restrict the geodesic equation on $\on{Imm}_f(M,\R^d)$ to horizontal curves.

As an example for the resulting equations we will present the geodesic equations on $B_{i,f}(M,\R^d)$ for the almost local metric with $\Ps(q) = 1 + AH_q^2$, which is a generalization of the metric \eqref{ga_metric}, and the family of metrics with
$\Ps(q) = \Ph(\on{Vol}_q)$, which are conformal rescalings of the $L^2$-metric.

\begin{theorem}
Geodesics of the almost local $G^\Ps$-metric with $\Ps(q) = 1 + AH_q^2$ on $B_{i,f}(M,\R^d)$ are given by solutions of
\begin{align*}
q_t &= an_q\,,\qquad g=q^\ast\ol g\,,\\
a_t&=\tfrac12 a^2 H_q+\tfrac{2A}{1+AH_q^2} g(H_q \nabla^g a + 2a \nabla^g H_q,\nabla^g a)\\
&\qquad\qquad\;-\tfrac{Aa^2}{1 + AH_q^2} \left(\De^g H_q + \on{Tr}\left(\left(g^{-1} s_q\right)^2\right) \right)\,.
\end{align*}
For the family of metrics with $\Ps(q) = \Ph(\on{Vol}_q)$ geodesics are given by
\begin{align*}
q_t&=\frac{b(t)}{\Ph(\Vol_q)} n_q\,,\qquad g=q^{*}\ol g\,, \\
a_t&=\frac{H_q}{2
\Ph(\Vol_q)}\left(a^2-\frac{\Ph'(\Vol_q)}{\Ph(\Vol_q)}
\int_M a^2 \on{vol}({g})\right)\,.
\end{align*}
\end{theorem}
For the $G^A$-metric and planar curves the geodesic equation was calculated in \cite[Sect.~4.1]{Michor2006c}, whereas for conformal metrics on planar curves it is presented in \cite[Sect.\ 4]{Shah2008}.
For hypersurfaces in higher dimensions the equations are calculated in  \cite[Sect.\ 10.2 and 10.3]{Bauer2012a}.

Note that both for $A=0$ and $\Phi(q)\equiv 1$ one recovers the geodesic equation for the $L^2$-metric, 
\[a_t=\frac12 H_q a^2.\]

Similarly to the case of the $L^2$-metric it is unknown, whether the geodesic equations are well-posed.

\begin{openquestion}
Are the geodesic equations on either of the spaces $\on{Imm}(M,\R^2)$ or $B_{i,f}(M,\R^2)$ for the almost local metrics (locally) well-posed?
\end{openquestion}

\subsection{Conserved quantities}\label{almost_conserved}
If the map $\Ps$ is equivariant with respect to the $\on{Diff}(M)$-action, then the $G^\Ps$-metric 
is invariant, and we obtain by Noether's theorem that the reparametrization momentum is constant 
along each geodesic. The reparametrization momentum for the $G^\Ps$-metric is given by
\[
\Ps(q) g(q_t^\top,\cdot) \on{vol}(q^\ast \ol g) \in \Ga(T^\ast M \otimes \La^{d-1}T^*M)\,,
\]
with $g = q^\ast \ol g$ and the pointwise decomposition of the tangent vector $q_t = q_t^\top + q_t^\bot$ of $q_t$ into $q_t^\bot = \ol g(q_t, n_q)n_q$ and $q_t^\top(x) \in T_{q(x)} q(M)$. This means that for each $X \in \mf X(M)$ we have
\[
\int_M \Ps(q) g(q_t^\top, X) \on{vol}(q^\ast \ol g) = \textrm{const.}
\]

If $\Ps$ is additionally invariant under the action of the Euclidean motion group $\R^d \rtimes \on{SO}(d)$, i.e., $\Ps(O.q + v) = \Ps(q)$, then so is the $G^\Ps$-metric and by Noether's theorem the linear and angular momenta are constant along geodesics. These are given by
\begin{align*}
\int_M \Ps(q) q_t \on{vol}(q^\ast \ol g) &\in \R^d \\
\int_M \Ps(q)\,  q \wedge q_t \on{vol}(q^\ast \ol g) &\in \bigwedge\nolimits^2 \R^d \cong \mf{so}(d)^\ast\,.
\end{align*}
The latter means that for each $\Om \in \mf{so}(d)$ the quantity
\[
\int_M \Ps(q) \ol g(\Om.q, q_t) \on{vol}(q^\ast \ol g) 
\]
is constant along geodesics.

If the function $\Psi$ satisfies the scaling property
\[\Psi(\la q)=\la^{-\on{dim}(M) - 2}\Ps(q),\, q\in\Imm(M,\R^d),\, \la\in\R_{>0},\]
then the induced metric $G^\Ps$ is scale invariant. In this case the scaling momenta are conserved along geodesics as well:
\[
\int_M\Ps(q)
 \langle q,q_t\rangle\on{vol}(q^*\ol g)\qquad\qquad \textrm{scaling momentum}
\]
For plane curves the momenta are
\begin{align*}
\int_{S^1} \Ps(c) \langle c_\theta, c_t \rangle \mu \ud s & & \textrm{reparametrization momentum}\\
\int_{S^1} \Ps(c) c_t \ud s & & \textrm{linear momentum}\\
\int_{S^1} \langle Jc, c_t \rangle \ud s & & \textrm{angular momentum}\\
\int_{S^1} \Ps(c) \langle c,c_t\rangle \ud s & & \textrm{scaling momentum}
\end{align*}
with $\mu \in C^\infty(S^1)$ and $J$ denoting rotation by $\tfrac{\pi}2$.

\subsection{Completeness}
\label{al_compl}

Regarding geodesic completeness, one can look at the set of spheres with a common center. This 
one-dimen\-sional subset  of $B_{i,f}(S^{d-1},\R^d)$
is a totally geodesic submanifold, i.e., a geodesic up to 
parameterization. One can explicitly calculate the length of this geodesic as spheres shrink towards a 
point and when they expand towards infinity. When it is possible to shrink to a point with a 
geodesic of finite length, the space can obviously not be geodesically complete. This is the case 
under the following conditions.   

\begin{theorem}[Thm. 9.1, \cite{Bauer2012a}]
\label{al_incomplete}
If $\Ps$ satisfies one of the conditions,
\begin{enumerate}
\item
$\Psi(q) \leq C_1 + C_2 H_q^{2k}$
\item
$\Psi(q) \leq C_1 \Vol_q^k$,
\end{enumerate}
with some constants $C_1, C_2 > 0$ and $k < \tfrac{d+1}2$, then the spaces $\Imm(S^{d-1},\R^d)$ and $B_{i,f}(S^{d-1},\R^d)$ are not geodesically complete with respect to the $G^\Ps$-metric.
\end{theorem}

Note that these are the same conditions as in Thm. \ref{thm_lip_noncont}. For other choices of $M$ scalings will in general not be geodesic, but under the same condition an immersion can be scaled down to a point with finite energy. What conditions are sufficient to prevent geodesics from developing singularities and thus make the spaces geodesically complete is unknown.

Concerning metric completeness, it cannot be expected that a weighted $L^2$-type metric will be able to prevent immersions from losing smoothness in the completion. We have only a partial result available for the $G^A$-metric \eqref{ga_metric} on plane curves.

Similarly to the definition of $B_{i,f}(S^1,\R^2)$, we can define the larger space
\[
B^{lip}_i(S^1,\R^2) = \on{Lip}(S^1,\R^2) / \sim
\]
of equivalence classes of Lipschitz curves. We identify two Lipschitz curves, if they differ by a {\it monotone correspondence}. This can be thought of as a generalization of reparametrizations, which allow for jumps and intervals of zero speed; see \cite[Sect.\ 2.11]{Michor2006c}. Equipped with the Fr\'echet-distance \eqref{frechet_dist}, the space $B^{lip}_i(S^1,\R^2)$ is metrically complete.

A curve $C \in B^{lip}_i(S^1,\R^2)$ is called a {\it 1-BV rectifiable curve}, if the turning angle function $\al$ of an arc-length parametrized lift $c \in \on{Lip}(S^1,\R^2)$ of $C$ is a function of bounded variation.

\begin{theorem}[Thm. 3.11, \cite{Michor2006c}]
The completion of the metric space $(B_{i,f}(S^1,\R^2), \on{dist}^{G^A})$ is contained in the shape space $B^{lip}_i(S^1,\R^2)$ of Lipschitz curves and it contains all 1-BV rectifiable curves.
\end{theorem}

\subsection{Curvature}

The main challenge in computing the curvature for almost local metrics on $\on{Imm}(M,\R^d)$ is 
finding enough paper to finish the calculations. It is probably due to this that apart from the 
$L^2$-metric we are not aware of any curvature calculations on the space $\on{Imm}(M,\R^d)$. For 
the quotient space $B_{i,f}(M,\R^d)$ the situation is a bit better and the formulas a bit shorter. This is because $B_{i,f}(M,\R^d)$ is modeled on $C^\infty(M,\R)$, while the space $\on{Imm}(M,\R^d)$ is modeled on $C^\infty(M,\R^d)$. In coordinates elements of $B_{i,f}(M,\R^d)$ are represented by scalar functions, while immersions need functions with $d$ components.
For plane curves and conformal metrics the curvature has been calculated in \cite{Shah2008} and for $\Psi(c)=\Phi(\ell_c,\ka_c)$ in \cite{Michor2007}. Similarly for higher dimensional surfaces the curvature has been calculated for $\Psi(q)=\Phi(\Vol_q,H_q)$ in \cite{Bauer2012a}.

The sectional curvature for the $L^2$-metric on plane curves \eqref{l2_sec_curv} is non-negative. In general the expression for the sectional curvature for almost local metrics with $\Ph \not\equiv 1$ will contain both positive definite, negative definite and indefinite terms. For example the sectional curvature of the metric \eqref{ga_metric} with $\Ps(c) = 1+A\ka_c^2$ on plane curves has the following form.

\begin{theorem}[Sect. 4.6, \cite{Michor2006c}]
Let $C \in B_{i,f}(S^1,\R^2)$ and choose $c \in \on{Imm}(S^1,\R^2)$ such that $C = \pi(c)$ and $c$ is parametrized by constant speed. Let $a.n, b.n \in \on{Hor}(c)$ two orthonormal horizontal tangent vectors at $c$. Then the sectional curvature of the plane spanned by $a, b \in T_C B_{i,f}(S^1,\R^2)$ for the $G^A$ metric is 
\begin{multline*}
k_C(P(a,b))=-\int_{S^1}A(a.D_s^2b-b.D_s^2a)^2ds\\ +\int\frac{(1-A\ka^2)^2-4A^2\ka.D_s^2\ka+8A^2(D_s\ka)^{2}}{2(1+A\ka^2)}\cdot \\
\cdot(a.D_sb-b.D_sa)^2ds
\end{multline*}
\end{theorem}

It is assumed, although not proven at the moment, that for a generic immersion, similar to Thm. \ref{thm_curv_pm}, the sectional curvature will assume both signs.

\subsection{Examples}

To conclude the section we want to present some examples of numerical solutions to 
the geodesic boundary value problem for given shapes $Q_0,Q_1 \in B_{i,f}(M,\R^d)$ with metrics of 
the form \eqref{almost-local-metric}. 
One method  to tackle this problem is to directly minimize the  horizontal path energy
$$E^{\operatorname{hor}}(q) = \int_0^1 \int_M \Phi(\operatorname{Vol}_q,H_q) \langle q_t,n_q 
\rangle^2 \operatorname{vol}(q^*\bar g)\,\ud t$$
over the set of paths $q$ of immersions with fixed endpoints $q_0, q_1$ that project onto the target surfaces $Q_0$ and $Q_1$, i.e.,  $\pi(q_i)=Q_i$. 
The main advantage of this approach  for the class of almost local metrics lies in the simple form of the horizontal bundle. 
Although we will only show one specific example in this article it is worth to note that several numerical experiments are available; see:
\begin{itemize}
 \item \cite{Michor2006c,Michor2007} for the $G^A$--metric and planar curves.
 \item \cite{YezziMennucci2004a,YezziMennucci2005} for conformal metrics and  planar curves.
 \item \cite{Bauer2010,Bauer2012a,Bauer2012b} for surfaces in $\R^3$.
\end{itemize}
\begin{figure}[ht]
\includegraphics[width=.48\textwidth]{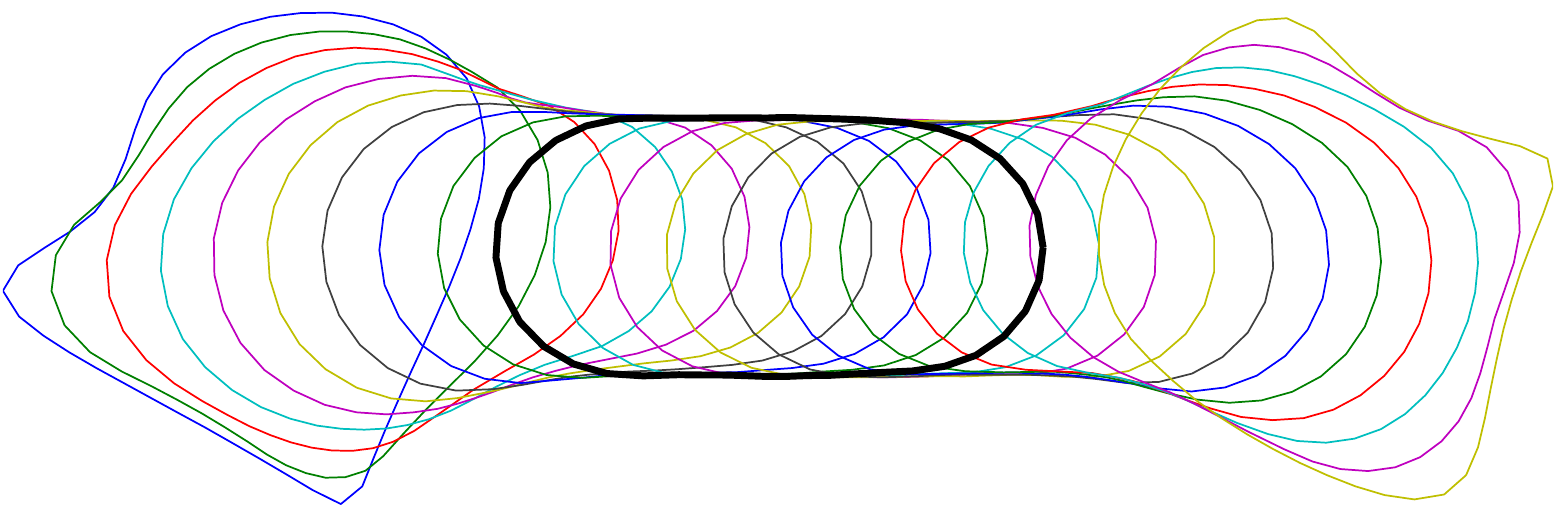}
\caption{A geodesic in the $G^A$--metric joining two 
shapes of size about $1$ at distance $5$ apart with
$A=.25$, using $20$ time samples and a 48-gon approximation for all curves. Original image published in \cite{Michor2006c}.}
\label{fig:almost1}
\end{figure}
The example we want to present here, is concerned with the behavior of the $G^A$-metric matching curves that are far apart in space.
In the article \cite{Michor2006c} the authors showed that pure translation of a cigar-like shape with a  cross--section of $2\sqrt{A}$
is (locally) a geodesic for the $G^A$--metric. Thus one might expect that a geodesic between distant curves
will asymptotically utilize this cigar shaped curve, translate this optimal curve and then deform it to the target shape. In fact the 
numerical examples resemble this behavior as can be seen in Fig.\ref{fig:almost1}. Note that the cross--section of the middle figure -- which is highlighted -- 
is slightly bigger than $2\sqrt{A}$. A reason for this might be that the distance between the two boundary shapes is not sufficiently large. In the article \cite{Bauer2012a}
it has been shown that this behavior carries over to the case of higher-dimensional surfaces, c.f. Fig.~\ref{fig:almost2}. Note that the behavior of the geodesics changes dramatically if one increases the distance 
further, namely for shapes that are sufficiently far apart the geodesics will go through a shrink and grow behavior. This phenomenon  is based on the fact that it is possible to shrink a sphere to zero in finite time for the $G^A$--metric. 
Then geodesics of very long translations will go via a strong shrinking part and growing part, and almost all of the translation will be done with the shrunken version of the shape. 
This behavior, which also occurs for the class of conformal metrics, is described in \cite{Bauer2012a}. 
\begin{figure}[ht]
\centering
\includegraphics[width=.48\textwidth]{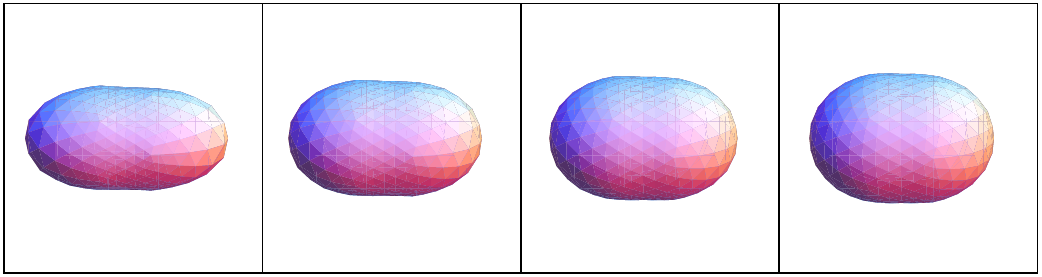}%
\caption{Middle figure of a geodesic between two unit spheres at distance $3$ apart for $A=0.2$, $A=0.4$, $A=0.6$, $A=0.8$. In each of the simulations 20 time steps and a triangulation with 720 triangles were used. Original image published in \cite{Bauer2012a}.} 
\label{fig:almost2}
\end{figure}

\section{Sobolev type metrics on shape space}\label{sobolev_inner}

Sobolev-type inner metrics on the space $\on{Imm}(M,\R^d)$ of immersions are metrics of the form
\[
G_q^L(h,k)=\int_M  \langle L_q h, k \rangle \on{vol}(q^*\ol g)\,,
\]
where for each $q \in \on{Imm}(M,\mathbb R^d)$, $L_q$ is a pseudo-differen\-tial operator on $T_q\on{Imm}(M,\R^d)$. To be precise we assume that the operator field
\[
L : T\on{Imm}(M,\R^d) \to T\on{Imm}(M,\R^d)
\]
is a smooth base-point preserving bundle isomorphism, such that for every $q \in \on{Imm}(M,\R^d)$ the map
\[
L_q : T_q\on{Imm}(M,\R^d) \to T_q\on{Imm}(M,\R^d)
\]
is a pseudo-differential operator, that is symmetric and positive with respect to the $L^2$-metric. 
Ordinarily, $L_q$ will be elliptic and of order $\ge 1$, with the order being constant in $q$. However, the operator 
fields in \cite{Michor2012d_preprint} are not elliptic.
An example for such an operator field $L$ is
\begin{equation}\label{sobolevmetric}
L_qh = h + (\De^g)^l h,\quad l\geq 0\,,
\end{equation}
where $\De^g$ is the Laplacian of the induced metric $g = q^\ast \ol g$ on $M$.

We will also assume that the operator field $L$ is invariant under the action of the reparametrization group $\on{Diff}(M)$, i.e.,
\begin{equation}
\label{L_rep_inv}
(L_q h) \on{\circ} \ph = L_{q\on{\circ}\ph}(h\on{\circ}\ph)\,,
\end{equation}
for all $\ph, q$ and $h$. Then the metric $G^L$ is invariant under $\on{Diff}(M)$ and it induces a Riemannian metric on the quotient space $B_{i,f}(M,\R^d)$.

In contrast to the class of almost local metrics, for whom the horizontal bundle of the submersion 
\[
\on{Imm}_f(M,\R^d)\to B_{i,f}(M,\R^d)
\]
consisted of tangent vectors, that are pointwise orthogonal to the surface, here the horizontal bundle cannot be described explicitly. Instead we have
\[
\on{Hor}^L(q) = \{ h \in T_q\on{Imm}_f(M,\R^d)\,:\, L_q h = a.n_q \}\,,
\]
where $a \in C^\infty(M,\R)$ is a smooth function. Thus to parametrize the horizontal bundle we need to invert the operator $L_q$.

General Sobolev-type inner metrics on the space of immersed plane curves have been studied in \cite{Michor2007} and on surfaces in higher dimensions in \cite{Bauer2011b}. 
Numerical experiments for special cases of order one Sobolev type metrics are presented in the articles \cite{Jermyn2011,Jermyn2012,Bauer2011a}.
\begin{figure}[ht]
\includegraphics[width=.48\textwidth]{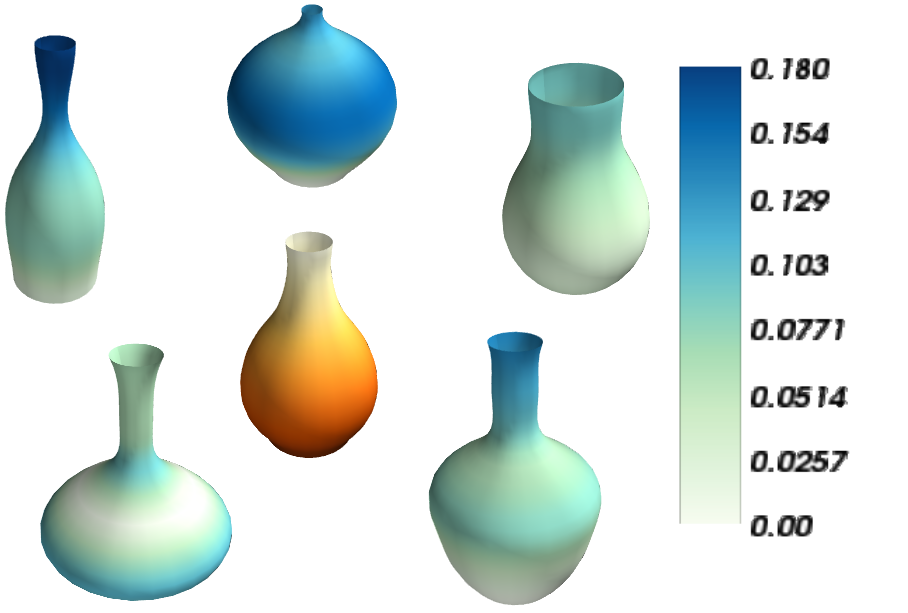}
\caption{In this figure we show the Karcher mean of five vase-shaped objects with respect to the Sobolev metric of order one -- as defined in \eqref{sobolevmetric} -- on the space of parametrized surfaces $\Imm(S^1\times[0,1],\R^3)$. The mean shape, which is displayed in the center of the figure is computed using an iterated shooting method. The colored regions  on the averaged shapes encode the Euclidean length of the initial velocity of the geodesic, which connects each shape to the mean. The color of the mean was chosen for artistic purposes only. Original image published in \cite{Bauer2011a}.}
\label{fig:Karchermean}
\end{figure}

In \cite{Bauer2012d} the authors consider metrics of the form
\[
G_q^L(h,k) = \Ph(\on{Vol}_q) \int_M \langle L_q h, k \rangle \on{vol}(q^*\ol g)\,.
\]
These are a combination of Sobolev-type metrics with a non-local weight function, that can be chosen such that the resulting metric is scale-invariant.
Sobolev type metrics  are far better investigated and understood on the the manifold of plane curves than in higher dimension. 
Therefore, we will discuss this case separately in Sect.\ \ref{Sob:curves}.

\subsection{Sobolev metrics on plane curves}\label{Sob:curves}

A reparametrization invariant Sobolev-type metric on the space of plane curves has been first introduced by Younes in  \cite{Younes1998}. 
There he studied the homogeneous $\dot H^1$ metric
\[G^{\dot H^1}_c(h,k)=\int_{S^1}\langle D_sh, D_sk\rangle \ud s\,.\]
However, this is not a metric on $\Imm(S^1,\R^2)$ but only on the quotient space
$\Imm(S^1,\R^2)/\on{transl}$. In order to penalize bending and stretching of the curve differently it has been generalized  in \cite{Mio2004, Mio2007}
to
\begin{equation}\label{elasticmetric1}
\begin{aligned}
 G_c^{a,b}(h,k)&=\int_{S^1}a^2 \langle D_sh, n_c \rangle \langle D_sk, n_c \rangle \\&\qquad\qquad+b^2 \langle D_sh, v_c\rangle\langle D_sk, v_c\rangle \ud s\;.
\end{aligned}
\end{equation}
In this metric the parameters $a,b$ can be interpreted as the tension and rigidity coefficients of the curves. 
For $a=1, b=\tfrac 12$ a computationally efficient representation of this metric -- called the {\it Square Root Velocity Transform (SRVT)} -- has been found in \cite{Jermyn2011} and it has been generalized 
for arbitrary parameters $a,b$ in \cite{Bauer2012b_preprint}. Following \cite{Jermyn2011} we will 
describe this transformation for the case  $a=1, b=\frac12$: 
\[R: \left\{
\begin{array}{ccc}
\on{Imm}(S^1,\R^2)/\on{transl.}&\to& C^{\infty}(S^1,\R^2) \\
c&\mapsto &\sqrt{|c_\theta|}v.
\end{array} \right.
\]
The inverse of this map is given by
\[R^{-1}: \left\{
\begin{array}{ccc}
C^{\infty}(S^1,\R^2)&\to& \on{Imm}([0,2\pi],\R^2)/\on{transl.} \\
e&\mapsto &\int_0^\theta |e(\si)|e(\si) \ud\si\,.
\end{array} \right.
\]
Here $\on{Imm}([0,2\pi],\R^2)/\on{transl.}$ is viewed as the subspace of curves $c$ with $c(0)=0$.

Note that  $R^{-1}(e)$ is a closed curve if and only if 
\[ \int_0^{2\pi} |e(\theta)|e(\th)\ud\th=0\,.\]

\begin{theorem}
Consider the flat $L^2$-metric 
\[
G_q^{L^2,\on{flat}}(e, f) = \int_{S^1} \langle e,f \rangle \ud\th
\]
on $C^{\infty}(S^1,\R^2)$. The pullback of the metric $G^{L^2,\on{flat}}$ by the $R$-transform is the elastic metric with coefficients $a=1$, $b=\tfrac 12$.

The image of the space $\on{Imm}(S^1,\R^2)/\on{transl.}$ under the $R$-transform is a co-dimension $2$ submanifold of the flat space $C^{\infty}(S^1, \R^2)$.
\end{theorem}

This representation not only allows to efficiently discretize the geodesic equation, but also to compute the curvature of $\on{Imm}(S^1,\R^2)$; See \cite{Jermyn2011,Bauer2012b_preprint} for details.

\begin{figure}[ht]
\includegraphics[width=.48\textwidth]{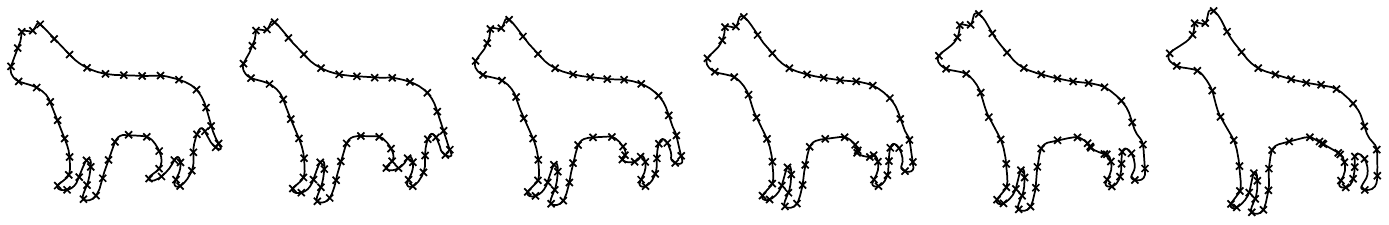}
\caption{A geodesic in the shape space $B_{i,f}(S^1,\R^2)$ equipped with the elastic metric that connects the cat--shaped figure to the dog--shaped figure. Original image published in \cite{Bauer2012b_preprint}.}
\label{fig:Rmap}
\end{figure}

A scale invariant version of the $\dot{H^1}$-metric
\[G_c(h,k)=\frac{1}{\ell_c}\int_{S^1}\langle D_sh, D_sk\rangle \ud s\]
has been studied in \cite{Michor2008a}. There the authors derive an explicit solution formula for the geodesic equation and calculate the sectional curvature. 
More general and higher order Sobolev metrics on plane curves have been studied in \cite{Mennucci2008,Michor2007}, and they have been applied to the field of active contours in 
\cite{Mennucci2007,Charpiat2005}. Other Sobolev type metrics on curves that have been studied include a metric for which  translations, scale changes and deformations of the curve are orthogonal \cite{Sundaramoorthi2011} and 
an $H^2$-type (semi)-metric whose kernel is generated by translations, scalings and rotations. \cite{Shah2010}.

For curves we can use arclength to identify each element $C \in B_{i,f}(S^1,\R^2)$ of shape space with a (up to rotation) unique parametrized curve $c \in \on{Imm}(S^1,\R^2)$. This observation has been used by Preston to induce a Riemannian metric on the shape space of unparametrized curves, via metrics on the space of arclength parametrized curves; see \cite{Preston2012,Preston2011}. A similar approach has been chosen in  \cite{Klassen2004}. 

\subsection{Geodesic distance}

Sobolev-type metrics induce a point-separating geodesic distance function on $B_{i,f}(M,\R^d)$, if the order of the operator field $L$ is high enough. For the $H^1$-metric
\begin{equation}
\label{h1_metric}
G^{H^1}_q(h,k) = \int_M \left\langle \left(\on{Id} + \De^g\right) h, k \right\rangle \on{vol}(q^\ast \ol g)\,,
\end{equation}
one can bound the length of a path by the area (volume) swept-out, similarly to the case of almost local metrics.

\begin{theorem}
If the metric $G^L$ induced by the operator field $L$ is at least as strong as the $H^1$-metric \eqref{h1_metric}, i.e.,
\[
G_q^L(h,h) \geq C G^{H^1}_q(h,h)
\]
for some constant $C>0$, then $G^L$ induces a point-separating geodesic distance function on the shape space $B_{i,f}(M,\R^d)$.
\end{theorem}

A proof can be found in \cite[Thm 7.6]{Bauer2011b}. An ingredient in the proof is the Lipschitz continuity of $\sqrt{\on{Vol}_q}$.

\begin{theorem}
The $H^1$-metric satisfies
\[
\left|\sqrt{\Vol_{Q_1}}-\sqrt{\Vol_{Q_2}}\right| \leq  \frac{1}{2} \on{dist}^{H^1}_{B_{i,f}}(Q_1,Q_2)\,.
\]
In particular the map
\[
\sqrt{\Vol}:(B_{i,f}(M,\R^d),\on{dist}^{H^1}_{B_{i,f}}) \to \mathbb R_{> 0}
\]
is Lipschitz continuous.
\end{theorem}

A proof for plane curves can be found in \cite[Sect. 4.7]{Michor2007} and for higher dimensional surfaces in \cite[Lem. 7.5]{Bauer2011b}.

The behavior of the geodesic distance on the space $\on{Imm}_f(M,\R^d)$ is unknown. Similar to Sect. \ref{al_dist} we can restrict the $G^L$-metric to an orbit $q \on{\circ} \on{Diff}(M)$ and the induced metric on $\on{Diff}(M)$ will be a right-invariant Sobolev metric. Since Sobolev-type metrics of a sufficiently high order on the diffeomorphism group have point-separating geodesic distance functions, there is no a-priori obstacle for the distance $\on{dist}^L_{\on{Imm}}$ not to be point-separating.

\begin{openquestion}
Under what conditions on the operator field $L$ does the metric $G^L$ induce a point-separating geodesic distance function on $\on{Imm}_f(M,\R^d)$?
\end{openquestion}

\subsection{The geodesic equation}

The most concise way to write the geodesic equation on $\on{Imm}(M,\R^d)$ for a general operator 
field $L$ involves its covariant derivative $\nabla L$ and adjoint $\on{Adj}(\nabla L)$. See 
\cite[Sect. 4.2]{Bauer2011b} for the definition of $\nabla L$; note that $\nabla$ here is not related to 
the metric $G^L$. For a general operator field $L$ on 
$\on{Imm}(M,\R^d)$ we define the adjoint $\on{Adj}(\nabla L)$ to be the adjoint of $(\nabla_k L) h$ 
in the $k$ variable with respect to the $L^2$-metric, i.e., for all $h, k, m \in T_q 
\on{Imm}(M,\R^d)$ we have     
\begin{multline}
\label{L_adj}
\int_M \langle (\nabla_k L) h, m\rangle \on{vol}(q^\ast \ol g) =  \\
= \int_M \langle k, \on{Adj}(\nabla L)(h, m)\rangle \on{vol}(q^\ast \ol g)\,.
\end{multline}
The existence and smoothness of the adjoint has to be checked for each metric by hand. This usually 
involves partial integration and even for simple operator fields like $L=\on{Id}+(\De^g)^l$ the 
expressions for the adjoint quickly become unwieldy.  

Assuming the adjoint in the above sense exists, we can write the geodesic equation in the following form in terms of the momentum.
\begin{theorem}[Thm. 6.5, \cite{Bauer2011b}]
Let $L$ be a smooth pseudo-differential operator field, that is invariant under reparametrizations, such that the adjoint $\on{Adj}(\nabla L)$ exists in the sense of \eqref{L_adj}. Then the geodesic equation for the $G^L$-metric on $\Imm(M,\R^d)$ is given by:
\begin{equation}
\label{sob_ge}
\begin{split}
p &= L_q q_t \otimes \on{vol}(q^\ast \ol g) \\
p_t &=  \frac12\Big(\on{Adj}(\nabla L)(q_t,q_t)^\bot- 2Tq.\langle L_q q_t,\nabla q_t\rangle^\sharp\\
&\qquad\qquad\qquad
-H_q \langle L_q q_t,q_t\rangle n_q \Big)\otimes\on{vol}(q^*\ol g)
\end{split}
\end{equation}
\end{theorem}
Note, that only the normal part of the adjoint
\[\on{Adj}(\nabla L)(q_t,q_t)^\bot=\langle \on{Adj}(\nabla L)(q_t,q_t),n_q\rangle n_q \]
appears in the geodesic equation. The 
tangential part is determined by the reparametrization invariance of the operator field $L$; see \cite[Lem. 6.2]{Bauer2011b}.

\begin{example}
Consider the simple operator field $L=D_s$ on the space $\on{Imm}(S^1,\R^2)$ of plane curves. To emphasize the nonlinear dependence of $L$ on the footpoint $c$ we write it as $L_c h = \tfrac{1}{|c_\theta|} h_\th$. The covariant derivative $\nabla L$ is simply the derivative of $L$ with respect to the footpoint,
\[
(\nabla_k L)  h = -\frac{1}{|c_\theta|^3} \langle k_\th, c_\theta \rangle h_\th = -\langle D_s k, v_c \rangle D_s h\,.
\]
for the operator field $L=D_s$. To compute its adjoint, we use the following identity, obtained by integration by parts,
\begin{multline*}
\int_{S^1} \langle D_s k, v_c \rangle \langle D_s h, m \rangle \ud s = \\
=- \int_{S^1} \langle k, \ka_c n_c \rangle \langle D_s h, m \rangle + \langle k, v_c \rangle D_s \langle D_s h, m \rangle \ud s\,,
\end{multline*}
which leads to
\[
\on{Adj}(\nabla L)(h, m) = \langle D_s h, m\rangle \ka_c n_c + D_s \left(\langle D_s h, m\rangle \right) v_c\,.
\]
The normal part $\on{Adj}(\nabla L)^\bot$, which is necessary for the geodesic equation is
\[
\on{Adj}(\nabla L)^\bot(h, m) = \langle D_s h, m\rangle \ka_c n_c\,.
\]
Note that while the full adjoint is a second order differential operator field, the normal part has only order one. This reduction in order will be important for the well-posedness of the geodesic equation.
\end{example}

To prove that geodesics on 
$B_{i,f}(M,\R^d)$ can be represented by horizontal geodesics on $\Imm_f(M,\R^d)$ we need the following lifting property.

\begin{lemma}[Lem. 6.8 and 6.9, \cite{Bauer2011b}]
Let $L$ be a smooth pseudo-differential operator field with order constant in $q$, that is 
invariant under reparametrizations, and such that for each $q$, the operator $L_q$ is elliptic, 
symmetric, and 
positive-definite. 
Then the decomposition 
\[
T\on{Imm}_f(M,\R^d) = \on{Hor}^L \oplus \on{Ver}
\]
of tangent vectors into horizontal and vertical parts is a smooth operation.

For any smooth path $q(t)$ in $\Imm_f(M,\R^d)$ there exists a
smooth path $\ph(t)$ in $\on{Diff}(M)$ depending smoothly on $q(t)$ such that the path $\wt q(t) = q(t) \on{\circ} \ph(t)$ is horizontal, i.e.,
\[
G^L_{\wt q(t)}(\p_t \wt q(t), T \wt q. X) = 0\,,\qquad \forall X\in \mf X(M)\,.\]

Thus any path in shape space can be lifted to a horizontal path of immersions.
\end{lemma}

\subsection{Well-posedness of the geodesic equation}
The well-posedness of the geodesic equation can be proven under rather general assumptions on the operator field.

{\bf Assumptions.} For each $q \in \on{Imm}(M,\R^d)$ the operator $L_q$ is an elliptic, pseudo-differential operator of order $2l$ and it is positive and symmetric with respect to the $L^2$-metric.

The operator field $L$, the covariant derivative $\nabla L$, and the normal part of the  adjoint 
$\on{Adj}(\nabla L)^\bot$ are all smooth sections of the corresponding bundles. For fixed $q$ the 
expressions 
\[
L_q h,\quad (\nabla_h L_q)k,\quad \on{Adj}(\nabla L)_q(h,k)^\bot
\]
 are pseudo-differential operators of order $2l$ in $h,k$ separately. As mappings in the footpoint 
 $q$ they can be a composition of non-linear differential operators and linear pseudo-differential 
 operators as long as the total order is less than $2l$.  

The operator field $L$ is reparametrization invariant in the sense of \eqref{L_rep_inv}.

With these assumptions we have the following theorem from \cite[Thm. 6.6]{Bauer2011b}. A similar theorem has been proven for plane curves in \cite[Thm 4.3]{Michor2007}.

\begin{theorem}
\label{thm:sob_wp}
Let the operator field $L$ satisfy the above assumptions with $l\geq 1$ and let $k > \tfrac {\on{dim}(M)}2 + 2l + 1$. Then the geodesic spray of the $G^L$-metric is smooth on the Sobolev manifold $\on{Imm}^k(M,\R^d)$ of $H^k$-immersions.

In particular the initial value problem for the geodesic equation \eqref{sob_ge} has unique solutions
\[
(t \mapsto q(t,\cdot)) \in C^\infty((-\ep,\ep), \on{Imm}^k(M,\R^d))\,,
\]
for small times and the solution depends smoothly on the initial conditions $q(0,\cdot)$, $q_t(0,\cdot)$ in $T\on{Imm}^k(M,\R^d)$.
\end{theorem}

\begin{remark}
For smooth initial conditions $q(0,\cdot)$, $q_t(0,\cdot)$ in $T\on{Imm}^k(M,\R^d)$ we can apply 
the above theorem for different $k$ and obtain solutions in each Sobolev completion 
$\on{Imm}^k(M,\R^d)$. It can be shown that the maximal interval of existence is independent of the 
Sobolev order $k$ and thus the solution of the geodesic equation itself is in fact smooth. 
Therefore the above theorem continues to hold, if $\on{Imm}^k(M,\R^d)$ is replaced by 
$\on{Imm}(M,\R^d)$.     
\end{remark}
\begin{remark}
Due to the correspondence of horizontal geodesics on $\Imm_f(M,\R^d)$ to geodesics on  shape 
space $B_{i,f}(M,\R^d)$ the above well-posedness theorem implies in particular the well-posedness 
of the geodesic problem on  $B_{i,f}(M,\R^d)$.   
\end{remark}

\begin{example}
The assumptions of this theorem might look very abstract at first. The simplest operator fulfilling them is
\[
L_q = \on{Id} + \De^g
\]
or any power of the Laplacian, $L_q = \on{Id} + (\De^g)^l$. We can also introduce non-constant coefficients, for example
\[
L_q = f_1(H_q, K_q) + f_2(H_q, K_q) (\De^g)^l \,,
\]
as long as the operator remains elliptic, symmetric and positive. To check symmetry and positivity 
it is sometimes easier to start with the metric. For example the expression 
\begin{multline*}
G_q(h,k) = \int_M g_1(\on{Vol}_q) \langle h, k \rangle + \\
+g_2(\on{Vol}_q)\sum_{i=1}^d g(\nabla^g h^i, \nabla^g k^i)\on{vol}(q^\ast \ol g)\,,
\end{multline*}
defines a metric and the corresponding operator $L_q$ will be symmetric and positive, provided 
$g_1$ and $g_2$ are positive functions. We can compute the operator via integration by parts, 
\[
(L_qh)^i = g_1(\on{Vol}_q) h^i - \on{div}^g \left( g_2(\on{Vol}_q) \nabla^g h^i \right)\,.
\]
For this operator field to satisfy the assumptions of Thm. \ref{thm:sob_wp}, if $g_2$ is the constant function, because $L_q h$ has order 2 in $h$, so it can depend at most on first derivatives of $q$.
\end{example}

\subsection{Conserved quantities}
If the operator field $L$ is invariant with respect to reparametrizations, the $G^L$-metric will be invariant under the action of $\Diff(M)$.
By Noether's theorem  the reparametrization momentum is constant along each geodesic, c.f. Sect.\ \ref{almost_conserved}. 
This means that for each $X \in \mf X(M)$ we have
\[
\int_M  \langle L_q q_t, Tq. X\rangle  \on{vol}(q^\ast \ol g) =  \textrm{const.}
\]
If $L$ is additionally invariant under the action of the Euclidean motion group $\R^d \rtimes \on{SO}(d)$ then so is the $G^L$-metric and  the linear and angular momenta are constant along geodesics. These are given by
\begin{align*}
\int_M( L_q q_t) \on{vol}(q^\ast \ol g) &\in \R^d \\
\int_M   q \wedge (L_q q_t) \on{vol}(q^\ast \ol g) &\in \bigwedge\nolimits^2 \R^d \cong \mf{so}(d)^\ast\,.
\end{align*}
If the operator field $L$ satisfies  the scaling property
\[L_{\la .q}=\la^{-\on{dim}(M) - 2}L_q,\quad q\in\Imm(M,\R^d), \la\in\R\,,\]
then the induced metric $G^L$ is scale invariant. In this case the scaling momentum is conserved along geodesics as well. It is given by:
\[
\int_M
 \langle q,L_qq_t\rangle\on{vol}(g) \in \R\,.
\]
See Sect.~\ref{almost_conserved} for a more detailed explanation of the meaning of these quantities.

\subsection{Completeness}

Concerning geodesic completeness it is possible to derive a result similar to Thm. \ref{al_incomplete}. The set of concentric spheres with a common center is again a totally geodesic submanifold and we can look for conditions, when it is possible to shrink spheres to a point with a geodesic of finite length.

\begin{theorem}[Lem. 9.5, \cite{Bauer2011b}]
If $L=\on{Id} + (\De^g)^l$ and $l < \tfrac{\on{dim}(M)}2 + 1$, then the spaces $\Imm(S^{d-1},\R^d)$ and $B_{i,f}(S^{d-1},\R^d)$ are not geodesically complete with respect to the $G^L$-metric.
\end{theorem}

For other choices of $M$ scalings will in general not be geodesic, but under the same condition an immersion can be scaled down to a point with finite energy. Under what conditions these spaces become geodesically complete is unknown. We do however suspect that similarly as Thm. \ref{diff_complete} for the diffeomorphism group, a differential operator field of high enough order will induce a geodesically complete metric.

The metric completion of  $B_{i,f}(S^1,\R^2)$ is known for the Sobolev metrics
\[
G_c^{H^j}(h, k) = \int_{S^1} \frac {1}{\ell_c}\langle h, k \rangle + \ell^{2j}_c\langle D_s^j h, D_s^j k \rangle \ud s\,,
\]
with $j=1,2$. For the metric of order 1 we have the following theorem.

\begin{theorem}[Thms. 26 and 27, \cite{Mennucci2008}]
The metric completion of $(B_{i,f}, \on{dist}^{H^1})$ is $B^{lip}_i(S^1,\R^2)$, the space 
of all rectifiable curves with the Fr\'echet topology.
\end{theorem}

See Sect. \ref{al_compl} for details about $B^{lip}_i(S^1,\R^2)$. There is a similar result for the metric of second order.

\begin{theorem}[Thm. 29, \cite{Mennucci2008}]
The completion of the metric space $(B_{i,f}, \on{dist}^{H^2})$ is the set of all those rectifiable curves that admit curvature $\ka_c$ as a measurable function and $\int_{S^1} \ka_c^2 \ud s < \infty$.
\end{theorem}

\subsection{Curvature}

Apart from some results on first and second order metrics on the space of plane curves, very little is known about the curvature of Sobolev-type metrics on either $\on{Imm}_f(M,\R^d)$ or $B_{i,f}(M,\R^d)$.

The family \eqref{elasticmetric1} of $H^1$-type metrics on the space $\on{Imm}([0,2\pi],\R^2)/{\on{trans}}$ of open curves modulo translations is isometric to an open subset of a vector space and therefore flat; see \cite{Bauer2012b_preprint}. It then follows from O'Neil's formula that the quotient space $B_{i,f}(M,\R^d)$ has non-negative sectional curvature.

The scale-invariant $H^1$-type semi-metric
\[
G_c(h,k) = \frac{1}{\ell_c} \int_{S^1}\langle D_s h, D_s k \rangle \ud s\,,
\]
descends to a weak metric on $B_{i,f}(S^1,\R^2)/{(\on{sim})}$, 
which is the quotient of $B_{i,f}(S^1,\R^2)$ by similarity transformations --- translations, 
rotations and scalings.
The sectional curvature has been computed explicitly in \cite{Michor2008a}; it is again 
non-negative and upper bounds of the following form can be derived. 

\begin{theorem}[Sect. 5.8, \cite{Michor2008a}]
Take a curve $c \in \on{Imm}(S^1,\R^2)$ and let $h_1,h_2 \in T_c\on{Imm}(S^1,\R^2)$ be two orthonormal tangent vectors. Then the sectional curvature at $C = \pi(c)$ of the plane spanned by the projections of $h_1, h_2$ in the space $B_{i,f}(S^1,\R^2)/{(\on{sim})}$ is bounded by
\begin{multline*}
0 \leq k_C(P(h_1,h_2)) \leq 2 + A_1(\ka_c) + \\
+A_2(\ka_c) \| D_s h_2 \cdot n \|_{\infty}
+ A_3(\ka_c) \| D_s (D_s h_2 \cdot n) \|_{\infty}\,.
\end{multline*}
where $A_i : C^\infty(S^1,\R) \to \R$ are functions of $\ka$ that are invariant under reparametrizations and similarity transformations.
\end{theorem}

Explicit formulas of $A_i(\ka)$ can be found in \cite{Michor2008a}. This is a bound on the 
sectional curvature, that depends on the first two derivatives of $h_2$ and is independent of 
$h_1$. Moreover, the explicit formulas for geodesics given in \cite{Michor2008a} show that conjugate points are not dense on geodesics.

A similar bound has been derived in \cite{Shah2010} for a second order metric on the space of plane curves.

\section{Diffeomorphism groups}
\label{diff_groups}

In the context of shape spaces  diffeomorphism groups arise two-fold: 
\begin{itemize}
 \item The shape space $B_{i,f}(M,\R^d)$ of immersed submanifolds is the quotient
\[
B_{i,f}(M,\R^d) = \on{Imm}_f(M,\R^d) / \on{Diff}(M)
\]
of the space of immersions by the reparametrization group $\on{Diff}(M)$.
\item By fixing an embedding $q_0 \in \on{Emb}(M,\R^d)$ we have the map
\[
\on{Diff}_c(\R^d) \to \on{Emb}(M,\R^d),\quad \ph \mapsto \ph \on{\circ} q_0
\]
A right-invariant Riemannian metric on $\on{Diff}_c(\R^d)$ induces a Riemannian metric on $\on{Emb}(M,\R^d)$ such that the above map is a Riemannian submersion. See Sect.\ \ref{sec_outer} for this construction.
\end{itemize}
These are the two main applications of the diffeomorphism group discussed in this paper. Thus we mainly will treat the group
\begin{align*}
&\on{Diff}(M) = \left\{ \ph\! \in \! C^\infty(M,M): \ph \text{ bij., }\!  \ph^{-1}\! \!  \in\!  C^\infty(M,M) \right\}
\end{align*}
of smooth diffeomorphisms of a closed manifold $M$ and groups of diffeomorphisms of $\R^d$ with the 
following decay conditions towards infinity 
\begin{align*}
\on{Diff}_c(\R^d) &= \left\{ \ph\,:\, \on{supp}(\ph - \on{Id}) \text{ compact} \right\} \\
\Diff_{\mc S}(\R^d)&=\left\{ \ph  \,:\, (\ph - \on{Id})\in \mc S(\R^d) \right\} \\
\Diff_{H^{\infty}}(\R^d)&=\left\{ \ph \,:\, (\ph - \on{Id})\in H^{\infty}(\R^d) \right\}\,.
\end{align*}
Here $H^{\infty}(\R^d)$ denotes the intersection of all Sobolev spaces $H^{k}(\R^d), \, 
k\in \mathbb N_{\geq 0}$, $\mc S(\R^d)$ denotes the Schwartz space of all rapidly decreasing 
functions. All these are smooth regular Lie groups. Their Lie algebras are the spaces $\mf X_c(\R^d), 
\mf X_{\mc S}(\R^d)$ and $\mf X_{H^\infty}(\R^d)$ of compactly supported, rapidly decreasing and 
Sobolev vector fields, respectively. See \cite{Michor2006a} and \cite{Michor2012b_preprint} for 
details.   

None of the diffeomorphism groups on $\R^d$ introduced above contain translations, rotations, or, 
more generally, affine maps, since they require the diffeomorphisms to decay towards the identity. 
It is possible to extend the groups by considering a semidirect product, for example 
$\on{Diff}_c(\R^d) \rtimes SO(d)$. But for our purposes this is not necessary: given two embedding 
$q_1, q_2 \in \on{Emb}(M,\R^d)$ differing by an affine map $q_2 = A \circ q_1$, since $M$ is 
compact, there exists a diffeomorphism $\ph$, decaying to the identity such that $q_2 = 
\ph \circ q_1$. Thus when considering the action of $\on{Diff}_c(\R^d)$ or the other diffeomorphism 
groups on $\on{Emb}(M,\R^d)$ in Sect. \ref{sec_outer}, we are not really losing affine maps, 
although they are not literally elements of the considered groups.       

On general non-compact manifolds $N$ one can also consider the group of compactly supported 
diffeomorphisms like on $\R^d$; see \cite[Sect.\ 43]{Michor1997}. 

\subsection{Right-invariant Riemannian metrics}\label{metrics_diff}

A right-invariant metric on $\on{Diff}_c(\R^d)$ is determined via
\begin{multline}
\label{eq_rinv_rm}
G_\ph(X_\ph,Y_\ph) 
= \langle X_\ph \on{\circ} \ph^{-1}, Y_\ph\on{\circ}\ph^{-1} \rangle_L\,,\\ X_\ph,Y_\ph \in T_\ph \on{Diff}_c(\R^d)\,,
\end{multline}
by an inner product $\langle\cdot,\cdot\rangle_L$ on the space $\mf X_c(\R^d)$ of vector fields. We assume that the inner product is defined via a symmetric, positive definite, pseudo-differential operator field $L: \mf X(\R^d) \to \mf X(\R^d)$ by
\[
\langle X, Y \rangle_L = \int_{\R^d} LX \cdot Y \ud x\,.
\]
Examples of such inner products include
\begin{itemize}
\item
The $L^2$-metric with $L=\on{Id}$,
\[
\langle X, Y \rangle_{L^2} = \int_{\R^d} \langle X, Y \rangle \ud x\,.
\]
\item
The Sobolev-type metrics of order $s$ with $s>0$,
\[
\langle X, Y \rangle_{H^s} = \int_{\R^d} \left( 1 + |\xi|^2\right)^s \langle \wh X(\xi), \wh Y(\xi) \rangle \ud \xi\,,
\]
with $\wh X(\xi) = (2\pi)^{-d/2} \int_{\R^d} e^{-i \langle x, \xi\rangle} X(x) \ud x$ being the Fourier transform.
Note that for $s \in \mathbb N$ these metrics can be written as
\[
\langle X, Y \rangle_{H^s} = \int_{\R^d} \langle \left(\on{Id} - \De\right)^s X, Y \rangle \ud x\,,
\]
i.e., $L=(\on{Id} - \De)^s$.
\item
The family of $a$-$b$-$c$-metrics, introduced in \cite{Khesin2011_preprint},
\begin{multline}
\label{abc_metric}
\langle X, Y \rangle_{a,b,c} = \int_{\mathbb \R^d} a \langle X, Y \rangle + b \on{div} X \on{div} Y + \\
+c \langle d X^\flat, dY^\flat \rangle \ud x\,.
\end{multline}
\end{itemize}
Recall that $\Delta=-\Delta^{\ol{g}}$ denotes the usual Laplacian on $\R^d$, which is the negative of the geometric Laplacian; see Sect.~\ref{notation}.
In dimension $d=1$ the second and the third term coincide and the metric simplifies to the family of $a$-$b$ metrics
\[
\langle X, Y \rangle_{a,b} = \int_{\mathbb \R} a X Y + b  X'  Y'\ud x\,.
\]

On manifolds other than $\R^d$, one can use the intrinsic differential operator fields to define inner 
products on $\mf X(M)$, which are then extended to right-invariant Riemannian metrics on 
$\on{Diff}(M)$ via \eqref{eq_rinv_rm}. For example, when $(M,g)$ is a Riemannian manifold  
Sobolev-metrics of integral order can be defined using the Laplacian $\De^g$ via   
\[
\langle X, Y \rangle_{H^k} = \int_M g\left( (\on{Id} + \De^g)^k X, Y \right) \on{vol}^g\,.
\]
Similarly the family of $a$-$b$-$c$ metrics have an intrinsic representation given by
(where $\de=-*d*$ is the codifferential)
\begin{multline*}
\langle X, Y \rangle_{a,b,c} = \int_M a g(X, Y) + b g(\de X^\flat, \de Y^\flat) + \\
+c g(dX^\flat, dY^\flat) \on{vol}^g\,.
\end{multline*}
More general Sobolev spaces $H^s(M)$ with $s \notin \mathbb N$ and the corresponding norms can be 
introduced using partitions of unity and Riemannian exponential coordinates. See the books 
\cite{Triebel1992} and \cite{Eichhorn2007} for the theory of function spaces, including Sobolev 
spaces of fractional order, on manifolds.   
\begin{remark}
An alternative approach to induce a metric on the diffeomorphism group is to use a reproducing kernel Hilbert space $\mathcal H$ of vector fields, with  $\X(\R^d)\subset \mathcal H$ and consider the restriction of the inner product on $\mathcal H$ to $\X(\R^d)$. This approach is described in Sect.\ \ref{metric_outer}.
\end{remark}

\subsection{Geodesic Equation}\label{EPDiff}

The geodesic equation on any Lie group $\mc G$ with a right-invariant metric is given as follows. A curve $g(t) \in \mc G$ is a geodesic if the right logarithmic derivative $u(t) = \p_t g(t) g(t)^{-1}$ satisfies
\[
\p_t u = -\on{ad}_u^T u\,,
\]
where $\on{ad}^T$ is the transpose of $\on{ad}$ with respect to the given inner product $\ga(\cdot,\cdot)$ on the Lie algebra, i.e., 
\[
\ga(\on{ad}_u^T v, w) = \ga(v, \on{ad}_u w)
\]

On $\on{Diff}(\R^d)$ with a metric given via an operator field $L$ we can write the equation as a PDE in terms of the momentum $m=Lu$,
\begin{equation}
\label{epdiff}
\p_t m + (u \cdot \nabla)m + m \on{div} u + Du^T.m = 0,\quad m = Lu\,,
\end{equation}
and $\p_t \ph = u \on{\circ} \ph$. For different choices of $L$ one can obtain the following PDEs as geodesic equations.

The $L^2$-metric with $Lu=u$ in one dimension has as geodesic equation {\it Burgers' equation},
\[
u_t + 3uu_x = 0\,.
\]
This equation was used as a model equation for turbulence in \cite{Burgers1948}.

The $H^1$-metric with $Lu = u - u_{xx}$ in one dimension has as geodesic equation the {\it Camassa-Holm equation} \cite{Holm1993},
\[
u_t - u_{xxt} + 3uu_x - 2u_x u_{xx} - uu_{xxx} = 0\,.
\]
It describes the propagation of shallow water waves on the free surface under gravity. It is a completely integrable equation and possesses a bihamiltonian structure, that gives rise to an infinite number of conservation laws. 

The homogeneous $\dot H^1$ semi-metric on $\Diff(S^1)$ with the operator $Lu = -u_{xx}$ descends to a metric on the right coset space $\on{Rot}(S^1)\backslash \on{Diff}(S^1)$ of diffeomorphisms modulo rigid rotations. The geodesic equation is the {\it periodic Hunter-Saxton equation}
\[
u_{xxt} + 2u_x u_{xx} + uu_{xxx} = 0\,.
\]
The Hunter-Saxton equation was proposed as a model for the propagation of orientation waves in nematic 
liquid crystals in \cite{Hunter1991}. Its geodesic nature was discovered in \cite{Misiolek2003}. It 
is also a completely integrable, bihamiltonian equation with an infinite number of conservation 
laws \cite{Hunter1994}. As a Riemannian manifold $(\on{Diff}(S^1), \dot H^1)$ is isometric to an 
open subset of a sphere and as such has positive constant curvature \cite{Lenells2007}. It was 
shown recently in \cite{Bauer2012d_preprint}, that a related result also
holds for the non-periodic Hunter-Saxton equation, which is the geodesic equation for the  
$\dot H^1$-metric on a certain extension of $\Diff_c(\R)$.

Between the Hunter-Saxton and the Camassa-Holm equation lies the {\it $\mu$-Hunter-Saxton equation},
\[
u_{xxt} - 2\mu(u) u_x + 2u_x u_{xx} + uu_{xxx} = 0\,,
\]
which is the geodesic equation on the circle with respect to the $\mu \dot H^1$-metric defined by the operator $Lu = \mu(u) - u_{xx}$, with $\mu(u) = \tfrac 1{2\pi} \int_{S^1} u \ud x$ being the mean. It was introduced in \cite{Lenells2008a} as a non degenerate metric on $\on{Diff}(S^1)$, such that the projection
\[
(\on{Diff}(S^1), \mu \dot H^1) \to (\on{Rot}(S^1)\backslash \on{Diff}(S^1), \dot H^1)
\]
is a Riemannian submersion. It is also a completely integrable, bihamiltonian equation.

The geodesic equation for the homogeneous $\dot H^{1/2}$-metric is the {\it modified Constantin-Lax-Majda (mCLM) equation},
\[
m_t + um_x + 2u_x m = 0,\quad m = \mc H u_x\,.
\]
The mCLM equation is part of a family of one dimensional models for the vorticity equation \cite{Constantin1985, DeGregorio1990, Okamoto2008}. Its geodesic nature was recognized in \cite{Wunsch2010a}. As for the Hunter-Saxton equation we have to regard the mCLM equation on the coset space $\on{Rot}(S^1)\backslash \on{Diff}(S^1)$.

In the context of hydrodynamics a closely related space is the Virasoro-Bott group 
\[\on{Vir}(S^1)= \Diff(S^1)\x_c \mathbb R\,,\] 
with the group operations
\begin{displaymath}
\binom{\ph}{\al}\binom{\ps}{\be}  
     =\binom{\ph\on{\circ}\ps}{\al+\be+c(\ph,\ps)},\quad 
\binom{\ph}{\al}^{-1}=\binom{\ph^{-1}}{-\al}
\end{displaymath}
for $\ph, \ps \in \on{Diff}(S^1)$, and $\al,\be \in \mathbb R$. 
The  Virasoro-Bott group is a central extension of
$\Diff(S^1)$ with respect to the Bott-cocycle:
\begin{align*} 
c:\on{Diff}(S^1)\x \on{Diff}(S^1)&\to \mathbb R
\\
c(\ph,\ps)& = \frac12\int\log(\ph'\on{\circ}\ps)d\log\ps'\,,
\end{align*}
and it is the unique non-trivial central extension of $\on{Diff}(S^1)$. For a detailed exposition of the Virasoro-Bott group see the book of Guieu and Roger \cite{Guieu2007}.
It was found in \cite{Segal1991, Ovsienko1987} that  the geodesic equation of
the right invariant $L^2$-metric on the Virasoro-Bott group  
is the \emph{Korteweg-de Vries equation}
\[u_t+3u_xu+au_{xxx}=0,\quad a\in \R \;.\] 
Similarly the \emph{Camassa Holm equation with dispersion} 
\[
u_t - u_{xxt} + 3uu_x - 2u_x u_{xx} - uu_{xxx} +2\ka u_x= 0
\]
was recognized to be the geodesic equation on the Vira\-soro-Bott group with respect to the $H^1$-metric in \cite{Misiolek1998}.
\begin{figure}
\begin{center}
\renewcommand{\arraystretch}{1.2}
\begin{tabular}{ |c | c |c|}
  \hline                        
Space & Metric & Geod. equation  \\
\hline
  $\Diff(S^1)$&$L^2$&Burgers \\
  $\Diff(S^1)$&$H^1$&Camassa-Holm \\
  $\Diff(S^1)$&$\mu \dot H^1$&$\mu$-Hunter-Saxton\\
  $\on{Rot}(S^1)\backslash\Diff(S^1)$&$\dot H^1$&Hunter-Saxton\\
  $\on{Rot}(S^1)\backslash\Diff(S^1)$&$\dot H^{1/2}$&mCLM\\
  $\on{Vir}(S^1)$&$L^2$&KdV\\
  $\on{Vir}(S^1)$&$H^1$&Camassa-Holm w. disp.\\
  \hline  
\end{tabular}
\caption{Some geodesic equations on diffeomorphism groups, that are relevant in mathematical physics.}
\end{center}
\end{figure}

The geodesic equation \eqref{epdiff} can be rewritten as an integral in Lagrangian coordinates. For 
a metric given by a differential operator, let $K(x,y)$ be its Green's function. We assume that the 
initial momentum $m_0$ is a vector-valued distribution, whose components are finite measures. The 
initial velocity can be obtained from $m_0$ via $u_0(x) = \int_{\R^d} K(x,.) m_0(.)$ and conversely 
$m_0 = Lu_0 \otimes dx$. Then \eqref{epdiff} can be written as
\begin{equation}
\begin{aligned}
\label{epdiff_lagrange}
\p_t \ph(t, x) = \int_{\R^d} K(\ph(t, x), .)\, \ph(t)_\ast m_0(.) \,.
\end{aligned}
\end{equation}

\subsection{Well-posedness of the geodesic equation}

One possible method to prove the well-posedness of the geodesic equations is to extend the group and the metric to the Sobolev-completion
\[
\on{Diff}^q(M) = \{ \ph \in H^q(M, M) : \ph \on{bij.},\, \ph^{-1}\! \in\! H^q(M,M) \}
\]
which is a Hilbert manifold and a topological group for $q>\on{dim}(M)/2+1$. It is however not a Lie group any more, since the right-multiplication is only continuous but not smooth. Nevertheless it is possible to show that the geodesic spray of various metrics on the Sobolev-completion is smooth for $q$ large enough and then an application of the theorem of Picard-Lindel\"of for ODEs shows the existence and smoothness of the exponential map. This method was first applied in \cite{Ebin1970} for the $L^2$-metric on the group of volume-preserving diffeomorphisms to show the existence of solutions for Euler's equations, which model inviscid, incompressible fluid flows. On the full diffeomorphism group the following well-posedness results can be obtained via the same method.
\begin{theorem}[Thm. 3.3, \cite{GayBalmaz2009}]\label{GayBalmaz}
Let $(M,g)$ be a compact Riemannian manifold without boundary. The geodesic spray of the $H^1$-metric 
\[\langle u, v \rangle = \int_M g(u,v) + g(\nabla u, \nabla v) \on{vol}^g\] is smooth as a map $T\on{Diff}^q(M) \to T^2 \on{Diff}^q(M)$ for $q > \tfrac{\on{dim}(M)}2 + 1$.

The (higher-dimensional) Camassa-Holm equation with initial condition $u_0 \in \mf X^q(M,M)$ admits 
a unique solution $u(t)$ for small times, the map $t \mapsto u(t)$ is in 
$
C^0((-\ep,\ep), \mf X^q(M)) \cap C^1((-\ep,\ep), \mf X^{q-1}(M))
$,
and the map $u_0 \in \mf X^q(M) \mapsto u(t) \in \mf X^q(M)$ is continuous.
\end{theorem}

This result holds more generally also for manifolds with boundary with either Dirichlet, Navier or mixed boundary conditions. See \cite{GayBalmaz2009} for more details. For the one-dimensional case the smoothness of the geodesic spray was noted already in \cite{Kouranbaeva1999}. 

For the circle $M=S^1$ we have the stronger result that the geodesic sprays for Sobolev metrics $H^s$ are smooth for $s \geq \tfrac 12$.

\begin{theorem}[Cor.~4.2, \cite{Escher2012_preprint}]
The geodesic spray of the $H^s$-metric
\[
\langle u, v \rangle = \sum_{n\in \mathbb Z} (1+n^2)^s \wh u(n) \ol {\wh v(n)}
\]
on the diffeomorphism group $\on{Diff}^q(S^1)$ of the circle is smooth for $s \geq \tfrac 12$ and $q > 2s + \tfrac 32 $. 
Here $\wh u$ denotes the Fourier series of $u$.
Thus the geodesic equation is, similarly to Thm. \ref{GayBalmaz}, locally well-posed.
\end{theorem}

The case of Sobolev metrics of integer order, which includes the periodic Camassa-Holm equation, was proven in \cite{Constantin2003}. For the homogeneous $\dot H^{1/2}$-metric this result was proven in \cite{Escher2010} and the estimates were then extended to cover general metrics given via Fourier multipliers in \cite{Escher2012_preprint}.

As a consequence of the well-posedness result for Sobolev metrics on $\Imm(M,\R^d)$ it has been shown that the Lagrangian form of the geodesic equation is locally well posed for higher order Sobolev metrics on $\Diff(M)$.
\begin{theorem}[Thm. 10, \cite{Bauer2011b}]
Let $(M,g)$ be a compact Riemannian manifold without boundary and let $k\geq 1$ with $k \in \mathbb N$. For $q > \tfrac{\on{dim}(M)}2 +2k+1$ the geodesic spray of the $H^k$-metric 
is smooth as a map  on $\on{Diff}^q(M)$ and the geodesic equation has unique local solutions on $\Diff^q(M)$.
\end{theorem}

If the metric is strong enough, it is possible to show the long-time existence of solutions. 
\begin{theorem}\label{diff_globalwell}
If the Green's function $K$ of the operator $L$ inducing the metric is a $C^1$-function, then for 
any vector-valued distribution $m_0$, whose components are finite signed measures,  equation 
\eqref{epdiff_lagrange} with $\ph(0,x)=x$ can be solved for all time and the solution is a map   
\[
(t \mapsto \ph(t,\cdot)) \in C^1(\R, C^1(\R^d,\R^d))\,.
\]
\end{theorem}

This result is implicit in the work \cite{Trouve2005}, an explicit proof can be found in \cite{Michor2012d_preprint}. See also \cite{Younes2010}.

\begin{remark}
This method of proving well-posedness is not universally applicable as not all geodesic sprays are 
smooth. For example the spray induced by the right-invariant $L^2$-metric on $\on{Diff}(S^1)$ is 
not smooth. More precisely in \cite{Constantin2002} it is shown that the exponential map is not a 
$C^1$-map from a neighborhood of $T_{\on{Id}}\on{Diff}^q(S^1)$ to $\on{Diff}^q(S^1)$ for any 
$q\geq 2$. Nevertheless the geodesic equation, which is Burgers' equation in this case, has 
solutions $t \mapsto u(t)$ for small time with    
\[
u \in C^0((-\ep,\ep), H^q(S^1)) \cap C^1((-\ep,\ep), H^{q-1}(S^1)),
\]
when $u_0 \in H^q(S^1)$; see \cite{Kato1975}. A similar statement holds for the KdV-equation, which 
is the geodesic equation on the Virasoro-Bott group with respect to the right-invariant $L^2$-metric; see \cite{Constantin2007}.
\end{remark}

\subsection{Geodesic distance}\label{diff_dist}

It was shown in \cite{Michor2005} that the geodesic distance on the group $\on{Diff}_c(M)$ vanishes for the $L^2$-metric and is positive for the $H^1$-metric. This naturally raises the question, what happens for the $H^s$-metric with $0<s<1$. For $M=S^1$ a complete answer is provided in \cite{Bauer2013b}, whereas for more general manifolds $N$ a partial answer was given in the articles \cite{Bauer2013b,Bauer2012c_preprint}.

\begin{theorem}[Thm.~3.1, \cite{Bauer2013b,Bauer2012c_preprint}]
The geodesic distance on $\on{Diff}_c(M)$ induced by the Sobolev-type metric of order $s$ vanishes
\begin{itemize}
\item
for $s<\tfrac 12$,
\item
for $s=\tfrac 12$, when $M=S^1\x C$ with $C$ compact.
\end{itemize}
The geodesic distance is positive
\begin{itemize}
\item
for $s\geq 1$,
\item
for $s>\tfrac 12$, when $\on{dim}(M)=1$.
\end{itemize}
\end{theorem}

\begin{remark}
By taking $C=\{ \text{point} \}$ to be the zero dimensional manifold, the above theorem provides a complete answer for $M=S^1$: the geodesic distance vanishes if and only if $s\leq\tfrac 12$.
\end{remark}

\begin{remark}
The $H^{1/2}$-metric on $\on{Diff}(S^1)$ is the only known example, where the  geodesic spray is smooth on the Sobolev-completions $\Diff^q(S^1)$ for all $q\geq \tfrac 52$ 
and the geodesic distance vanishes at the same time.

It is shown in \cite{Escher2012_preprint} that for $q > \tfrac 52$ the exponential map is a local diffeomorphism $\on{exp} : U \subseteq H^q(\R) \to \on{Diff}^q(S^1)$.
In particular we have the inequality 
\[
\on{Len}^{H^{1/2}}(\ph) \geq \| \on{exp}^{-1}(\ph(1)) \|_{H^{1/2}}
\]
for all paths $\ph : [0,1] \to \on{exp}(B_\ep^q(0))$ with $\ph(0) = \on{Id}$. In other words we have a lower bound on the length for all paths, that remain $H^q$-close to $\on{Id}$. This does not however imply anything about the geodesic distance, since a path can have small $H^{1/2}$-length or equivalently remain $H^{1/2}$-close to $\on{Id}$, but leave the $H^q$-neighborhood.
\end{remark}

\begin{openquestion}
For a Sobolev metric of order $s$ the behavior of  the geodesic distance on $\Diff_c(M)$ remains open for
\begin{itemize}
 \item $\tfrac 12<s<1$ and  $\on{dim}(N)\geq2$.
 \item $s = \tfrac 12$ and $N\neq S^1\times M$, with $M$ compact.
\end{itemize}
Extrapolating from the known cases, we conjecture the following result:
\emph{The geodesic distance induced by the Sobolev-type metric of order $s$ on $\Diff_c(N)$ vanishes for $s\leq \tfrac 12$ and is non-degenerate for $s > \tfrac 12$.}
\end{openquestion}
A main ingredient for the vanishing result is the following property of  the geodesic distance on $\on{Diff}_c(N)$: 
\begin{lemma}
Let $s\geq 0$.
 If the geodesic distance  on $\Diff_c(N)$ for a right-invariant Sobolev $H^s$-metric vanishes for one pair $\ph, \ps \in \on{Diff}_c(N)$ with $\ph \neq \ps$, 
 then the geodesic distance already vanishes identically on all of $\Diff_c(N)$.
 \end{lemma}
This lemma follows from the fact that the set
\[
A = \left\{ \ph \,:\, \on{dist}^{H^s}(\on{Id}, \ph)=0 \right\}
\]
is a normal subgroup of $\on{Diff}_c(N)$ for all $s\geq 0$ and because $\on{Diff}_c(N)$ is a simple group. 
Thus, if $A$ contains any element apart from $\on{Id}$ it has to be the whole group. 

\begin{remark}
We can also consider the geodesic distance on the Virasoro-Bott group, which is the one-dimensional 
central extension of $\on{Diff}(S^1)$. There the geodesic distance vanishes for $s=0$, i.e., for 
the $L^2$-metric. For $s>\tfrac 12$ the geodesic distance cannot vanish identically. Whether it is 
point-separating is not known.   
\end{remark}
\begin{openquestion}
For a Sobolev metric of order $s$ the behavior of  the geodesic distance on the Virasoro-Bott group  remains open for $0<s<1$.
\end{openquestion}

One way to define geodesics is to fix two diffeomorphisms $\ph_0, \ph_1$ and to consider the set 
\[
B=\{\ph(t) \,:\, \ph(0)=\ph_0,\, \ph(1)=\ph_1\}
\]
of all paths joining them. Geodesics then correspond to critical points of the energy or equivalently the length functional restricted to the set $B$. Vanishing of the geodesic distance implies that these functionals have no global minima. The following theorem shows that for the $L^2$-metric there are no local minima either.

\begin{theorem}[Thm 3.1, \cite{Bruveris2013}]
Let $\varphi(t,x)$ with $t \in [0, T]$ be a path in $\operatorname{Diff}_{c}(\mathbb R)$. Let $U$ be a neighborhood of $\varphi$ in the space $C^{\infty}_c([0,T] \times \mathbb R)$. Then there exists a path $\psi \in U$ with the same endpoints as $\varphi$ and
\[ E(\psi) < E(\varphi), \]
where $E(.)$ is the energy w.r.t. the right-invariant $L^2$-metric.
\end{theorem}
In the article \cite{Bruveris2013} the result is proven for $\operatorname{Diff}_{\mathcal S}(\mathbb R)$, but essentially the same proof works also for $\operatorname{Diff}_{c}(\mathbb R)$.

\subsection{Completeness}\label{diff_complete}
As a corollary of Thm.~\ref{diff_globalwell} we obtain the result that the diffeomorphism group equipped with a metric of high enough order is geodesically complete:
\begin{theorem}
Let $(M,g)$ be a compact Riemannian manifold and let $G^s$ be the Sobolev metric of order $s$. 
For $s\geq\frac{\dim(M)+3}{2}$ the space $\big(\Diff(M),G^s\big)$ is geodesically complete.
\end{theorem}
This result is based on the observation, that for $s\geq\frac{\dim(M)+3}{2}$ the kernel of the operator inducing the metric $G^s$ is a $C^1$-function.

\subsection{Curvature}

Denote by $\ga(\cdot,\cdot)$ the inner product on the Lie algebra $\mf g$ of any Lie group $\mc G$ and let $u,v \in \mf g$ be orthonormal vectors. Then the sectional curvature of the plane $P(u,v)$ in $\mc G$ with respect to the right-invariant metric induced by $\ga$ is given by
\begin{align*}
 k(P(u,v))&=\frac14 \|\on{ad}^T_v u+ \on{ad}^T_u v\|^2_\ga-\ga(\on{ad}^T_v u,\on{ad}^T_u v)\\&\quad-\frac34\| \on{ad}_u v\|^2_\ga
+\frac12 \ga(\on{ad}_u v,\on{ad}^T_v u- \on{ad}^T_u v),
\end{align*}
where $\on{ad}^T$ is the transpose of $\on{ad}$ with respect to the given inner product $\ga$ inducing the right invariant metric.

For general Sobolev metrics there are no results on curvature available, but for the family of  $a$-$b$-$c$-metrics \eqref{abc_metric} on the $d$-dimensional torus $\mathbb T^d$, it was shown in \cite{Khesin2011_preprint} that the curvature assumes both signs.
\begin{theorem}[Thm. 7.1, \cite{Khesin2011_preprint}]
\label{thm_curv_pm}
If $d\geq 2$ and at least two of the parameters $a$, $b$, $c$ are non-zero, then the sectional curvature of the $a$-$b$-$c$-metric on $\on{Diff}(\mathbb T^d)$ assumes both signs.
\end{theorem}

In dimension one we have the same behavior for the family of $a$-$b$ metrics.

\begin{theorem}[Sect.\ 6, \cite{Khesin2011_preprint}]
If $d=1$ and both parameters $a$, $b$ are non-zero, then the sectional curvature of the $a$-$b$-metric on  $\on{Diff}(S^1)$ assumes both signs.
\end{theorem}

There are two special cases, where the sign of the curvature is constant. The first is the $L^2$-metric ($b=0$) in one dimension.

\begin{theorem}[Sect.\ 5.4, \cite{Michor2005}]
If $d=1$ and $b=0$ then the sectional curvature of the plane spanned by two orthonormal vector fields $u,v \in \mf X(S^1)$ for the $a$-$b$ metric on $\Diff(S^1)$ is given by
\begin{align*}
 k(P(u,v))&=\int_{S^1}(uv'-vu')^2\ud x\,.
\end{align*}
In particular the sectional curvature is non-negative.
\end{theorem}

This does not generalize to higher dimensions. Denote by $\mathbb T^d$ the flat $d$-dimensional torus. 

\begin{theorem}[Prop. 7.2, \cite{Khesin2011_preprint}]
\label{curv_d2_l2}
If $d\geq 2$ and $b=c=0$ then then the sectional curvature of the $a$-$b$-$c$ metric on $\on{Diff}(\mathbb T^d)$ assumes both signs.
\end{theorem}

The sectional curvature of the $L^2$-metric has been calculated for an arbitrary Riemannian 
manifold $N$. The expression for sectional curvature is the sum of a non-negative term and a term 
whose sign is indefinite. Although we conjecture that the statement of Thm. \ref{curv_d2_l2} 
extends to arbitrary manifolds $N$, this has not been proven yet.   

The second special case is the homogeneous $\dot H^1$-metric with $a=c=0$ for $d\geq 2$ and $a=0$ for $d=1$. The metric is degenerate on $\on{Diff}(M)$, but it induces the Fisher-Rao metric on the space $\on{Diff}(M)/\on{Diff}_\mu(M)$ of densities. Remarkably the induced metric has constant sectional curvature.
\begin{theorem}[Cor. 3.2, \cite{Khesin2011_preprint}]
 Let $(M,g)$ be a compact Riemannian manifold. Then the homogeneous $\dot H^1$-metric
 \[\langle u,v\rangle_{\dot H^1}=\int_M\on{div}(u)\on{div}(v)\on{vol}(g)  \]
 on $\on{Diff}(M)/\on{Diff}_{\mu}(M)$ has constant positive sectional curvature
\[
k(P(u,v)) =  \frac{1}{\Vol(M)}\,.
\]
\end{theorem}

This result is based on the observation, that the $\dot H^1$ metric on $\on{Diff}(M)/\on{Diff}_{\mu}(M)$ is  
isometric to a sphere in the Hilbert space $L^2(M,\on{vol}(g))$.  
For $M=S^1$ this result has been proven already in  \cite{Lenells2007}. Recently it has been shown 
that the $\dot H^1$-metric on a certain extension of $\Diff_c(\R)$ is a flat space in the sense of Riemannian geometry; see \cite{Bauer2012d_preprint}.

\section{Metrics on shape space induced by $\on{Diff}(\R^d)$}
\label{sec_outer}

In this section we will consider Riemannian metrics on $B_e(M,\R^d)$, 
the space of embedded type $M$ submanifolds that are induced by the left action of $\on{Diff}(\R^d)$. 
Let $\on{Diff}(\R^d)$ stand for one of the groups $\on{Diff}_c(\R^d)$, 
$\on{Diff}_{\mc S}(\R^d)$ or $\on{Diff}_{H^\infty}(\R^d)$ described in Sect.~\ref{diff_groups}. 
We also relax the assumption on the dimension of $M$ and only require $\on{dim}(M) < d$. The action is given by
\[
\on{Diff}(\R^d) \x B_e(M,\R^d) \ni (\ph, Q) \mapsto \ph(Q) \in B_e(M,\R^d)\,.
\]
This action is in general not transitive -- consider for example a knotted and an unknotted circle in $\R^3$ -- but its orbits are open subsets of $B_e(M,\R^d)$. Since the groups $\on{Diff}_c(\R^d)$, $\on{Diff}_{\mc S}(\R^d)$ and $\on{Diff}_{H^\infty}(\R^d)$ connected and $M$ is compact, the orbits are the connected components of $B_e(M,\R^d)$. For $Q \in B_e(M,\R^d)$ the isotropy group
\[
\on{Diff}(\R^d)_Q = \left\{ \ph\,:\, \ph(Q) = Q \right\}\,,
\]
consists of all diffeomorphisms that map $Q$ to itself. Thus each orbit $\on{Orb}(Q) = \on{Diff}(\R^d).Q$ can be identified with the quotient
\[
B_e(M,\R^d) \supseteq \on{Orb}(Q) \cong {\on{Diff}(\R^d)}/{ \on{Diff}(\R^d)_Q}\,.
\]
Let us take a step backwards and remember that another way to represent $B_e(M,\R^d)$  was as the quotient
\[
B_e(M,\R^d) \cong \on{Emb}(M,\R^d) / \on{Diff}(M)\,.
\]
The diffeomorphism group $\on{Diff}(\R^d)$ also acts on the space $\on{Emb}(M,\R^d)$ of embeddings, that is parametrized submanifolds with the action
\[
\on{Diff}(\R^d) \x \on{Emb}(M,\R^d)\!\ni\! (\ph, q) \mapsto \ph \on{\circ} q \in \on{Emb}(M,\R^d).
\]
This action is generally not transitive either, but has open orbits as before. For fixed $q \in \on{Emb}(M,\R^d)$,  the isotropy group
\[
\on{Diff}(\R^d)_q = \left\{ \ph\,:\, \ph|_{q(M)} \equiv \on{Id} \right\}\,,
\]
consists of all diffeomorphisms that fix the image $q(M)$ pointwise. Note the subtle difference between the two groups $\on{Diff}(\R^d)_q$ and $\on{Diff}(\R^d)_Q$, when $Q = q(M)$. The former consists of diffeomorphisms that fix $q(M)$ pointwise, while elements of the latter only fix $q(M)$ as a set. As before we can identify each orbit $\on{Orb}(q) = \on{Diff}(\R^d).q$ with the set
\[
\on{Emb}(M,\R^d) \supseteq \on{Orb}(q) \cong {\on{Diff}(\R^d)}/{ \on{Diff}(\R^d)_q}\,.
\]
The isotropy groups are subgroups of each other
\[
\on{Diff}(\R^d)_q \unlhd \on{Diff}(\R^d)_Q \leq \on{Diff}(\R^d)\,,
\]
with $\on{Diff}(\R^d)_q$ being a normal subgroup of $\on{Diff}(\R^d)_Q$. Their quotient can be identified with
\[
\on{Diff}(\R^d)_Q / \on{Diff}(\R^d)_q \cong \on{Diff}(M)\,.
\]
Now we have the two-step process,
\begin{multline*}
\on{Diff}(\R^d) \to \\
\to \on{Diff}(\R^d) / \on{Diff}(\R^d)_q \cong \on{Orb}(q) \subseteq \on{Emb}(M,\R^d) \to \\
 \to \on{Emb}(M,\R^d) / \on{Diff}(M) \cong B_e(M,\R^d)\,.
\end{multline*}
In particular the open subset $\on{Orb}(Q)$ of $B_e(M,\R^d)$ can be represented as any of the quotients
\begin{multline*}
\on{Orb}(Q) \cong \on{Orb}(q) / \on{Diff}(M) \cong \\
\cong \left.{}^{\displaystyle {\on{Diff}(\R^d)}/{ \on{Diff}(\R^d)_q} }
\middle/_{\displaystyle \on{Diff}(\R^d)_Q/\on{Diff}(\R^d)_q}\right. \cong \\
\cong \on{Diff}(\R^d)/{ \on{Diff}(\R^d)_Q}\,.
\end{multline*}

Let a right-invariant Riemannian metric $G^{\on{Diff}}$ be given on $\on{Diff}(\R^d)$. Then we can attempt to define a metric on $\on{Emb}(M,\R^d)$ in the following way: fix $q_0 \in \on{Emb}(M,\R^d)$ and let $q = \ph \on{\circ} q_0$ be an element in the orbit of $q_0$. Define the (semi-)norm of a tangent vector $h \in T_q \on{Emb}(M,\R^d)$ by
\[
G^{\on{Emb}}_q(h,h) = \inf_{X_\ph \on{\circ} q_0 = h} G^{\on{Diff}}_\ph(X_\ph, X_\ph)\,,
\]
with $X_\ph \in T_\ph\on{Diff}(\R^d)$. If we define $\pi_{q_0}$ to be the projection
\[
\pi_{q_0} : \on{Diff}(\R^d) \to \on{Emb}(M,\R^d),\quad \pi_{q_0}(\ph) = \ph \on{\circ} q_0\,,
\]
then
\[
h = X_\ph \on{\circ} q_0 = T_\ph \pi_{q_0}.X_\ph\,,
\]
and the equation defining $G^{\on{Emb}}$ is the relation between two metrics that are connected by a Riemannian submersion. Because $G^{\on{Diff}}$ is right-invariant and the group action is associative we can rewrite the defining equation as
\[
G_q^{\on{Emb}}(h, h) = \inf_{X \on{\circ} q = h} G^{\on{Diff}}_{\on{Id}}(X, X)\,,
\]
with $X \in T_{\on{Id}}\on{Diff}(\R^d)$. Thus we see that $G^{\on{Emb}}$ does not depend on the choice of $q_0$. 

One has to prove in each example, that $G^{\on{Emb}}$ is smooth and a metric, i.e., that it is non-degenerate. We will see for landmark matching in Sect. \ref{sec_lm}, that even though the metric $G^{\on{Diff}}$ on $\on{Diff}(\R^d)$ is smooth, the induced metric on the landmark space $\mathcal{L}^n(\R^d)$ has only finitely many derivatives. 

While $\pi_{q_0}$ is a Riemannian submersion this is an example, where the horizontal bundle exists only in a suitable Sobolev-completion; see Sect.~\ref{Submersion}. 
In Sect.~\ref{metric_outer} we will take care of this by defining the metric via a reproducing 
kernel Hilbert space $\mc H$. 

Assuming that this construction yields a Riemannian metric on the space $\on{Emb}(M,\R^d)$, then this metric is invariant 
under reparametrizations, because
the left-action by $\on{Diff}(\R^d)$ commutes with the right-action by $\on{Diff}(M)$:
\begin{multline*}
G^{\on{Emb}}_{q\on{\circ}\ph}(h\on{\circ}\ph,h\on{\circ}\ph) = \inf_{X\on{\circ} q\on{\circ}\ph = h\on{\circ}\ph} G^{\on{Diff}}_{\on{Id}}(X, X) =\\
= \inf_{X\on{\circ} q = h} G^{\on{Diff}}_{\on{Id}}(X, X) = G^{\on{Emb}}_{q}(h,h)\,.
\end{multline*}
The metric $G^{\on{Emb}}$ then projects to a Riemannian metric on $B_e(M,\R^d)$ as explained in Sect.~\ref{Submersion}.

\subsection{Pattern theory}

This section is closely related to ideas in Grenander's pattern theory \cite{Grenander1993, Grenander2007, Mumford2010}. 
The principle underlying pattern theory is to explain changes of shape by a deformation group 
acting on the shape. 
In our case shapes are elements of either $\on{Emb}(M,\R^d)$ or $B_e(M,\R^d)$ 
and the deformation group is the group $\on{Diff}(\R^d)$.

There is a lot of flexibility in the choice of the group and the space it acts upon. 
If $M$ is a finite set of $n$ points, then $\Emb(M,\R^d)\subseteq (\R^d)^n$ is the set of landmarks. 
We have inclusion instead of equality because landmarks have to be distinct points. 
We will return to this space in Sect. \ref{sec_lm}.

An important example is when the shape space is the space of volumetric grey-scale images modeled as functions in $C^\infty(\R^d,\R)$ and the deformation group is $\on{Diff}(\R^d)$. The action is given by
\[
\on{Diff}(\R^d) \x C^\infty(\R^d,\R)\!\ni\! (\ph, I) \mapsto I \on{\circ} \ph^{-1} \in C^\infty(\R^d,\R).
\]
This action is far from being transitive. 
Thus it is not possible to rigorously define a Riemannian metric on $C^\infty(\R^d,\R)$ 
that is induced by $\on{Diff}(\R^d)$. 
Nevertheless the idea of images being deformed by diffeomorphisms led to the image registration method 
known as LDDMM \cite{Beg2005, Trouve2002, Miller2001, Trouve1998}. 
It is being applied in computational anatomy with images being MRI and CT scans to study the connections between anatomical shape and physiological function. See \cite{Bruveris2013_preprint} for an overview of image registration within the LDDMM framework.

\subsection{Defining metrics on $\on{Diff}(\R^d)$}\label{metric_outer}
Following the presentation in \cite{Micheli2013} we assume that the inner product on 
$\mf X_c(\R^d)$ is given in the following way: let $(\mc H, \langle\cdot,\cdot\rangle_\mc H)$ be a 
Hilbert space of vector fields, such that the canonical inclusions in the following diagram  
\[
\mf X_c(\R^d) \hookrightarrow \mc H \hookrightarrow C^k_b(\R^d, \R^d)
\]
are bounded linear mappings for some $k\geq 0$. We shall also assume that the Lie algebra $\mf X_c(\R^d)$ 
of $\Diff(\mathbb R^d)$
is dense in $\mc H$. 
Here $C^k_b(\R^d, \R^d)$ is the space of all globally bounded $C^k$-vector fields 
with the norm $\| X\|_{k,\infty} = \sum_{0\leq j \leq k} \| D^j X\|_{\infty}$.

Given these assumptions, the space $\mc H$ is a {\it reproducing kernel Hilbert space}, i.e., for all 
$x,a\in \R^d$ the directional point-evaluation $\on{ev}_x^a : \mc H \to \R$ defined as 
$\on{ev}_x^a(u) = \langle u(x), a \rangle$ is a continuous linear functional on $\mc H$. See 
\cite{Aronszajn1950} or \cite{Saitoh1988} for a detailed treatment.
The relation
\[
\langle u, K(.,x)a\rangle_{\mc H} = \langle u(x), a \rangle
\]
defines a matrix-valued function $K:\R^d \x \R^d \to \R^{d\x d}$, called the {\it kernel} of 
$\mc H$. It satisfies the two properties
\begin{itemize}
\item
$K \in C^k_b(\R^d \x \R^d,\R^{d\x d})$ and
\item
$K(y,x) = K(x,y)^T$.
\end{itemize}

Associated to $\mc H$ we have the canonical isomorphism $L : \mc H \to \mc H^\ast$. Note that the 
kernel satisfies $K(y,x)a = L^{-1}(\on{ev}^a_x)(y)$; this relation is even more general: 
the space $\mf M^k(\R^d,\R^d)$ of vector-valued distributions, 
whose components are $k$-th derivatives of finite signed measures is a subspace of the dual space $C^k_b(\R^d,\R^d)^\ast$ and the operator
\[
K : \mf M^k(\R^d,\R^d) \to C^k_b(\R^d, \R^d),\; m \mapsto \int_{\R^d} K(.,x) m(x)
\]
coincides with $L^{-1}$. This is represented in the diagram
\[
\xymatrix @-2.6pt{
\mf X_c(\R^d)^\ast & \mc H^\ast \ar[l] & \mf M^k(\R^d) \ar[l] \ar[d]^K \ar[r]^-\subseteq & C^k_b(\R^d,\R^d)^\ast \\
\mf X_c(\R^d) \ar[r] & \mc H \ar[r] \ar[u]^L & C^k_b(\R^d,\R^d)
}
\]
Here $\mf X_c(\R^d)^\ast$ denotes the space of vector-valued distributions dual to $\X(\mathbb R^d)$, 
depending on the decay conditions chosen. 
The inner product on $\mf X_c(\R^d)$ is the restriction of the inner product on $\mc H$,
\[
\langle X, Y \rangle_{\mc H} = \int_{\R^d} \langle LX, Y \rangle \ud x\,,
\]
where the expression on the right hand side is a suggestive way to denote the pairing $\langle LX, Y \rangle_{\mf X_c^\ast \x \mf X_c}$ between a distribution and a vector field.

\begin{example} Let $\mc H = H^k(\R^d,\R^d)$ be the Sobolev space of order $k>\tfrac d2$ with the inner product
\[
\langle X, Y \rangle_{H^k} = \int_{\R^d} \langle (\on{Id} - \De)^k X, Y \rangle \ud x\,.
\]
Then by the Sobolev embedding theorem we have
\[
H^k(\R^d,\R^d) \hookrightarrow C^{l}_b(\R^d,\R^d)\quad \text{  for }l<k-d/2.
\]
In this example $L : H^k(\R^d,\R^d) \to H^{-k}(\R^d,\R^d)$ is the operator $L = (\on{Id} - \De)^k$ and the kernel $K$ is the Green's function of $L$,
\[
K(x,y) = (2\pi)^{-d/2} \frac{2^{1-k}}{(k-1)!} |x-y|^{k-\tfrac d2} J_{k-\tfrac d2}(|x-y|) \on{Id}\,.
\]
$J_\al(x)$ is the modified Bessel function of order $\al$. Around $x=0$ the Bessel function behaves like $J_\al(|x|) \sim |x|^\al$ and so 
\[
|x-y|^{k-\tfrac d2} J_{k-\tfrac d2}(|x-y|) \sim |x-y|^{2k-d}\,\, \textrm{around } x-y=0\,.
\]
Thus $K \in C_b^{2k-d-1}(\R^d\x\R^d,\R^{d\x d})$; 
this will be relevant in the case of landmarks.
\end{example}

In the above example $L$ was a scalar differential operator;  
it acted on each component of the vector field equally and was a multiple of the 
identity matrix. 
This is not always the case. 
For example the operator associated to the family of $a$-$b$-$c$-metrics is in general not scalar 
and the corresponding kernel is a dense (not sparse) matrix.

In Sect.\ \ref{diff_groups} the metric on $\on{Diff}(\R^d)$ was introduced by choosing a 
differential operator. Given an operator with appropriate properties, it is possible to reconstruct 
the space $\mc H$. The reason for emphasizing the space $\mc H$ and the reproducing kernel is 
twofold: Firstly, the induced metrics on $\on{Emb}(M,\R^d)$ and the space of landmarks have a simpler 
representation in terms of the kernel. Secondly, in the literature on LDDMM (e.g., in 
\cite{Younes2010}) the starting point is the space $\mc H$ of vector fields and by presenting both 
approaches we show their similarities.

\subsection{The metric on $\on{Emb}(M,\R^d)$}

Let $G^{\mc H}$ be a right-invariant metric on $\on{Diff}(\R^d)$. The induced metric on $\on{Emb}(M,\R^d)$ is defined via
\[
G_q^{\mc H}(h, h) = \inf_{X \on{\circ} q = h} \langle X, X \rangle_{\mc H}\,.
\]
To compute a more explicit expression for $G_q^{\mc H}$, we decompose $\mc H$ into
\begin{align*}
\mc H^{\on{vert}}_q &= \{ X\,:\, X|_q \equiv 0 \}\,,&
 \mc H^{\on{hor}}_q &= (\mc H^{\on{vert}}_q)^\perp\,.
\end{align*}
Then the induced metric is
\[
G_q^{\mc H}(h, h) = \langle X^{\on{hor}}, X^{\on{hor}} \rangle_{\mc H}\,,
\]
with $X \in \mf X_c(\R^d)$ any vector field such that $X \on{\circ} q = h$. 
The horizontal projection does not depend on the choice of the lift, i.e., 
if $X, Y \in \mf X_c(\R^d)$ coincide along $q$, then $X^{\on{hor}} = Y^{\on{hor}}$. 
We identify $\mc H^{\on{hor}}_q$ 
with the $G^{\mc H}$-completion of the tangent space $T_q \on{Emb}(M,\R^d)$. 
There are maps
\[
\begin{array}{ccc}
T_q \on{Emb}(M,\R^d) & \to & \mc H^{\on{hor}}_q \\
h & \mapsto & X^{\on{hor}}
\end{array}\,,
\qquad
\begin{array}{ccc}
\mc H^{\on{hor}}_q & \to & C^k_b(M,\R^d) \\
X & \mapsto & X\on{\circ} q
\end{array}\,.
\]
The composition of these two maps is the canonical embedding $T_q \on{Emb}(M,\R^d) \hookrightarrow C_b^k(M,\R^d)$. 
The space $\mc H^{\on{hor}}_q$ is again a reproducing kernel Hilbert space with the kernel given by
\[
K_q : M\x M \to \R^{d\x d}\,,\quad
K_q(x,y) = K(q(x), q(y))\,.
\]
Thus we have identified the induced Riemannian metric $G^{\mc H}$ on $\on{Emb}(M,\R^d)$ as
\[
G_q^{\mc H}(h, h) = \langle h, h \rangle_{\mc H^{\on{hor}}_q}\,.
\]
In this formula we identified $\mc H^{\on{hor}}_q$ with vector fields on $M$ with values in $\R^d$.

\subsection{Geodesic distance}

If the metric on $\on{Diff}(\R^d)$ is strong enough, then the induced Riemannian metric on $\on{Emb}(M,\R^d)$ has a point-separating geodesic distance function and we conjecture that the same is true for $B_e(M,\R^d)$.
\begin{theorem}
If the norm on $\mc H$ is at least as strong as the $C^0_b$-norm. i.e., $\mc H \hookrightarrow C^0_b$, then there exists $C>0$ such that for $q_0, q_1 \in \on{Emb}(M,\R^d)$ we have
\[
\| q_0 - q_1 \|_{\infty} \leq C \on{dist}^{\mc H}_{\on{Emb}}(q_0,q_1)\,.
\]
\end{theorem}

\begin{proof}
Since $\mc H \hookrightarrow C^0_b$, there exists a constant $C>0$, such that $\| X\|_{\infty} \leq C \| X\|_{\mc H}$ holds for all $X \in \mc H$.
Given $h \in \on{Emb}(M,\R^d)$ and $x \in M$ let $X \in \mf X(\R^d)$ be any vector field with $X \on{\circ} q = h$.  From
\[
|h(x)| = |X(q(x))| \leq \| X\|_{\infty} \leq C \| X\|_{\mc H}\,,
\]
we see that $\| h\|_\infty \leq C \| X \|_\infty$ and by taking the infimum over all $X$ we obtain
\[
\| h\|^2_\infty \leq C^2\, G_q^{\mc H}(h,h)\,.
\]
Now for any path $q(t)$ between $q_0$ and $q_1$ we have
\[
q_1(x) - q_0(x) = \int_0^1 \p_t q(t,x) \ud t\,,
\]
and thus
\begin{multline*}
|q_1(x) - q_0(x)| \leq \int_0^1 |\p_t q(t,x) | \ud t \leq \\
\leq C \int_0^1 \sqrt{G_{q(t)}^{\mc H}(\p_t q(t), \p_t q(t))} \ud t = C \on{Len}^{\mc H}_{\on{Emb}}(q)\,.
\end{multline*}
By taking the supremum over $x \in M$ and the infimum over all paths we obtain
\[
\| q_0 - q_1 \|_{\infty} \leq C \on{dist}^L_{\on{Emb}}(q_0,q_1)
\]
as required.\qed
\end{proof}

For the geodesic distance on shape space we have a positive result for the space $B_e(S^1,\R^2)$ of plane curves and the family $\mc H = H^k(\R^d)$ of Sobolev spaces. A lower bound on $\on{dist}_{B_e}^L$ is given by the Fr\'echet distance \eqref{frechet_dist}.

\begin{theorem}
The geodesic distance on $B_e(S^1,\R^2)$ of the outer metric induced by $\mc H = H^k(\R^d)$ with the operator $L=(1-A\De)^k$ for $A>0$ and $k\geq 1$ is bounded from below by the Fr\'echet distance, i.e., for $Q_0, Q_1 \in B_e(S^1,\R^2)$ we have
\[
\on{dist}_{B_e}^{L^\infty}(Q_0,Q_1) \leq \on{dist}^{H^k}_{B_e}(Q_0,Q_1)\,.
\]
\end{theorem}

\begin{proof}
Take $Q_0, Q_1 \in B_e(M,\R^d)$ and let $Q(t)$ be a path between them. Then by \cite[Prop. 5.7]{Michor2006a} we can lift this path to a horizontal path $q(t)$ on $\on{Emb}(S^1,\R^2)$. Then
\begin{multline*}
\on{dist}_{B_e}^{L^\infty}(Q_0,Q_1) \leq \| q(0) - q(1) \|_\infty = \\
=\on{Len}_{\on{Emb}}(q) = \on{Len}_{B_e}(Q)\,,
\end{multline*}
and by taking the infimum over all paths we obtain the result. \qed
\end{proof}

In order to generalize this result to $B_e(M,\R^d)$ one would need to be able to lift horizontal paths from $B_e(M,\R^d)$ to $\on{Emb}(M,\R^d)$. A careful analysis of the induced metric $G^\mc H$ in the spirit of \cite{Michor2006a} should provide such a result for a fairly general Sobolev-type metric.

\subsection{Geodesic equation}\label{outer_geod}
The geodesic equation on $\on{Emb}(M,\R^d)$ is most conveniently written in Hamiltonian form 
in terms of the position $q(t)$ and the momentum $\wt p(t) = p(t) \otimes \on{vol}^g$,
where $\vol^g=\vol(g)=\vol(q^*\bar g)$. 
The momentum defines a vector-valued distribution with support in the image of $q(t)$. 
The momentum $\wt p$ 
acts on  $X \in \mf X(\R^d)$ by
\[
\int_M \langle X \on{\circ} q(t), \wt p(t) \rangle = \int_{M} \langle X(q(t,x)), p(t, x) \rangle \on{vol}^g(x).
\]
Let us introduce the notation
\[
K_{q(t)}'(x, y) = D_1 K(q(t,x), q(t,y))
\]
for the derivative of the kernel with respect to the first variable. The geodesic equation is given by
\begin{align*}
  \p_t q(t,x) ={}& X(t,q(t,x)) \\
\p_t \left(p(t) \otimes \on{vol}^g\right)(t,x) ={}& \\
=-\bigg( \int_M p(t,x)^T K_{q(t)}'&(x,y) p(t,y) \on{vol}^g(y) \bigg) \otimes \on{vol}^g(x)\\
X(t,u) ={}& \int_M K(u, q(t,y)) p(t,y) \on{vol}^g(y)\,.
\end{align*}
See \cite{Michor2006a} for a derivation of the geodesic equation for plane curves and 
\cite{Micheli2013} for the related geodesic equation on $B_e(M,\R^d)$.

The vector field $X$ is not smooth but only $X \in \mc H$. Therefore it is not possible to 
horizontally lift geodesics from $\on{Emb}(M,\R^d)$ to $\on{Diff}(\R^d)$. One can however work in a suitable Sobolev completion of $\on{Diff}(\R^d)$. Then right-invariance of the Riemannian metric on $\on{Diff}(\R^d)$ implies the conservation of the momentum:
\[
\ph(t)^\ast \wt p(t,.) \textrm{ is independent of $t$}\,.
\]
From here we obtain via $\wt p(t,.) = \ph(t)_\ast \wt p(0,.)$ that $\ph(t)$ satisfies the following form of the {\it Euler-Poincar\'e equation on the diffeomorphism group} (EPDiff),
\begin{equation}
\label{outer_lagrange}
\p_t \ph(t,x) = \int_M K(\ph(t,x), .)\, \ph(t)_\ast \wt p(0,.)\,.
\end{equation}
See \cite{Holm2005} for details on singular solutions of the EPDiff equations. Theorem \ref{diff_globalwell} can be applied to show long-time existence of solutions of \eqref{outer_lagrange}.

\subsection{Curvature}

The representation of $B_e(M,\R^d)$ as the quotient 
\[
B_e(M,\R^d) = {\on{Diff}(\R^d)}/\on{Diff}(\R^d)_Q
\] 
was used in \cite{Micheli2013} together with an infinite dimensional version of O'Neil's formula to compute an expression for the sectional curvature on $B_e(M,\R^d)$. For details consult \cite[Sect.~5]{Micheli2013}.

\section{The space of landmarks}
\label{sec_lm}
By choosing $M$ to be the finite set $M=\{1,\dots,n\}$ we obtain as $\on{Emb}(M,\R^d)$ the set of landmarks, i.e., the set of $n$ distinct, labeled points in $\R^d$. Let us denote this space by
\[\L^n(\R^d):=\left\{(q^1,\ldots,q^n)|q^k\in\R^d, q^k\neq q^j, k\neq j\right\}.\]
Note that $\L^n(\R^d)$ is an open subset of $\R^{nd}$ and thus it is the first example of a finite dimensional shape space in this paper. As a consequence some of the questions discussed for other shape spaces have a simple answer for the space of landmarks. The geodesic distance is guaranteed to be point-separating, the geodesic equation is an ODE and therefore locally well-posed and due to Hopf-Rinow geodesic completeness implies metric completeness.

\begin{remark}
We regard landmark space as the set of all \emph{labeled} collections of $n$ points in $\R^d$, i.e., the landmarks
$q = (q^1,q^2,\ldots q^n)$, $\tilde q = (q^2,q^1,\ldots q^n)$
are regarded as different elements of $\L^n(\R^d)$. One could also consider the space of unlabeled landmarks $\L_u^n(\R^d)$, which would correspond to $B_e(M,\R^d)$. It is sometimes called also \emph{configuration space}. Since $\on{Diff}(M) = S_n$ is the symmetric group of $n$ elements, we have $\L^n_u(\R^d) = \L^n(\R^d)/S_n$. The group $S_n$ is a finite group, therefore the projection $\L^n(\R^d) \to \L^n_u(\R^d)$ is a covering map and so for local properties of Riemannian geometry it is enough to study the space $\L^n(\R^d)$.
\end{remark}

Before we proceed we need to fix an ordering for the coordinates on $\R^{nd}$. There are two canonical choices and we will follow the convention of \cite{Joshi2000}. A landmark $q$ is a vector $q=(q^1,\ldots,q^n)^T\in\L^n(\R^d)$ and each $q^i$ has $d$ components
$q^i=(q^{i1},\ldots,q^{id})^T$. We concatenate these vectors as follows
\begin{equation}\label{lmk:coordinates}
 q=(q^{11},\ldots,q^{1d},q^{21},\ldots,q^{2d},\ldots,q^{nd})^T\;.
\end{equation}

Riemannian metrics on $\L^n(\R^d)$, that are induced by the action of the diffeomorphism group, have been studied in \cite{Joshi2000,Micheli2012,Marsland2007} and on the landmark space on the sphere in \cite{Glaunes2004}. Other metrics on landmark space  include Bookstein's thin-plate spline distance \cite{Bookstein1997,Bookstein1976} and  Kendall's similitude invariant distance \cite{Kendall1984}. See \cite{Trouve2002} for an overview comparing the different approaches.

\subsection{A metric on $\L^n(\R^d)$ induced by $\on{Diff}(\R^d)$}

As in Sect.\ \ref{metric_outer} let the metric $G^{\mc H}$ on $\on{Diff}(\R^d)$ be defined via a Hilbert space $\mc H$ of vector fields satisfying the conditions given in Sect. \ref{metric_outer} and let $K$ be the reproducing kernel of $\mc H$. As before we will write $\on{Diff}(\R^d)$ for any of the groups $\on{Diff}_c(\R^d)$, 
$\on{Diff}_{\mc S}(\R^d)$ or $\on{Diff}_{H^\infty}(\R^d)$. The metric $G^{\mc H}$ induces a Riemannian metric $g_{\mc H}$ on $\L^n(\R^d)$ and we can calculate it explicitly; see Thm \ref{lmk_metric1}.

 For the convenience of the reader we will repeat the definition of the distance function on $\L^n(\R^d)$ induced by
 the metric $G^{\mathcal H}$;
  see  Sect. \ref{sec_outer} for the  the more general situation of embeddings  of an arbitrary manifold $M$ in $\R^d$.  Let $E$ be the energy functional of the metric $G^{\mathcal H}$ on the diffeomorphism group, i.e., 
\begin{equation}\label{lmk_energy}
 E(v)=\int_0^1 \|v(t,\cdot)\|^2_{\mathcal H}\ud t \,.
\end{equation}
 The induced distance function of the action of the diffeomorphism group on the landmark space is given by
 \begin{equation}\label{lmk_dist}
 \on{dist}^{\mc H}(q,\wt q)= \underset{v}{\on{inf}}\left\{\sqrt{E(v)}:\ph^v(q^i)=\wt q^i\right\} \,,
 \end{equation}
 where $\ph^v$ is the flow of the time dependent vector field $v$ and  where the infimum is taken over all sufficiently smooth vector fields
 $v: [0,1]\to \X(\R^d)$. Given a solution $v$ of the above minimization problem, the landmark trajectories $q^i(t)$ are then given as the solutions of the ODE
 \[\dot q^i(t)=v(t,q^i(t)),\qquad i=1\ldots,n\,.\]

 We will now define a Riemannian metric on the finite dimensional space $\L^n(\R^d)$ directly and we will see that it is in fact induced by the metric
 $G^\mathcal H$ on the diffeomorphism group. For a landmark $q$ we define the matrix 
 \begin{equation}\label{lmk_metric}
 g_{\mc H}^{-1}(q)=
\left( \begin{array}{ccc}
K(q^1,q^1)  & \cdots & K(q^1,q^n)\\
\vdots & \ddots  & \vdots\\
K(q^n,q^1) & \cdots & K(q^n,q^n)\end{array} \right)\in \R^{nd \x nd} \,,
\end{equation}
where $K : \R^d\x\R^d \to \R^{d\x d}$ is the kernel of $\mc H$. 
That $g_{\mc H}$ defines a Riemannian metric on $\L^n(\R^d)$ can be easily shown using the properties of the kernel $K$.

The metric $g_{\mathcal H}$ defines, in the usual way, an energy functional directly on the space of landmark trajectories,
 \begin{equation}\label{lmk_energy2}
 \wt E(q(t))=\int_0^1 \dot q(t)^T g_{\mc H}(q(t))\dot q(t)\ud t\,,
\end{equation}
and one can  also define the induced distance function of $g$ as
\begin{equation}\label{lmk_energy2a}
\wt{\on{dist}}{}^{\mc H}(q,\wt q) = \underset{q(t)}{\on{inf}}\left\{\sqrt{\wt E(q(t))}:q(0)=q, q(1)=\wt q\right\}\,,
\end{equation}
where the infimum is taken over all sufficiently smooth paths in landmark space $q: [0,1]\to \L^n(\R^d)$.
 
 It is shown in \cite[Prop. 2]{Micheli2012} that the minimization problems 
 \eqref{lmk_dist} and  \eqref{lmk_energy2a} are equivalent and that the induced distance functions are equal:
 \begin{theorem}[Prop. 2, \cite{Micheli2012}]
 Let $v$ be a minimizer of the energy functional \eqref{lmk_energy}. Then the trajectory $q(t)$, which is obtained as the solution of the system of ODE's
  \[\dot q^i(t)=v(t,q^i(t)),\qquad i=1\ldots,n\,,\]
  minimizes the energy functional \eqref{lmk_energy2} and 
$E(v) = \wt E(q)$.
On the other hand, if $q(t)$ is a minimizer of the energy functional \eqref{lmk_energy2} define the vector field 
\begin{equation}\label{lmk_vectorfield}
v(t,x)=\sum_{i=1}^n p_i(t) K(x,q^i(t))
\end{equation}
with the \emph{momenta} $p_i: [0,1]\to \R^d$ given implicitely by
\begin{equation}\label{lmk_momenta}
\dot q^i(t) = \sum_{j=1}^n p_j(t) K(q^i(t),q^j(t))\,.
\end{equation}
Then the vector field $v$ is a minimizer of the energy  \eqref{lmk_energy} and we have  $\wt E(q)=E(v) $.
 \end{theorem}
 Thus we have:
\begin{theorem}\label{lmk_metric1}
If $\mc H \hookrightarrow C^k_b$, then the induced metric $g_{\mc H}$ on $\L^n(\R^d)$ is given by
\begin{equation}
 g_{\mc H}(q)=
\left( \begin{array}{ccc}
K(q^1,q^1)  & \cdots & K(q^1,q^n)\\
\vdots & \ddots  & \vdots\\
K(q^n,q^1) & \cdots & K(q^n,q^n)\end{array} \right)^{-1} \in \R^{nd \x nd} \,,
\end{equation}
where $K \in C^k(\R^d\x\R^d,\R^{d\x d})$ is the kernel of $\mc H$. We have $g_{\mc H} \in C^k(\R^{nd},\R^{nd \x nd})$.
\end{theorem}
We will discuss the solutions of the minimization problem  \eqref{lmk_energy2} in Sect.\ \ref{lmk_completeness}.

\begin{remark}
 Note that in the articles \cite{Micheli2012,MicheliPhD} the coordinates were ordered in a different way. Given $q=(q^1,\dots,q^n)$ they flatten it as
\[ q=(q^{11},\ldots,q^{n1},q^{12},\ldots,q^{n2},\ldots ,q^{nd})^T\,.\]
If the kernel $K(x,y)$ of $\mc H$ is a multiple of the identity matrix, i.e., $K(x,y) = \boldsymbol k(x,y) \on{Id}_{d\x d}$ for a scalar function $\boldsymbol k$, then the matrix $g_{\mc H}(q)$ is sparse and these coordinates allow us see the sparsity in an elegant way,
\begin{equation*}
g_{\mc H}^{-1}(q)=
\left( \begin{array}{cccc}
  \boldsymbol k(q)  & 0 & \cdots& 0 \\
0 &\boldsymbol k(q)  & \cdots&0\\
\vdots & \vdots & \ddots&\vdots\\
0&\cdots&0& \boldsymbol k(q) \end{array} \right)\;,
\end{equation*}
Here $\boldsymbol k(q)$ denotes the $n\x n$-matrix $(\boldsymbol k(q^i,q^j))_{1\leq i,j \leq n}$.
\end{remark}

\subsection{The geodesic equation}
The geodesic equation can be deduced from the equation in the general case $\Emb(M,\R^d)$; see Sect.\ \ref{outer_geod}.
\begin{theorem}
If $\mathcal H\hookrightarrow C^1_b$, then the Hamiltonian form of the geodesic equation of the metric 
$g_{\mc H}$ on $\L^n(\R^d)$  is given by
\begin{equation}\begin{aligned}\label{lmk_geod}
\dot q^i& = \sum_{j=1}^N  K(q^i,q^j)p_j, \\
\dot p_i &= -\sum_{j=1}^N  p_i^\top  (\partial_1K)(q^i,q^j)  p_j 
\end{aligned}
\end{equation}
with $p_i(t)=K(q(t))^{-1} q^i(t)$ the vector valued momentum.
\end{theorem}
For scalar kernels this system has been studied  in the articles \cite{Marsland2007,Micheli2012}; see also the PhD-thesis of Micheli \cite{MicheliPhD}.
Two examples of a two-particle interaction can be seen in Fig.\ref{lmk_twoparticles}.

\begin{remark}
A different possibility to derive the above geodesic equation is to consider directly the Hamiltonian function of the finite dimensional Riemannian manifold $(\L^n(\R^d), g_{\mc H})$.
Following \cite[Eqn. 1.6.6]{Jost2008} it is given by
\[\on{Ham}(p,q)=\frac12p^{T}g(q)^{-1} p=\sum_{i,j=1}^N p_i^T K(q_i,q_j) p_j\,,\]
Then the  geodesic equations \eqref{lmk_geod} are just Hamilton's equation for $\on{Ham}$:
\begin{align*}
\dot q^i = \frac{\partial \on{Ham}}{\partial p_i}, \qquad 
\dot p_i = -\frac{\partial \on{Ham}}{\partial q^i} \,.
\end{align*}
\end{remark}

\begin{figure}[ht]
\begin{minipage}{0.49\linewidth}
     \includegraphics[width=\textwidth]{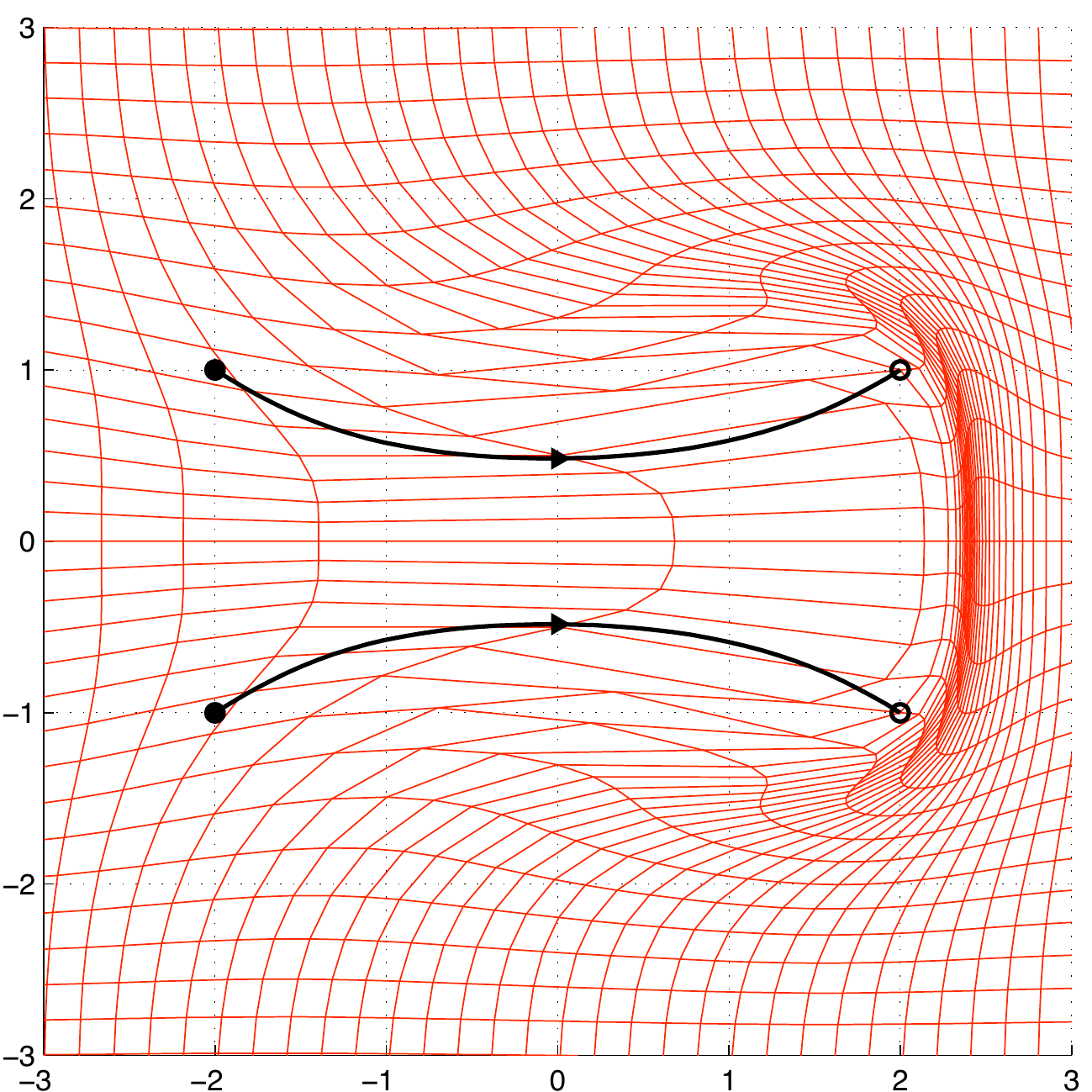}
   \end{minipage}
   \hfill
   \begin{minipage}{0.49\linewidth}
    \includegraphics[width=\textwidth]{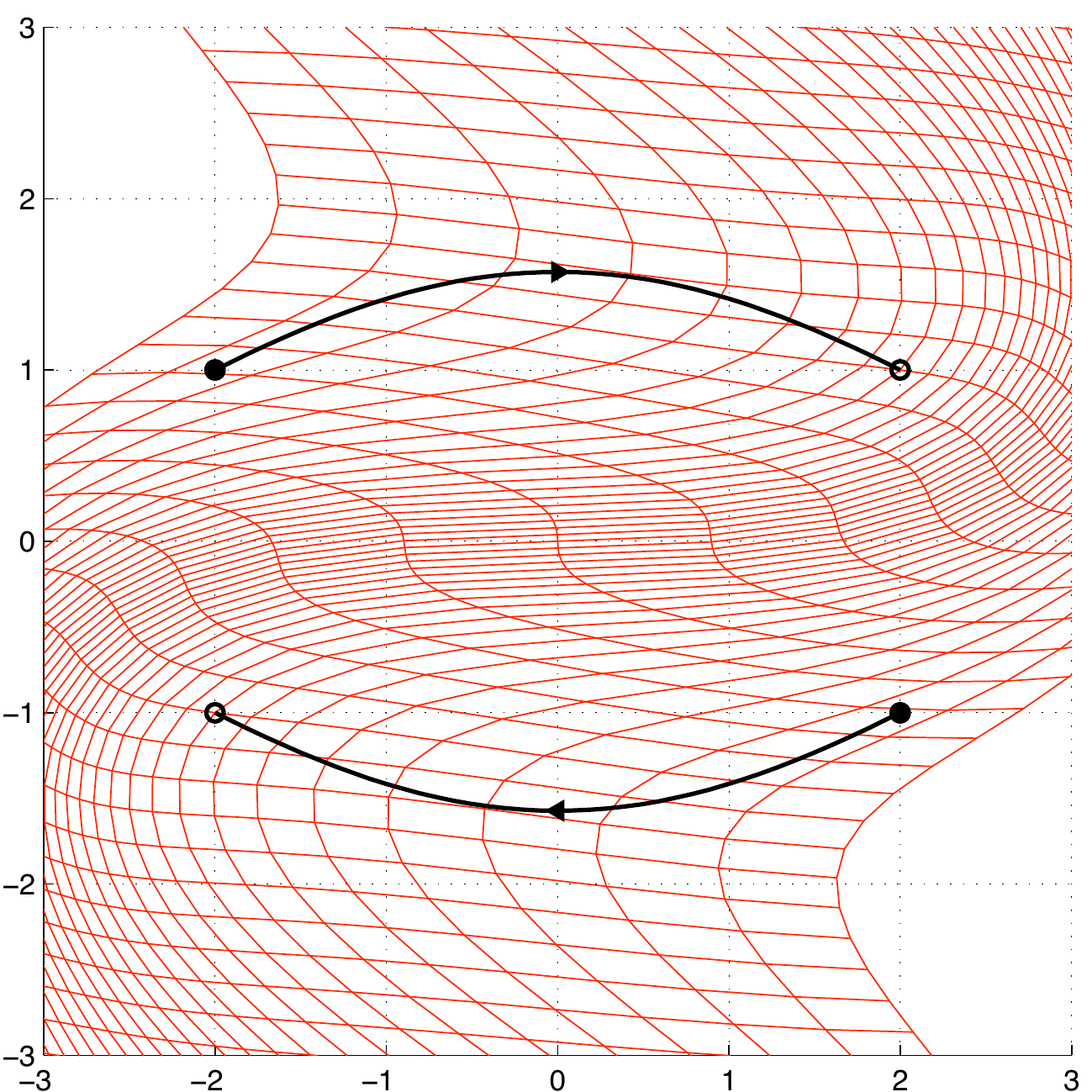}
   \end{minipage}
\caption{Two geodesics in $\L^2(\R^2)$. The grid represents the corresponding diffeomorphisms. On the left-hand side both landmarks travel in the same direction, and the two paths tend to attract each other. On the right hand side the landmarks travel in opposite directions and the paths try to avoid each other. Original image published in \cite{Micheli2012}.}
\label{lmk_twoparticles}
\end{figure}

\begin{remark}\label{lmk_soliton}
We can regard a geodesic curve of landmarks as a soliton-like solution of the geodesic equation on 
$\Diff(\mathbb R^d)$ where the corresponding momentum is a linear combination of vector valued 
delta distributions and travels as such. 
\end{remark}

\subsection{Completeness}\label{lmk_completeness}
As a  consequence of the global well-posedness theorem on the full diffeomorphism group -- Thm. \ref{diff_globalwell} --
we can deduce the long-time existence of geodesics on Landmark space. To do so we solve the geodesic equation \eqref{outer_lagrange} on the diffeomorphism group
for a singular initial momentum $p(0,x)=\sum_{j=1}^n p_j \delta(x-q^j)$. Then the landmark trajectories are given by $q^i(t)=\ph(t,q^i(0))$, where $\ph\in\Diff(\R^d)$ is the
solution of \eqref{outer_lagrange}.

\begin{theorem}
If $\mathcal H \hookrightarrow C^1_b$, then  the Riemannian manifold  $\left( \L^n(\R^d),g_{\mc H}\right)$ is geodesically complete.
\end{theorem}
A consequence of this theorem is that two landmarks will never collide along a geodesic path. For finite dimensional Riemannian manifolds with a metric that is at least $C^2$ the theorem of Hopf-Rinow asserts that the notions of geodesic completeness and metric completeness are equivalent. 
\begin{corollary}
If $\mathcal H\hookrightarrow C^2_b$, then $(\L^n(\R^d), \on{dist}^{\mc H})$ is a complete metric space.
\end{corollary}
For a $C^2$-metric $g_{\mc H}$ one can use once more the theorem of Hopf-Rinow to show the well-posedness of the geodesic boundary value problem. 
\begin{corollary}[Prop. 1, \cite{Micheli2012}]
If $\mathcal H \hookrightarrow C^2_b$ then for each pair of landmarks $q,\tilde q\in \L^n(\R^d)$ there exists a minimizer $q(t)\in C^{1}([0,1],\L^n(\R^d))$ of the energy functional  \eqref{lmk_energy2}.
\end{corollary}
In fact the existence of minimizers to the boundary value problem on landmark space can be proven under even weaker smoothness conditions on the metric $g_{\mc H}$,
see \cite[Sect.\ C]{Joshi2000}.

\subsection{Curvature}
We see from \eqref{lmk_metric} that the expression for the co-metric $g_{\mc H}^{-1}$ is much simpler 
than that for $g_{\mc H}$. 
In the article \cite{Micheli2012} the authors took this observation as a motivation to derive a formula for the sectional curvature in terms of the co-metric, now called MarioÂŽs formula; see \cite[Thm. 3.2]{Micheli2012}. Using this formula they were able to calculate the the sectional curvature of the landmark space $(\L^n(\R^d),g_{\mc H})$; see \cite[Thm. 9]{Micheli2012}. We will not present these formulas in the general case but only for the special case of two landmarks in $\R$:
\begin{theorem}[Prop. 23, \cite{Micheli2012}]
The sectional 
curvature on $\L^2(\R)$ depends only on the distance $\rho=|q-\tilde q|$ between the two landmarks $q,\tilde q$. For a metric $g_{\mathcal H}$, with reproducing kernel $K \in C^2(\R,\R)$, it is given by
\[k(\rho)=\frac{K(0)-K(\rho)}{K(0)+K(\rho)}K''(\rho)-\frac{2K(0)-K(\rho)}{(K(0)+K(\rho))^2}K'(0)^2 \;.\]
\end{theorem}
For a Gaussian kernel $K$ a plot of the curvature depending on the distance between the landmarks can be seen in Fig.\ref{fig:curvature}.
\begin{figure}[ht]
\includegraphics[width=.48\textwidth]{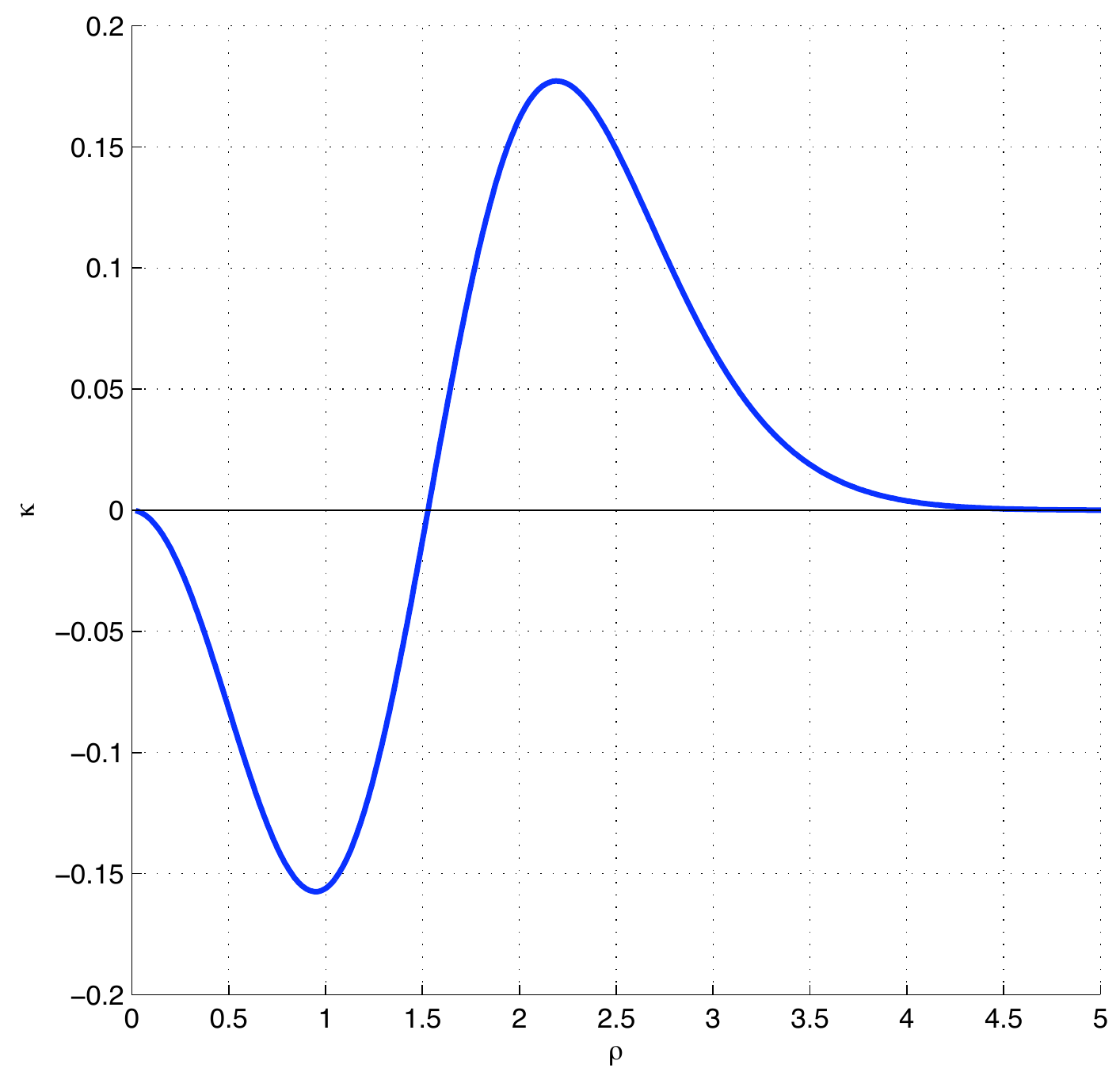}
\caption{Sectional curvature of $\L^2(\R),g_K$ as a function of the distance between the landmarks $q_0,q_1$.
Here $K$ was the Gaussian kernel $K(x)=\on{exp}(-\tfrac{x}{2})$. Original image published in \cite{Micheli2012}.}
\label{fig:curvature}
\end{figure}

\section{Universal Teichm\"uller space as shape space}
Here we sketch how $\Diff(S^1)/PSL(2,\mathbb R)$ parametrizes the shape space of simple closed 
smooth plane curves modulo translations and scalings and discuss the associated Riemannian metric, called 
the Weil-Peterson metric. This metric has 
nonpositive curvature, is geodesically complete, and any two shapes can be connected by a unique 
minimal geodesic. There exist soliton-like solutions  which are called \emph{teichons} which are given 
by a finite dimensional Hamiltonian system. They relate to geodesics of shapes like landmarks do to 
geodesics in $\Diff(\mathbb R^d)$; see Sect.\ \ref{lmk_soliton}.
This theory and the corresponding numerical analysis has been developed in \cite{Mumford2004,Mumford2006}. 
The use of teichons has been developed in \cite{Kushnarev2009}.  

Given a 1-dimensional closed smooth and connected submanifold $\Ga$ in $\mathbb R^2=\mathbb C$ inside 
the Riemann sphere $\ol{\mathbb C}=\mathbb C\cup\{\infty\}$,
we 
consider its interior $\Ga_{\text{int}}$ and its exterior $\Ga_{\text{ext}}$ which contains 
$\infty$; these are smooth 2-manifolds with boundary. 
Let $\mathbb D_{\text{int}}$ and $\mathbb D_{\text{ext}}$ denote the unit disk and the exterior of the unit
disk respectively. By the smooth Riemann mapping theorem (\cite[page 323]{Taylor2011I} or \cite{WikiSmoothRiemannMapping}) 
there exists a biholomorphic mapping 
$\Ph_{\text{int}}:\mathbb D_{\text{int}}\to \Ga_{\text{int}}$ extending smoothly to the 
boundaries, unique up to replacing it by 
$\Ph_{\text{int}}\on{\circ} A$ for a M\"obius transformation 
\[
A(z)= \frac{az+b}{\ol bz+\ol a} \text{   with }
\begin{pmatrix} a & b \\ \ol b & \ol a\end{pmatrix}\in PSU(1,1)\cong PSL(2,\mathbb R).
\]
Likewise we have a biholomorphic map between the exteriors
$\Ph_{\text{ext}}:\mathbb D_{\text{ext}}\to \Ga_{\text{ext}}$ which is unique 
by the requirement that $\Ph_{\text{ext}}(\infty)=\infty$ and $\Ph_{\text{ext}}'(\infty)>0$.
The resulting diffeomorphism 
\[
\Ps := \Ph_{\text{ext}}^{-1}\on{\circ}\Ph_{\text{int}}:S^1\to S^1,
\]
projects to a unique element of $\Diff^+(S^1)/PSL(2,\mathbb R)$ (here $S^1$ is viewed as 
$P^1(\mathbb R)$).
It is called the \emph{fingerprint} of $\Ga$. Any coset $\Ps.PSL(2,\mathbb R)$ comes from a shape 
$\Ga$, and two shapes give the same coset if they differ by a M\"obius transformation in 
$\on{Aut}(\ol{\mathbb C})$ which fixes $\infty$ and has positive derivative at $\infty$; i.e., by 
translations and scalings.

One can reconstruct the shape $\Ga$ from the fingerprint $\Ps.PSL(2,\mathbb R)$ by \emph{welding}: 
Construct a Riemann surface by welding the boundaries of $\mathbb D_{\text{int}}$ and 
$\mathbb D_{\text{ext}}$ via the mapping $\Ps$. The result is conformally equivalent to the Riemann 
sphere and we use a conformal mapping $\Ph$ from the welded surface to the sphere which takes $\infty$ to 
$\infty$ and has positive derivative at $\infty$. Then $\Ga$ equals $\Ph^{-1}(S^1)$ up to a 
translation and a scaling of $\mathbb C$. An efficient numerical procedure for welding is described 
in \cite{Mumford2004, Mumford2006}.

The quotient $\mathcal T:=\Diff(S^1)/PSL(2,\mathbb R)$, also known as universal Teichm\"uller 
space, is naturally a coadjoint orbit of the Virasoro group (see Sect.\ \ref{EPDiff})
and as such it carries a natural invariant K\"ahler structure; 
see \cite{KirillovYuriev88}. The corresponding Riemann metric can be described as follows. 
For $u\in \X(S^1)\cong C^\infty(S^1)$ we consider the Fourier series 
$u(\th)= \sum_{n\in \mathbb Z} a_n e^{in\th}$ with $\ol a_n= a_{-n}$ and the seminorm
\[
\|u\|_{\on{WP}}^2 = \sum_{n\in \mathbb Z}|n^3-n||a_n|^2. 
\]
The kernel of this seminorm consists of vector fields of the form $\ol a_1 e^{-i\th} + a_0 + a_1 e^{i\th}$; i.e., 
$\on{ker}(\|\cdot\|_{\on{WP}})=\mathfrak s\mathfrak l(2,\mathbb R)$. 
So this gives an inner product on the tangent space at the 
base point of  $\mathcal T$. 
This norm can also be defined by the elliptic pseudodifferential operator 
$L=\mathcal H(\partial_\th^3+\partial_\th)$ via 
$\|u\| = \int_{S^1} L(u).u \ud\th$, where the \emph{periodic Hilbert transform} $\mathcal H$ is 
given by convolution with $\tfrac 1{2\pi}\on{cotan}(\tfrac \th 2)$. The inverse of $L$ is convolution with the 
Green's function 
\begin{multline*}
K(\th) =\sum_{|n|>1}\frac{e^{in\th}}{n^3-n} 
=\\
= (1-\cos\th)\log\left(2(1-\cos\th)\right) + \frac32\cos\th-1.
\end{multline*}
According to Sect.\ \ref{EPDiff}, $\ph(t)\in \Diff(S^1)$ projects to a geodesic in $\mathcal T$ 
if and only if the right logarithmic derivative $u(t)= \partial_t \ph(t) \on{\circ} \ph(t)^{-1}$ satisfies  
\begin{gather*}
L(u_t) = - \ad_u^*(Lu) \quad\text{  or } \\
(Lu)_t + u.(Lu)_\th + 2u_\th.(Lu) =0
\end{gather*}
and $u(0)$ has vanishing Fourier coefficients of order $-1$, $0$, $1$. 
We call $m=Lu\in (\X(S^1)/\mathfrak s\mathfrak l(2,\mathbb R))'$ the \emph{momentum}, with 
$u=G*m$.
The Weil-Petersson metric described by $L$ is a Sobolev metric of order 3/2. The extension to the 
corresponding Sobolev completions has been worked out by 
\cite{Gay-Balmaz2013_preprint}.

If we look for the geodesic evolution of a momentum of the form 
\begin{align*}
m=\sum_{j=1}^N p_j\de(\th-q_j),\quad\text{so that}\quad
v=\sum_{j=1}^N p_j G(\th-q_j)
\end{align*}
a finite combination of delta distributions,  
which lies outside of the image of 
$L:\X(S^1)/\mathfrak s\mathfrak l\to (\X(S^1)/\mathfrak s\mathfrak l)'$, we see that the evolution of the parameters $q_j, p_j$ is given by 
the Hamiltonian system
\[
\begin{cases} &\dot p_k = -p_k\sum_{j=1}^N p_jG'(q_k-q_j) 
              \\
							&\dot q_k = \sum_{j=1}^N p_j G(q_k-q_j)
							\end{cases} 
\]
These solutions are called \emph{Teichons}, and they can be used to approximate smooth geodesics of 
shapes in a very efficient way which mimics the evolution of landmarks.
The disadvantage is, that near concave parts of a shape the
teichons crowd up exponentially.  An example of such a geodesic can be seen in Fig. \ref{fig:teichons};
see \cite{Kushnarev2009} and \cite{Kushnarev2012_preprint} for more 
details.

\begin{figure}[ht]
\includegraphics[width=.48\textwidth]{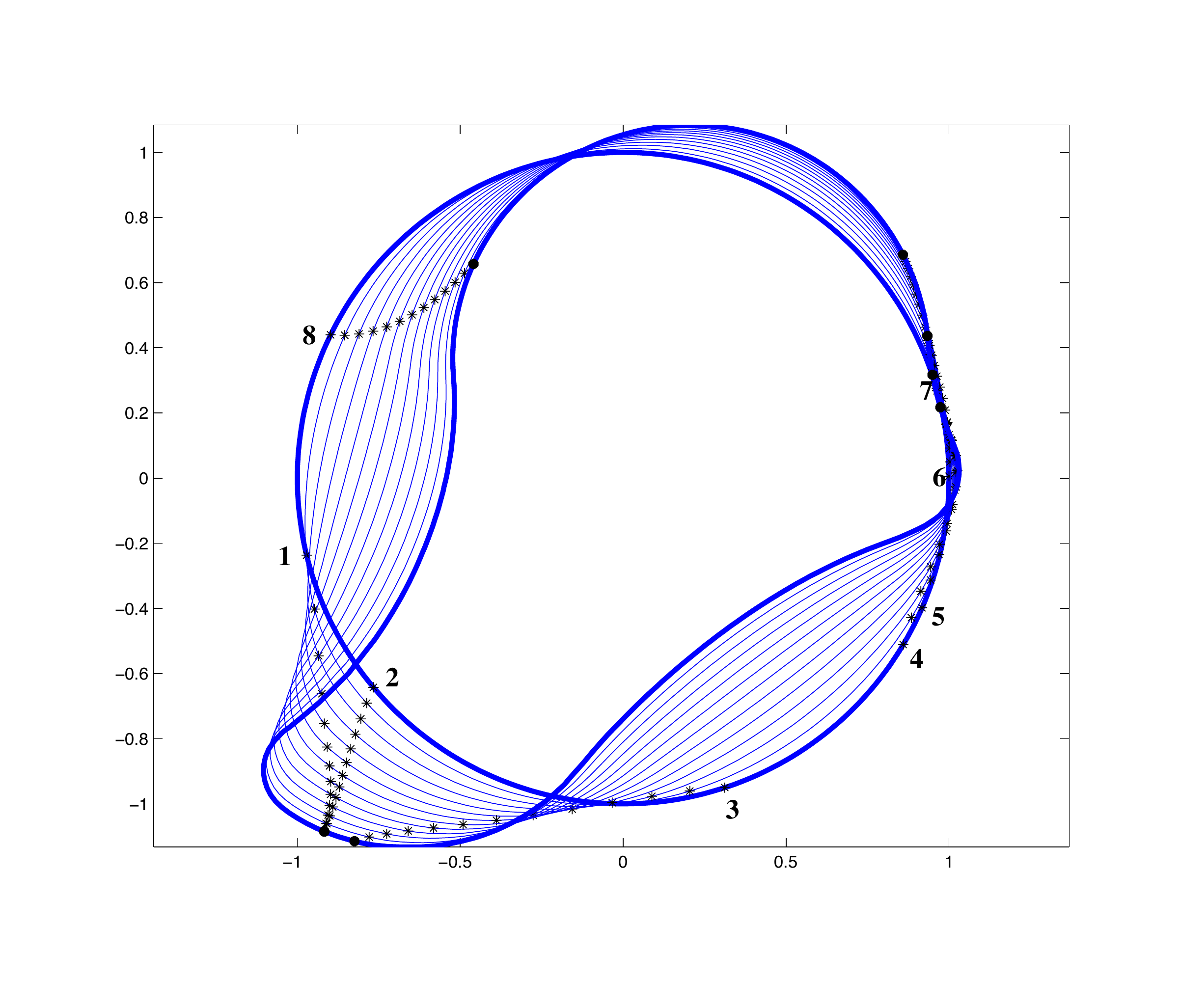}
\caption{Evolution of an $8$-Teichon from the circle to a Donald-Duck-like shape. Positions of individual $1$-Teichons are marked by asterisks. Original image published in \cite{Kushnarev2009}.}
\label{fig:teichons}
\end{figure}

\section{The space of Riemannian metrics}
\label{sec_mets}

Let $M$ be a compact manifold without boundary and $\on{dim}(M) = m$. In this part we describe the Riemannian geometry on $\Met(M)$, the manifold of all Riemannian metrics on $M$. The $L^2$-metric on $\Met(M)$ is given by
\[G^{E}_g(h,k)=\int_M \on{Tr}\big(g^{-1} hg^{-1} k\big)\vol(g)\,, \]
with $g \in \Met(M)$ and $h,k\in T_g\Met(M)$. Each tangent vector $h$ is a bilinear form $h: TM \x_M TM \to \R$, that is interpreted as a map $TM \to T^\ast M$. This metric has been introduced in \cite{Ebin1970a} and is also known as the Ebin-metric. Its geodesic equation and curvature 
have been calculated in \cite{Freed1989,Gil-Medrano1991},
and the induced distance function and metric completion have been studied 
by Clarke \cite{Clarke2013c,Clarke2009,Clarke2010,Clarke2011}.

Similar to Riemannian metrics on immersions, Sobolev metrics of higher order and almost local metrics can be defined using a (pseudo differential) operator field 
$L$ acting on the tangent space of $\Met(M)$. To be more precise, let 
\[L: T\Met(M) \to T\Met(M)\]
be a smooth base-point preserving bundle isomorphism, such that for every 
$g\in\Met(M)$ the map 
\[L_g: T_g\Met(M) \to T_g\Met(M)\]
is a pseudo differential operator, that is symmetric and positive with respect to the metric $G^E$.
Then we can define the metric $G^L$ by
\[G^L(h,k)=\int_M \on{Tr}\big(g^{-1} L_g(h)g^{-1} k\big)\vol(g)\,.\]
Let us also assume, that the operator field $L$ is invariant under the action of $\Diff(M)$, i.e.,
\[\ph^*(L_gh)=L_{\ph^*g}(\ph^*h)\;.\]
Then the metric $G^L$ induces a Riemannian metric on $\Met(M)/\on{Diff}_0(M)$ 
where $\Diff_0(M)$ denotes the group of all diffeomorphisms that are homotopic to the identity. 
In relativity theory the Lorentzian analog of the space $\Met(M)/\on{Diff}_0(M)$ is called \emph{super space}, 
since it is the true phase space of Einstein's equation. 

An example for an operator field $L$ is
\[L_gh=h+(\Delta^g)^lh\,,\qquad l\geq 0\,.\]
The resulting metric $G^L$, which is a the Sobolev metric of order $l$, has been introduced in \cite{Bauer2013a}.
Other metrics, that have been studied include conformal transformations of the $L^2$-metric \cite{Clarke2013b,Bauer2013a},
\[
G^\Ph_g(h,k) = \Ph(\on{Vol}_g) \int_M \on{Tr}\big(g^{-1} hg^{-1} k\big)\vol(g)\
\]
with $\Ph \in C^\infty(\R_{>0}, \R_{>0})$ and scalar curvature weighted metrics  \cite{Bauer2013a},
\[
G^\Ph_g(h,k) = \int_M \Ph(\on{Scal}^g) \on{Tr}\big(g^{-1} hg^{-1} k\big)\vol(g)\,,
\]
with $\Ph \in C^\infty(\R, \R_{>0})$.

The main focus of the section will be on the $L^2$-metric.

\subsection{Connections to Teichm\"uller theory and information geometry}
Our main motivation to consider the space of all Riemannian metrics in this article lies in its possible application to shape analysis of surfaces
as explained in Sect.\ \ref{intro_spaces}; see also \cite{Jermyn2012}.
 
Another  motivation for the study of the $L^2$-metric on the  manifold of metrics can be found in its  connections to Teichm\"uller theory. 
Let $M$ be a Riemann surface of genus greater than one. Then the $L^2$-metric, restricted to the space $\Met_1$ of hyperbolic metrics, induces the Weil-Peterson metric on Teichm\"uller space $\Met_1(M)/\on{Diff}_0(M)$.
This is described for example by Fischer and Tromba \cite{Fischer1984} or Yamada \cite{Yamada2004,Yamada2011}.

A surprising connection can be also found with the field of information geometry, since the $L^2$-metric  descends to the Fisher-Rao metric on the space of volume densities.
To understand this connection we will consider the Riemannian metric on $\Diff(M)$ induced by $G^E$. For a fixed metric $g_0\in \Met(M)$ we introduce the map:
\begin{equation*}
\label{diff_met}
\on{Pull}: \on{Diff}(M) \to \Met(M),\quad \ph \mapsto \ph^\ast g_0\,.
\end{equation*}
Now we can define a metric $G^{\on{Pull}}$ on $\Diff(M)$ as the pullback of the $L^2$-metric under the map $\on{Pull}$, i.e., 
\[G^{\on{Pull}}_\ph(h,k)=G^E(T_{\ph}\on{Pull}h,T_{\ph}\on{Pull}k)\,.\]
This mapping and the induced metric on $\Diff(M)$ for a variety of metrics on $\Met(M)$ is studied in \cite{Bauer2013_preprint}. The metric $G^{\on{Pull}}$ is invariant under the left action by the group $\on{Diff}_\mu(M)$ of volume-preserving diffeomorphisms and the metric induced on the quotient space 
\[
\on{Dens}(M) \cong \Diff_{\mu}(M)\backslash\Diff(M)
\]
of densities is the Fisher-Rao metric; see \cite[Thm. 4.9]{Modin2012_preprint}.

Another possibility, to see the connection to information geometry was implicitly presented in \cite{Clarke2013}. There the authors consider the subspace of K\"ahler metrics in a fixed K\"ahler class -- assuming that $M$ admits a 
K\"ahler structure. Then the Ebin metric induces the so-called Calabi geometry on the space of K\"ahler metrics.  It was then observed in \cite{Khesin2013} that this space is, via the Calabi-Yau map, isometric to 
the space of volume densities with the Fisher-Rao metric. 

\subsection{Geodesic distance}

In contrast to to the spaces of immersions, submanifolds and the diffeomorphism group, the $L^2$-metric on $\Met(M)$ induces a point-separating geodesic distance function.

\begin{theorem}[Thm. 18, \cite{Clarke2010}]
The $L^2$-metric induces a point-separating geodesic distance function on $\Met(M)$.
\end{theorem}
\begin{remark}
Note that this result also holds for all metrics, that are stronger than the $L^2$-metric, i.e., 
\[
G^E_g(h,h) \leq C G_g(h,h)\,,
\]
with a constant $C>0$, independent of $g$. This applies in particular to almost local metrics, if the function $\Ph$ is bounded from below by $\Ph \geq C > 0$, as well as to most Sobolev-type metrics.
\end{remark}

Fix a Riemannian metric $\wt g$ on $M$. For each $x \in M$ denote by $\Met(M)_x = S^2_+ T^\ast_x M$ the space of symmetric, positive definite $0 \choose 2$-tensors at $x$. Then for $b, c \in T_a \Met(M)_x$ the expression
\[
\ga_{x,a}(b, c) = \on{Tr}(a^{-1} b a^{-1} c) \sqrt{\on{det} \wt g(x)^{-1} a}
\]
defines a Riemannian metric on the finite-dimensional manifold $\Met(M)_x$. Denote by $d_x$ the induced geodesic distance function and define the following distance on $\Met(M)$,
\[\Omega_2(g_0,g_1)=\sqrt{\int_M d_x(g_0(x),g_1(x))^2\on{vol}(\wt g)}\,.\]
The following theorem states that computing the geodesic distance on $\Met(M)$ with respect to the $L^2$-distance, is equivalent to summing point-wise geodesic distances on $\Met(M)_x$.

\begin{theorem}[Thm. 3.8,\cite{Clarke2013b}]
Geodesic distance induced by the $L^2$ metric and the distance $\Om_2$ coincide, i.e., for all $g_0, g_1 \in \Met(M)$,
\[\dist^{E}(g_0,g_1)=\Omega_2(g_0,g_1)\,.\]
\end{theorem}

Similar as in the case of the $G^A$-metric and the Sobolev metrics on the space of immersions the square root of the volume is again a Lipschitz continuous function.
\begin{theorem}[Lem.~12,\cite{Clarke2010}]
Geodesic distance induced by the $L^2$-metric  satisfies the inequality 
\[
\left|\sqrt{\Vol(F, g_0)}-\sqrt{\Vol(F, g_1)}\right| \leq
 \frac{\sqrt{m}}4 \on{dist}^{E}_{\Met(M)}(g_0,g_1)
\]
for any measurable set $F\subset M$. Here $\Vol(F, g)$ denotes the volume of $F\subset M$ with respect to the metric $g$.  

This implies the Lipschitz continuity of the map
\[\sqrt{\Vol}:\left(\Met(M),\on{dist}^{F}_{\Met(M)}\right) \to \mathbb R_{\geq 0}.\]
\end{theorem}

On the other hand we also have the following upper bound for the geodesic distance.
\begin{theorem}[Prop. 4.1, \cite{Clarke2013c}]
For the $L^2$-metric the geodesic distance is bounded from above by
\[\on{dist}^E(g_0, g_1) \leq C(m)\left(\sqrt{\Vol(F, g_0)}+\sqrt{\Vol(F, g_1)}\right)\;,\]
where $F$ denotes the support of $g_1-g_0$
\[F = \ol{\left\{x \in M | g_0(x) \neq g_1(x)\right\}}\,,\]
and $C(m)$ is a constant depending only on the dimension of $M$.
\end{theorem}

The above corollary implies that the set $\Met_\mu(M)$ of all Riemannian metrics having a total volume less or equal than $\mu$ has a finite diameter with respect to the $L^2$-metric.

\subsection{The geodesic equation}\label{metrics_geod}
The Christoffel symbols for the $L^2$-metric were first calculated in \cite[Sect.\ 4]{Ebin1970a}. Subsequently Freed and Groisser \cite{Freed1989} and Michor and Gil-Medrano \cite{Gil-Medrano1991} computed the geodesic equation and found explicit solution formulas. The geodesic equation for higher order Sobolev type metrics and Scalar curvature metrics can be found in \cite{Bauer2013a} and for volume weighted metrics in \cite{Bauer2013a,Clarke2013}. 

The geodesic equation for the $L^2$-metric decouples the time and spatial variables, i.e., instead of being a PDE in $(t,x)$, it is only an ODE in $t$.
\begin{lemma}[Sect.\ 4, \cite{Ebin1970a}]
The geodesic equation for the $L^2$-metric is given by the ordinary differential equation:
\[g_{tt}=\frac14\on{Tr}(g^{-1} g_t g^{-1} g_t)g+g_t g^{-1} g_t-\frac12\Tr(g^{-1} g_t) g_t\,.\]
\end{lemma}
There exists an explicit solution formula for this ODE.
\begin{theorem}
The geodesic starting at $g_0\in\Met(M)$ in the direction of
$h\in T_{g_0}\Met(M)$ is  given by the curve
\[g(t)=g_0 e^{a(t)\on{Id}+b(t)H_0},\]
where $H_0$ is the traceless part of $H:= g_0^{-1} h$, i.e.,
$H_0=H-\frac{\on{Tr}(H)}m \on{Id}$, and where $a(t)$ and $b(t)\in
C^\infty(M)$ are defined by
\begin{align*}
a(t) &= \tfrac2m \log\left((1+\tfrac t4\on{Tr}(H))^2
        +\tfrac m{16}\on{Tr}(H_0^2)t^2\right)\\
b(t) &= 
\begin{cases}
        \frac4{\sqrt{m\on{Tr}(H_0^2)}}
                 \arctan\left(\frac{\sqrt{m\on{Tr}(H_0^2)}\,t}{4+t\on{Tr}(H)}\right),
                &\hspace{-.3cm}\on{Tr}(H_0^2)\neq0\\
        \frac t{1+\frac t4\on{Tr}(H)}, &\hspace{-.3cm}\on{Tr}(H_0^2)=0.
\end{cases}
\end{align*}
Here $ \arctan$ is taken to have values in
$(-\frac\pi2,\frac\pi2)$ for the points of the manifold where
$\on{Tr}(H)\geq0$, and on a point where $\on{Tr}(H)<0$ we define
\begin{align*}
&\arctan\left(\frac{\sqrt{m\on{Tr}(H_0^2)}\,t}{4+t\on{Tr}(H)}\right) \\&\qquad= 
\begin{cases}
         \arctan\text{ in }[0,\frac\pi2) &\text{ for }
                t\in [0,-\frac4{\on{Tr}(H)})\\
        \frac\pi2  &\text{ for }t=-\frac4{\on{Tr}(H)}\\
         \arctan\text{ in }(\frac\pi2,\pi) &\text{ for }
                t\in (-\frac4{\on{Tr}(H)},\infty). 
\end{cases}
\end{align*}
Let $N^h:= \{x\in M: H_0(x)=0\}$, and if $N^h\neq\emptyset$ let
$t^h:=\inf\{\on{Tr}(H)(x): x\in N^h\}$. Then the geodesic $g(t)$ is
defined for $t\in [0,\infty)$ if $N^h=\emptyset$ or if $t^h\geq0$,
and it is only defined for $t\in [0,-\frac4{t^h})$ if $t^h<0$.
\end{theorem}
These formulas have been independently derived by Freed and Groisser \cite{Freed1989} and Michor and Gil-Medrano \cite{Gil-Medrano1991}.
A similar result is also available for the metric $G^{\Ph}$ with $\Ph(\on{Vol})=\frac{1}{\Vol}$; see \cite{Clarke2013}.
\begin{remark}
The geodesic equation for higher order metrics will generally not be  an ODE anymore and explicit solution formulas do not exist. 
Nevertheless, it has been shown that the geodesic equations are  (locally) well-posed, assuming certain conditions 
on the operator field $L$ defining the metric; see \cite{Bauer2013a}. These conditions are satisfied by the class of Sobolev type metrics and  conformal metrics but not by the scalar curvature weighted metrics.
\end{remark}

\subsection{Conserved quantities}
Noether's theorem associates to any metric on $\Met(M)$, that is invariant under pull-backs by the diffeomorphism group  $\on{Diff}(M)$, for each $X \in \mf X(M)$ the quantity
\[G_g(g_t,\zeta_X(g))=\on{const}\,,\]
which is conserved along each geodesic $g(t)$. Here  $\zeta_X$ is the fundamental vector field of the right action of $\Diff(M)$,
\[\zeta_X(g)=\L_X g=2\on{Sym}\nabla^g(g(X))\,,
\]
and
$\on{Sym}\nabla^g(g(X))$ is the symmetrization of the bilinear form $(Y,Z) \mapsto \nabla_Y^g g(X,Z)$, i.e., 
\[\on{Sym} \nabla^g(g(X))(Y,Z) = \tfrac 12 \left(\nabla^g_Y g(X,Z)+ \nabla^g_Z g(X,Y)\right). \]
If $G_g(g_t,\zeta_X(g))$ vanishes for all vector fields $X \in \mf X(M)$ along a geodesic $g(t)$, then $g(t)$ intersects each $\on{Diff}(M)$-orbit orthogonally.

\subsection{Completeness}\label{Met:completion}
The $L^2$-metric on $\Met(M)$ is incomplete, both metrically and geodesically. The metric completion 
of it has been studied by Clarke in \cite{Clarke2009,Clarke2013b,Clarke2013c}.
To describe the completion let $\Met_f$ denote the set of measurable sections of the bundle $S^2_{\geq 0} T^\ast M$ of symmetric, positive semi-definite $0\choose 2$-tensors, which have finite total volume. Define an equivalence relation on $\Met_f$ by identifying $g_0 \sim g_1$, if the following statement holds almost surely:
\[
g_0(x) \neq g_1(x) \Rightarrow \text{both } g_i(x) \text{ are not positive definite.}
\]
In other words, let $D=\{ x : g_0(x) \neq g_1(x)\}$ and $A_i = \{ x : g_i(x) \text{ not pos. def.} \}$. Then
\[
g_0 \sim g_1 \Leftrightarrow D \setminus (A_0 \cap A_1) \text{ has measure } 0\,.
\]
Note that the map $\Met(M) \hookrightarrow \Met_f/\!\sim$ is injective.

\begin{theorem}[Thm. 5.17, \cite{Clarke2009}]
The metric completion of the space $(\Met(M), \on{dist}^E)$ can be naturally identified with $\Met_f/\!\sim$.
\end{theorem}

In the subsequent article \cite{Clarke2013b} it is shown that the metric completion  is a non-positively curved space in the sense of Alexandrov.
\begin{theorem}[Thm~5.6, \cite{Clarke2013b}]
\label{lem_cat0}
The metric completion $\overline{\Met(M)}$ of $\Met(M)$ with respect to the $\dist^E$-me\-tric
is a $\on{CAT}(0)$ space, i.e., 
\begin{enumerate}
  \item there exists a length-minimizing path (geodesic) between any two points in $\overline{\Met(M)}$ and
 \item $\big(\overline{\Met}(M),\on{dist}^E \big)$ is a non-positively curved space in the sense of Alexandrov.
\end{enumerate}
\end{theorem}

\subsection{Curvature}

For the $L^2$-metric, there exists a comparably simple expression for the curvature tensor.

\begin{theorem}[Prop. 2.6, \cite{Gil-Medrano1991}]
The Riemannian curvature for the
$L^2$--metric on the manifold $\Met(M)$ of all
Riemannian metrics is given by
\begin{align*} g^{-1} &R_g(h,k)l = \tfrac14 [[H,K],L]\\&\qquad+ \frac{m}{16}(\on{Tr}(KL)H-\on{Tr}(HL)K)\\
&\qquad+ \frac1{16}(\on{Tr}(H)\on{Tr}(L)K-\on{Tr}(K)\on{Tr}(L)H)\\&\qquad
+ \frac1{16}(\on{Tr}(K)\on{Tr}(HL)-\on{Tr}(H)\on{Tr}(KL))\on{Id}\,,
\end{align*}
where $H= g^{-1} h$, $K= g^{-1} k$ and $L=g^{-1} l$. 
\end{theorem}

In the article \cite{Freed1989} the authors have determined the sign of the sectional curvature:
\begin{theorem}[Cor. 1.17, \cite{Freed1989}]
The sectional curvature for the
$L^2$--metric on the manifold $\Met(M)$ of all
Riemannian metrics is non-positive.
For the plane $P(h,k)$ spanned by orthonormal $h,k$ it is
\begin{align*}
k_g^{\Met}(P(h,k)) &= \int_M \frac{m}{16}\Big(\on{Tr}(HK)^2-\on{Tr}(H)^2\on{Tr}(K)^2\Big)\\&\qquad\qquad
+\frac14 \on{Tr}\Big(([H,K])^2\Big)\on{vol}(g)\,,
\end{align*}
where $H= g^{-1} h$ and $K= g^{-1} k$. 
\end{theorem}
In \cite{Clarke2013} it is proven that this negative curvature carries over to the metric-completion of $(\Met(M),G^E)$, as it is a $\on{CAT}(0)$ space; see Lem. \ref{lem_cat0}.

\begin{acknowledgements}
We would like to thank the referees for their careful reading of the article as well as the thoughtful comments, that helped us improve the exposition.
\end{acknowledgements}

\def\cprime{$'$}

\end{document}